\documentclass[11pt,a4paper]{article}

\usepackage{amsfonts}
\usepackage{amsthm}
\usepackage{amsmath}
\usepackage{amssymb}
\usepackage{amscd}
\usepackage{mathrsfs}
\usepackage[mathscr]{eucal}
\usepackage{graphicx}
\usepackage{epstopdf} 
\usepackage{pict2e}
\usepackage{epic}
\usepackage{xcolor}
\usepackage{float}
\usepackage{enumerate}
\usepackage{tikz}
\usepackage{tikz-cd}
\usepackage{cite}
\usepackage{hyperref}
\hypersetup{colorlinks=true, urlcolor= blue, linkcolor=blue, citecolor=red}

\usepackage[utf8]{inputenc}
\usepackage[T1]{fontenc}
\usepackage{lmodern}
\usepackage{dsfont}
\usepackage{indentfirst}
\usepackage[margin=2.5cm]{geometry}

\numberwithin{equation}{section}

\theoremstyle{plain}
\newtheorem{Th}{Theorem}[section]
\newtheorem{Lemma}[Th]{Lemma}
\newtheorem{Cor}[Th]{Corollary}
\newtheorem{Prop}[Th]{Proposition}
\newtheorem{Claim}[Th]{Claim}
\newtheorem{op}{Open Problem}

\theoremstyle{definition}
\newtheorem{Def}[Th]{Definition}

\newtheorem{Rem}[Th]{Remark}

\newtheorem*{Lemma*}{Lemma}

\newtheorem*{Acknowledgements}{Acknowledgements}

\newcommand{\bfa}{\mathbf{a}}
\newcommand{\bfb}{\mathbf{b}}
\newcommand{\bfx}{\mathbf{x}}
\newcommand{\bfy}{\mathbf{y}}
\newcommand{\bfA}{\mathbf{A}}
\newcommand{\bft}{\mathbf{t}}

\newcommand{\N}{\ensuremath{\mathbb{N}}}
\newcommand{\R}{\ensuremath{\mathbb{R}}}
\newcommand{\C}{\ensuremath{\mathbb{C}}}
\newcommand{\Z}{\ensuremath{\mathbb{Z}}}

\newcommand{\PP}{\ensuremath{\mathbb{P}}}
\newcommand{\Pl}{\ensuremath{\mathbb{P}_{\lambda}}}
\newcommand{\PlN}{\ensuremath{\mathbb{P}_{\lambda}^{N^2}}}
\newcommand{\e}{\ensuremath{\mathrm{e}}}
\newcommand{\eps}{\varepsilon}
\newcommand{\E}{\mathbb{E}}

\newcommand{\ElN}{\mathbb{E}_{\lambda}^{N^2}}
\newcommand{\dist}{\vert\vert}
\newcommand{\A}{\mathcal{A}}

\newcommand{\1}{\mathds{1}}
\newcommand{\G}{\Gamma}
\newcommand{\om}{\omega}
\newcommand{\g}{\gamma}

\newcommand{\Lab}{\Lambda^{a \rightarrow b}}
\newcommand{\La}{\Lambda}
\newcommand{\la}{\lambda}

\newcommand{\EA}{\mathsf{ExcessArea}}
\newcommand{\MiFL}{\mathsf{MeanFL}}
\newcommand{\MFL}{\ensuremath{\mathsf{MaxFL}}}
\newcommand{\MiLR}{\mathsf{MeanLR}}
\newcommand{\MLR}{\ensuremath{\mathsf{MaxLR}}}
\newcommand{\LR}{\ensuremath{\mathsf{LR}}}
\newcommand{\MLRF}{\ensuremath{\mathsf{MLRF}}}

\newcommand{\GH}{\mathsf{GoodHit}}
\newcommand{\GS}{\mathsf{Shape}}
\newcommand{\BF}{\mathsf{BigFacet}}

\newcommand{\GAC}[3]{\mathsf{GAC} \left( #1 , #2, #3 \right)}
 
\newcommand{\Enc}{\textup{Enclose}}
\newcommand{\Unfav}{\textup{UNFAV}}

\newcommand{\BiF}{\mathsf{BigMeanFL}}

\newcommand{\goes}[3]{\xrightarrow[#2]{#1} {#3}}

\newcommand{\MaxFac}{\mathsf{MaxFac}}
\newcommand{\bad}{\mathsf{Bad}_{\eps}^{+,-}}
\newcommand{\badplus}{\mathsf{Bad}_{\eps,\mathbf{B}_i(N)}^+}
\newcommand{\badplusun}{\mathsf{Bad}_{\eps,\mathbf{B}_1(N)}^+}
\newcommand{\badplusA}{\mathsf{Bad}_{\eps,\mathbf{A}}^+}

\newcommand{\badmoins}{\mathsf{Bad}_{\eps, \mathbf{B}_i(N)}^-}
\newcommand{\badmoinsun}{\mathsf{Bad}_{\eps, \mathbf{B}_1(N)}^-}
\newcommand{\badmoinsA}{\mathsf{Bad}_{\eps, \mathbf{A}}^-}

\usepackage{mathtools}

\title{Exact cube-root fluctuations in an area-constrained random walk model}
\begin{document}

\author{Lucas D'Alimonte\footnote{Université de Fribourg,
\url{lucas.dalimonte@unifr.ch}}, \, Romain Panis\footnote{Université de Genève, \url{romain.panis@unige.ch}}}
\maketitle

\begin{abstract}
    This article is devoted to the study of the behaviour of a (1+1)-dimensional model of random walk conditioned to enclose an area of order $N^2$. Such a conditioning enforces a globally concave trajectory. We study the local deviations of the walk from its convex hull. To this end, we introduce two quantities --- the mean facet length $\MiFL$ and the mean local roughness $\MiLR$ --- measuring the typical longitudinal and transversal fluctuations around the boundary of the convex hull of the random walk. Our main result is that $\MiFL$ is of order $N^{2/3}$ and $\MiLR$ is of order $N^{1/3}$. Moreover, following the strategy of Hammond (Ann. Prob., 2012), we identify the polylogarithmic corrections in the scaling of the \emph{maximal} facet length and of the \emph{maximal} local roughness, showing that the former one scales as $N^{2/3}(\log N)^{1/3}$, while the latter scales as $N^{1/3}(\log N)^{2/3}$. The object of study is intended to be a toy model for the interface of a two-dimensional statistical mechanics model (such as the Ising model) in the phase separation regime --- we discuss this issue at the end of this work. 
\end{abstract}

\section{Introduction}

The phase separation problem is concerned with the study of the boundary appearing between two different phases of a statistical mechanics model, in a regime where those two phases can coexist. In his seminal work, Wulff~\cite{Wulff1901} proposed that such a boundary should macroscopically adopt a deterministic limit shape given by the solution of a variational problem involving thermodynamic quantities such as the surface tension. This prediction has been an object of intense study and has by now been made rigorous in a very wide variety of settings, see for instance the monographs~\cite{DobrushinKoteckyShlosmanWulffconstruction,bodineau2000rigorous,cerfpisztora,Cerf2006}.

While this macroscopic shape is dependent on the model, the \emph{fluctuations} of the random phase boundary are widely believed to behave in an universal way. A first natural candidate to measure these fluctuations is the deviations from the limit shape. These have been shown to be Gaussian in~\cite{dobrushinhryniv} in the context of area-constrained random walks, and later on in~\cite{DobrishinHrynivIsing} for the 2D Ising phase boundary at low temperatures. Two other quantities of interest are given by the \textit{maximal facet length} and the \textit{maximal local roughness} of the interface, that are respectively the length of the largest segment of the convex hull of the interface and the maximal deviation of the interface from this convex hull. In a seminal paper~\cite{Alexander2001CubeRootBF}, Alexander conjectured that the exponent governing the scaling of the maximal local roughness should be $1/3$. In the context of percolation models in the phase separation setting, he derived upper bounds for an averaged version of the local roughness. Alexander and Usun~\cite{alexanderusunlowerboundslocalroughness} then complemented this work by providing lower bounds for the local roughness in the setup of Bernoulli percolation. Later on, in a remarkable series of papers, Hammond~\cite{alan3, alan1, alan2} was able to identify the exact scale of the maximal facet length and the maximal local roughness of a droplet of volume $N^2$ in the planar subcritical random-cluster model. Indeed, he proved that the former is of order $N^{2/3}(\log N)^{1/3}$, while the latter is of order $N^{1/3}(\log N)^{2/3}$, validating the exponent derived in~\cite{Alexander2001CubeRootBF}. These results are built on the identification of $N^{2/3}$ as the scale at which the curvature effect enforced by the conditioning has the same order of magnitude as the Gaussian fluctuations of the interface. Finally, let us mention that in an earlier paper, Hammond and Peres~\cite{HammondPeresFluctuationsofaBrownianloopcapturinglargearea} introduced a continuous and Brownian version of the phase separation problem, studying a two-dimensional Brownian loop conditioned to enclose a large area. They proved results in favour of the appearance of the cube-root fluctuations in this setting.

In this paper, we study a model of random walks, with geometrically randomised length, conditioned to enclose an area at least equal to $N^2$. This model was suggested by Hammond in~\cite[Section 1.0.4]{alan1} and later in~\cite{hammond2019minerva}. We prove that, as predicted by the author, his techniques can successfully be applied to this setting, enabling us to establish the above-mentioned polylogarithmic corrections. However, the main innovation of this paper is the identification of the scaling of the \textit{typical} facet length and local roughness (rather than their maximal values). Call $\MiFL$ (resp. $\MiLR$) the length (resp. the local roughness) of the facet intersecting a given line. We prove that $\MiFL\asymp N^{2/3}$ and $\MiLR\asymp N^{1/3}$, see Theorem \ref{Theorem meanfl and meanlr}. 

The scaling exponents $1/3$ and $2/3$ have been shown to arise in various related contexts in statistical mechanics. An important example is the \textit{critical prewetting} in the Ising model, which has been extensively studied in a beautiful series of papers~\cite{hrynivvelenikuniversalityofcriticalbehaviourinaclassofrecurrentrandomwalks, Velenikentropicrepulsionofaninterfaceinanexternalfield, ISV, gangulygheissariLocalandglobalgeometryofthe2disinginterfaceincriticalprewetting, ioffeottshlosmanvelenikCriticalprewettinginthe2disingmodel} to mention a few of them. A remarkable aspect of these works is that the results are derived without the help of any integrable feature. Let us also note that~\cite{caputolubetzkymartinellislytoninellicuberootsos,caddeo2023level} provided strong evidence for a similar behaviour in the context of the SOS model above a wall in $(2+1)$ dimensions.

To the best of our knowledge, it is the first time that those exponents are identified in a context such as ours. Indeed, in the above-mentioned works, the interface lies above a facet of length much longer than $N^{2/3}$ --- this facet being deterministic and artificially created by looking at the interface along a side of a large box for instance. However, in our work, the facets are not deterministic, and themselves reflect the competition between the randomness and the curvature induced by the conditioning. 

Finally, we strongly believe (supported by~\cite{alan1}) that our approach is robust and should allow to derive the same result for a large variety of models, including the fluctuations of the outermost circuit in a subcritical random-cluster model conditioned to enclose a large area, which itself is a good toy model for the boundary of a droplet in a supercritical Potts model. Note that this suggests that at scale $N^{2/3}$, our model and more general phase boundaries models should lie in the same universality class. This is discussed in Section~\ref{section extension to other models}. 

\subsection{Definition of the model and statement of the main results}
Let $\La$ be the set of finite paths in the first quadrant of $\mathbb Z^2$ which start on the $y$-axis and end on the $x$-axis, and which are oriented in the sense that they only take rightward and downward steps, see Figure~\ref{fig:example of path}. For $\gamma \in \La$, we define $|\g|$ to be its length, i.e. its number of steps (which is also the sum of the $y$-coordinate of its starting point and the $x$-coordinate of its ending point). It will be convenient to identify an element $\g\in \La$ with the set of points of $\mathbb N^2$ it passes by, i.e $\g=(\g(k))_{0\leq k \leq |\g|}$. We set $\La_n$ to be the subset of $\La$ of such oriented paths of length $n$. It is clear that $\vert \La_n \vert = 2^n$. If $a,b\in \mathbb N^2$, we denote by $\Lambda^{a \rightarrow b}$ the set of downright paths from $a$ to $b$.

Let $0< \la < \frac{1}{2}$. We define a probability measure $\Pl$ on $\La$ by requiring that, for $\g \in \La$,
\begin{equation}
    \Pl[\g] = \frac{1}{Z_{\la}} \la^{|\g|},
\end{equation}
where $Z_{\la}$ is the normalisation constant given by $Z_{\la}=(1-2\la)^{-1}$.

\begin{figure}[H]
    \centering
    \includegraphics{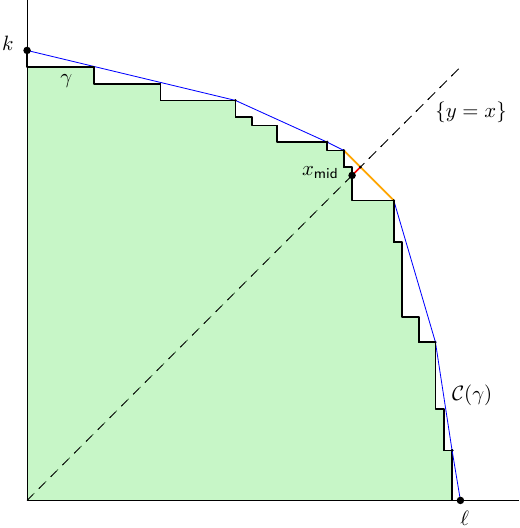}
    \caption{An example of path $\gamma$ (in bold black) with $|\gamma|=k+\ell$, and its least concave majorant $\mathcal{C}(\gamma)$ (in blue). The mean facet of $\gamma$ is coloured in orange, while the mean local roughness is the length of the red segment. The area $\mathcal{A}(\gamma)$ is the area of the light green shaded region.}
    \label{fig:example of path}
\end{figure}

The probability measure $\Pl$ is to be seen as a \emph{background} measure which imposes an exponential decay of the probability of sampling a long path. We now enforce a competing constraint of enclosing a large area. For $\g \in \La$, we define the area of $\g$, denoted by $\A(\g)$, to be the area enclosed by the graph of $\g$ and the two coordinate axes (see Figure~\ref{fig:example of path}). 
Let $N \in \N$ and let $\G$ be a sample of $\Pl$. We define a new probability measure $\PlN$ on $\La$ by 
\begin{equation}
    \PlN\left[\:\cdot\:\right] := \Pl \big[\:\cdot\:\vert \:  \A(\G) \geq N^2 \big].
\end{equation}
We also introduce $\La^{N^2}\subset \La$, the set of all paths of $\La$ that capture an area greater than $N^2$. Samples of $\PlN$ typically experience a competition between capturing a large area and minimizing the length of the paths. As explained before, it suggests that for large values of $N$, a typical sample of $\PlN$ will have a global curvature imposed by the area condition and Gaussian fluctuations imposed by the background measure (once the length is fixed, the background measure is the uniform law). The object of interest in our work, which tracks the competition between these two phenomena, is defined below.

\begin{Def}[Least concave majorant, facets]
Let $\g \in \La$. We set $\mathsf{Ext}(\gamma)$ to be the set of non-zero extremal points of the convex hull of $\gamma \cup \lbrace (0,0) \rbrace$. Let $\mathcal{C}(\gamma)$ be the \textit{least concave majorant} of $\gamma$, that is, the graph of the piecewise affine function passing through all the points of $\mathsf{Ext}(\gamma)$. Observe that with this definition, $\mathcal{C}(\gamma)$ becomes a finite union of line segments of $\R^2$ which are called \textit{facets} (see Figure~\ref{fig:example of path}), the endpoints of these facets precisely being the elements of $\mathsf{Ext}(\gamma)$. For $x \in \gamma$, we define its \textit{local roughness} to be its Euclidean distance to $\mathcal{C}(\gamma)$,
\begin{equation}
\LR(x) := \mathrm{d}(x, \mathcal{C}(\gamma)).
\end{equation}
\end{Def}

Our main result concerns the tightness of the length of a typical facet intercepting a given ray emanating from the origin (resp. the fluctuations of $\Gamma$ along the facet) at scale $N^{2/3}$ (resp. $N^{1/3}$). Let us first define the quantities of interest.

\begin{Def}[$\MiFL$, $\MiLR$]\label{def: meanfl meanlr}
    Let $\gamma \in \Lambda$. We define $\mathsf{MeanFac}(\gamma)$ to be the facet intersecting the line $\lbrace y=x \rbrace$ (if we are in the case where the ray $\lbrace y=x \rbrace$ intercepts the endpoint of two consecutive facets, we arbitrarily choose the leftmost one for $\mathsf{MeanFac}(\gamma)$). We call $x_\mathsf{mid}(\gamma)$ the point of $\N^2 \cap \gamma$ that is the closest to the line $\lbrace y=x\rbrace$ (as previously, in case of conflict, we arbitrarily choose the leftmost one). Assume that $\mathsf{MeanFac} = [a,b](=\lbrace at+(1-t)b,\: t\in[0,1]\rbrace)$, with $a,b \in \N^2$. Define the following random variables,
    \begin{equation}
        \MiFL(\gamma) := \mathrm{d}(a,b),\qquad \MiLR(\gamma) := \LR(x_{\mathsf{mid}}).
    \end{equation}
In words, $\MiFL$ is the length of $\mathsf{MeanFac}$ and $\MiLR$ is the local roughness of the ``mean vertex'' of $\gamma$.
\end{Def}
The following result is the main contribution of the paper.
\begin{Th}[Tightness of $\MiFL$ and $\MiLR$ at scales $N^{2/3}$ and $N^{1/3}$]\label{Theorem meanfl and meanlr}
    Let $0<\lambda< \frac{1}{2}$. For any $\eps > 0$, there exist $c, C > 0$ and $N_0 \in \N$ such that for any $N\geq N_0$, 
    \begin{equation}\label{equation theorem meanfl}
    \PlN\big[ cN^{\frac{2}{3}} <  \MiFL(\Gamma) < CN^{\frac{2}{3}}\big] > 1-\eps,
    \end{equation}
    and
    \begin{equation}\label{equation theorem meanlr}
    \PlN\big[cN^{\frac{1}{3}} < \MiLR(\Gamma) < CN^{\frac{1}{3}}\big] > 1-\eps.
    \end{equation}
\end{Th}
We chose for convenience to formulate Theorem \ref{Theorem meanfl and meanlr} in terms of the statistics of $\mathsf{MeanFac}$. However, the proof never uses the symmetry of the model around the line $\{y=x\}$ so the result is not specific to the choice of the line $\lbrace y = x \rbrace$: for any $\alpha>0$, the result holds\footnote{Note however that the constants $c, C$ might now depend on $\alpha$.} for $\mathsf{MeanFac}_\alpha$ which is defined to be the unique facet intercepting the line $\lbrace y = \alpha x\rbrace$.

\begin{Rem}
    In the proof of Theorem~\ref{Theorem meanfl and meanlr}, we actually derive a slightly stronger statement. Indeed, we obtain stretch-exponential upper tails for $\MiFL$ and $\MiLR$ (see Propositions~\ref{prop meanfl} and~\ref{prop meanlr})
\end{Rem}
Following~\cite{alan1} and~\cite{alan2}, our second result identifies the logarithmic corrections to the \textit{maximal} facet length and to the \textit{maximal} local roughness along $\mathcal{C}(\gamma)$. We call these quantities $\MFL(\gamma)$ (resp. $\MLR(\gamma)$). \
\begin{Th}\label{theorem maxfl and maxlr}
    Let $0 < \lambda < \frac{1}{2}$. There exist $c, C > 0$ such that,
    \begin{equation}\label{maxfl.thm}
    \PlN\big[ cN^{\frac{2}{3}}(\log N)^{\frac{1}{3}} < \MFL(\Gamma) < CN^{\frac{2}{3}}(\log N)^{\frac{1}{3}}\big] \goes{}{N \rightarrow \infty}{1},
    \end{equation}
and
\begin{equation}\label{maxlr.thm}
\PlN[ cN^{\frac{1}{3}}(\log N)^{\frac{2}{3}} < \MLR(\Gamma) < CN^{\frac{1}{3}}(\log N)^{\frac{2}{3}}]\goes{}{N \rightarrow \infty}{1}.
\end{equation}
\end{Th}
\subsection{Related works and known results}

As pointed out in the introduction, similar models have been studied quite extensively in the literature, especially in~\cite{dobrushinhryniv}, where they are introduced as toy models for the study of a low-temperature interface of a (1+1)-dimensional SOS model. In this work, the authors investigate the behaviour of the model at the macroscopic scale $N$ and at the mesoscopic scale $N^{1/2}$. It is possible to extend their result to our setup. To properly state it,  we introduce the following parametrisation of $\Gamma$: let $\G(t)$ be the linear interpolation between the points $\G(k)$ for $0\leq k \leq |\G|$. Using the discussion of~\cite[Section 1]{dobrushinhryniv}, together with the basic estimates given by Lemma~\ref{lem: exp tail length} and Proposition~\ref{prop: excessArea}, one can obtain the following result.

\begin{Th}\label{limit shape theorem}
    Let $0<\la<\frac{1}{2}$. There exists a deterministic, concave and continuous function $f_\lambda:[0,1]\rightarrow \mathbb R^+$ such that for any $\eps > 0$,
    \begin{equation}\label{equation limit shape}
        \PlN\Big[\sup_{t\in [0,1]}\big\vert N^{-1}\Gamma_N(|\Gamma_N|t)-f_\la(t)\big\vert>\varepsilon\Big]\goes{}{N\rightarrow \infty}{0}.
    \end{equation}
\end{Th}
\begin{Rem}
    Using Lemma~\ref{lem: exp tail length}, the time scaling in \eqref{equation limit shape} is linear in $N$.
\end{Rem}
Such a phenomenon is by now very well known under the name of a \textit{limit shape phenomenon}, and is known to arise in a large variety of situations (see for instance~\cite{Okounkovlimitshapes, kenyonrio}). In (1+1) dimensions such as in our setting, two different points of view can be adopted to prove a statement such as Theorem~\ref{limit shape theorem}: the first one is the classical theory of sample path large deviations culminating with the celebrated Mogulski'i Theorem (see~\cite[Section 5]{dembo_large_2010}). The second possible point of view has a more statistical mechanics flavour and is known as \textit{Wulff theory} (see the reference monograph~\cite{DobrushinKoteckyShlosmanWulffconstruction}). In the work~\cite{dobrushinhryniv}, it is shown that both approaches can be implemented and yield the same result. In both cases, the function $f_\lambda$ is identified as the minimiser of a deterministic variational problem. Let us conclude this discussion by noticing that actually much stronger statements than \eqref{equation limit shape} can be obtained, and in particular large deviations principles for the sample path $\Gamma$, though we will not focus on results of this type.

The main result of~\cite[Theorem 2.1]{dobrushinhryniv} focuses on the fluctuations of $\Gamma_N$ around $f_\lambda$. Again, minor modifications of their proof lead to the following result.

\begin{Th}\label{theoreme convergence fluctuations gaussiennes autour de la linit shape}
    Let $0<\la<\frac{1}{2}$. There exists a (centered) Gaussian process $\xi_\lambda$ on the space $\mathcal{C}([0,1])$ (equipped with the topology of the uniform convergence) such that under the measure $\PlN$,
    \begin{equation}
        \Big( \frac{1}{\sqrt{N}}(\Gamma_N(\vert \Gamma_N\vert t) - N f_\lambda(t) )\Big)_{t \in [0,1]} \goes{(d)}{N \rightarrow \infty}{\xi_\lambda},
    \end{equation}
    where the convergence holds in distribution. 
\end{Th}

These two results can be stated heuristically as follows: at the macroscopic scale $N$, the conditioning on the event $\lbrace \mathcal{A}(\Gamma) \geq N^2 \rbrace$ enforces a deterministic and global curvature, whereas at the mesoscopic scale $N^{1/2}$ the conditioning has no effect on $\PP_\lambda$ and the Gaussian nature of the measure $\PP_\lambda$ is unchanged. As explained above, Theorem~\ref{Theorem meanfl and meanlr} identifies $N^{2/3}$ as being  the scale at which those two competing effects are of the same order. 

\subsection{A resampling strategy}

The proofs below heavily relie on a particularly simple (yet crucial) feature of the model: the so-called \textit{Brownian Gibbs property} (see~\cite{airylineensemble}). It can be stated as follows: start from a sample $\Gamma$ of $\PlN$, choose two points $a,b \in \Gamma$ with any random procedure ``explorable from the exterior\footnote{This terminology is directly inspired by the notion of explorable set in percolation theory.} of the path'' (by this we mean that the event $\lbrace a=x, b=y \rbrace$ is measurable with respect to $\Gamma\setminus\Gamma_{x,y}$). Then, conditionally on $(a,b)$ and $\Gamma\setminus\Gamma_{a,b}$, the distribution of the random variable $\Gamma_{a,b}$  is the uniform distribution on $\Lambda^{a \rightarrow b}$ \textbf{conditionally on the fact that the resulting path of $\Lambda$ encloses an area greater than $N^2$}. This apparently naive observation allows one to implement a strategy of \textit{resampling}. Indeed, one can construct several Markovian dynamics on $\Lambda$ leaving the distribution $\PlN$ invariant in a quite general fashion: start from a sample $\Gamma$ of $\PlN$, choose two points $a,b \in \Gamma$ according to the above procedure and replace $\Gamma_{a,b}$ by a sample of the uniform distribution of $\Gamma_{a,b}$ subject to the above-mentioned conditioning. Then, it is clear, thanks to the preceding observation, that the distribution of the output is $\PlN$. In what follows, we shall call such a dynamic on $\Lambda$ a \textit{resampling procedure}.

This point of view will be used several times in the proofs to analyse marginals of $\PlN$ in well-chosen regions of the first quadrant. A simple illustration is given in the proof of Proposition~\ref{prop: rough upper bounds}.

\subsection{Open problems}
In light of the preceding discussions, two natural questions arise. We now describe them.

\paragraph{}As explained above, we expect the methods developed in this article to help the analysis of the droplet of a subcritical planar random-cluster model. We strongly believe that our strategy could complement the results obtained by Hammond in \cite{alan3,alan1,alan2}. We provide strong heuristics towards this result in Section \ref{section extension to other models}.
\begin{op}[Droplet boundary in the subcritical planar random-cluster model] Extend Theorem \textup{\ref{Theorem meanfl and meanlr}} to the study of a typical facet in the setup of the droplet in the subcritical planar random-cluster model.
\end{op}
Once tightness is obtained in Theorem \ref{Theorem meanfl and meanlr}, it is quite natural to try to identify the candidate for the scaling limit of the excursion below a typical facet. A similar question has been answered in \cite{ioffeottshlosmanvelenikCriticalprewettinginthe2disingmodel} where the authors obtained convergence (in the bulk phase) of the object of study to the so-called \emph{Ferrari--Spohn diffusion}. In our setup, the excursion below the mean facet should correspond to the excursion of a Ferrari-Spohn process. We intend to study this question in the future.
\begin{op}[Scaling limit of the excursion below a facet]\label{open problem scaling limit} Construct the Ferrari--Spohn excursion $\mathsf{(FSE)}$. Prove that the piece of path lying below the mean facet, after a proper rescaling, converges to the $\mathsf{FSE}$ under $\PlN$ in the limit $N\rightarrow \infty$.
\end{op}
Compared to~\cite{ioffeottshlosmanvelenikCriticalprewettinginthe2disingmodel}, the study of Open Problem~\ref{open problem scaling limit} requires two additional ingredients. Indeed, one needs a fine understanding of the way that the $O(1)N^{1/3}$ excursions of size $N^{2/3}$ interact together along the convex hull of the droplet. Moreover, the construction of $\mathsf{FSE}$ and the proof of the convergence of the mean excursion towards it are non-trivial as the pinning condition takes place \emph{at the scale of the correlation length} of the system. 

\subsection{Organisation of the paper}
The paper is organised as follows: in Section~\ref{section preliminary results}, we gather some preliminary results that are going to be our toolbox for the proofs of the main results. Section~\ref{section analysis of the mean facet length and the mean local roughness} is then devoted to the proof of Theorem~\ref{Theorem meanfl and meanlr}, while Section~\ref{section analysis of the maximal facet length and of the maximal local roughness} consists in an adaptation of the arguments of~\cite{alan1, alan2} to prove Theorem~\ref{theorem maxfl and maxlr}. Finally, Section~\ref{section extension to other models} is devoted to a discussion regarding the extension of the results to other statistical mechanics models in the Wulff setting. 

\paragraph{\textbf{Notations and conventions.}} We shall adopt Landau formalism for real valued-sequences. Namely, whenever $(a_n)$ and $(b_n)$ are two real-valued sequences, we will write $a_n = o(b_n)$ when $\vert a_n\vert /\vert b_n \vert \goes{}{n \rightarrow \infty}{0}$. We will also use the notation $a_n = O(b_n)$ when there exists some constant $C>0$ such that $\vert a_n \vert\leq C\vert b_n \vert$ for all $n$ large enough. If $a_n = O(b_n)$ and $b_n= O(a_n)$, we shall write that $a_n \asymp b_n$. Finally we shall write $a_n \sim b_n$ when $a_n/b_n \goes{}{n \rightarrow \infty}{1}$.

If $A$ is a set, we denote by $\mathcal{P}(A)$ the power set of $A$. For $x=(x_1,x_2)\in \mathbb R^2$, $\Vert x\Vert:=\sqrt{x_1^2+x_2^2}$ denotes the Euclidean norm of $x$. If $S\subset \mathbb R^2$ is a Borel set, we denote by $|S|$ its Lebesgue measure. Moreover, for $x \in \Z^2$ we will write $\arg(x) \in [0, 2\pi)$ to denote the complex argument of $x$ seen as an element of $\C$. For $t\in \R$, $\lfloor t \rfloor$ will denote the integer part of $t$ and $\lceil t\rceil:=\inf\lbrace k \in \mathbb Z, \: k\geq t\rbrace$. 


\begin{Acknowledgements}
We warmly thank Alan Hammond, who suggested the problem, for numerous stimulating discussions and precious writing advices. We also warmly thank Ivan Corwin for suggesting this collaboration and for stimulating discussions. We thank Raphaël Cerf, Trishen S. Gunaratnam, Ioan Manolescu and Yvan Velenik for inspiring discussions at various stages of the project. LD was supported by the Swiss National Science Foundation grant n°182237. RP was supported by the NSF through DMS-1811143, the Swiss National Science Foundation and
the NCCR SwissMAP. 
\end{Acknowledgements}

\section{Preliminary results}\label{section preliminary results}
In the rest of this work, we fix $0<\lambda < \frac{1}{2}$.

\subsection{Basic statistics of typical samples of \texorpdfstring{$\PlN$}{PlN}}\label{section basic statistics}
In this subsection, we study some basic properties of a typical sample of $\PlN$.   

\begin{Lemma}[Tail estimates for the length of a sample of $\PlN$]\label{lem: exp tail length} There exist $c,C>0$ such that for any $N\geq 1$, any $t\geq 0$,
\begin{equation}
    \PlN\left[ |\Gamma| \geq tN \right] \leq C\e^{-cN(t-2\sqrt{2})}.
\end{equation}
\end{Lemma}

\begin{proof} It is clear that 
\begin{equation}
\PlN\left[|\G| \geq tN \right] \leq\frac{\Pl\left[|\G|\geq tN\right]}{\Pl\left[\mathcal{A}(\G)\geq N^2\right]}.
\end{equation}
An easy computation yields
\begin{equation}
    \Pl\left[|\G|\geq tN\right] = \frac{1}{Z_\lambda}\sum_{k \geq \lfloor tN\rfloor}(2\lambda)^k = (2\lambda)^{\lfloor tN\rfloor}.
\end{equation}
It remains to lower bound $\Pl\left[\mathcal{A}(\G)\geq N^2\right]$. Let $\mathsf{y}(\G)$ (resp. $\mathsf{x}(\G)$) be the $y$-coordinate (resp. $x$-coordinate) of the first (resp. last) vertex of $\Gamma$. The measure $\Pl$ conditioned on $(\mathsf{y}(\G),\mathsf{x}(\G))$ is exactly the uniform measure over path starting at $(0,\mathsf{y}(\G))$ and ending at $(\mathsf{x}(\G),0)$.
We claim that
\begin{equation}
    \Pl\big[\mathcal{A}(\G)\geq N^2 \: |\: (\mathsf{y}(\G),\mathsf{x}(\G))=(\lceil\sqrt{2}N\rceil,\lceil\sqrt{2}N\rceil)\big] \geq \frac{1}{2}.
\end{equation}
Indeed, the square formed by the vertices $(0,\lceil\sqrt{2}N\rceil), (\lceil\sqrt{2}N\rceil,\lceil\sqrt{2}N\rceil), (\lceil\sqrt{2}N\rceil, 0)$ and $(0,0)$ has an area at least equal to $2N^2$. Hence, a symmetry argument shows that the proportion of oriented paths starting at $(0,\lceil\sqrt{2}N\rceil)$ and ending at $(\lceil\sqrt{2}N\rceil,0)$ that fulfill the requirement $\left\lbrace\mathcal{A}(\G) \geq N^2 \right\rbrace$ is at least $1/2$.  By a standard computation, we find $c_1=c_1(\lambda)>0$ such that, for all $N\geq 1$,
\begin{equation}
    \Pl\left[(\mathsf{y}(\G),\mathsf{x}(\G))=(\lceil\sqrt{2}N\rceil,\lceil\sqrt{2}N\rceil)\right] = \frac{1}{Z_\lambda}\lambda^{2\lceil\sqrt{2}N\rceil}\binom{2\lceil \sqrt{2}N\rceil}{\lceil\sqrt{2}N\rceil}\geq c_1(2\lambda)^{2\sqrt{2}N}N^{-1/2}.
\end{equation}
Putting all the pieces together
\begin{equation}
    \PlN[|\G|\geq tN] \leq c_1^{-1} \sqrt{N}(2\lambda)^{tN}(2\lambda)^{-2\sqrt{2}N}.
\end{equation}
The proof follows readily.
\end{proof}
This tail estimate allows us to argue that a typical sample of $\PlN$ stays confined between two balls of linear radii with very high probability. For $K\geq 0$, let $B_K$ be the Euclidean ball of radius $K$ centered at $0$.

\begin{Lemma}[Confinement lemma]\label{lem: confinement lemma} There exist $K_1,K_2,c,C>0$ such that for all $N\geq 1$, 
\begin{equation}
\PlN[\Gamma \subset B_{K_1 N}\setminus B_{K_2 N} ]\geq 1 - C\e^{-cN}.
\end{equation}

\end{Lemma}
\begin{proof}
Using Lemma~\ref{lem: exp tail length} and a simple geometric observation, we get that for $K_1\in \mathbb N_{>0}$, 
\begin{eqnarray*}
\PlN\left[\Gamma \cap(B_{K_1N})^c\neq \emptyset\right] & \leq & \PlN[|\G|\geq K_1N] \\ &\leq& C\e^{-c(K_1-2\sqrt{2})N},
\end{eqnarray*}
which gives that $\G\subset B_{K_1N}$ with high probability for $K_1>2\sqrt{2}$.

For the second part of the statement, notice that the area of a path that is contained in $B_{K_1N}$ and intersects $B_{KN}$ (for $K<K_1$) is necessarily smaller than $2K_1KN^2$. Choosing $K_2<1/(2K_1)$ so that the preceding quantity is smaller than $N^2$ yields that 
\begin{equation}
    \PlN\left[\G \cap B_{K_2N}\neq \emptyset, \Gamma \cap(B_{K_1N})^c= \emptyset\right]=0. 
\end{equation}
Hence,
\begin{eqnarray*}
    \PlN[\G \cap B_{K_2N}\neq \emptyset]&=&\PlN\left[\G \cap B_{K_2N}\neq \emptyset, \Gamma \cap(B_{K_1N})^c\neq \emptyset\right]+0\\&\leq& \PlN\left[\Gamma \cap(B_{K_1N})^c\neq \emptyset\right],
\end{eqnarray*}
and the result follows.
\end{proof}

In what follows, $K_1,K_2$ will always denote the constants given by Lemma~\ref{lem: confinement lemma}.

The following quantity will be of particular interest for the rest of this work. 
\begin{Def}[Excess area]
The \textit{excess area} of a path $\gamma \in \Lambda^{N^2}$ is the quantity $\EA(\gamma)$ defined by
\begin{equation}
\EA(\g):=\mathcal{A}(\g)-N^2.
\end{equation}
\end{Def}
Since \PlN exponentially penalizes long paths, we may expect the typical area of a path to be close to $N^2$. The following result quantifies this observation and will be very useful later.

\begin{Prop}[Tail behaviour of $\EA$]\label{prop: excessArea}
There exists $c> 0$ such that for all $0 \leq t \leq N$,
\begin{equation}\label{eq: excess area}
    \PlN[\EA(\G) \geq tN] \leq 2\e^{-ct}.
\end{equation}
\end{Prop}

The idea of the proof is simple. By Lemma~\ref{lem: exp tail length}, a sample $\G$ of $\PlN$ has a length of order $N$. Assume that this path satisfies $\A(\G) > N^2 + tN$. If we remove its first $t$ steps we obtain a path of area (roughly) at least $N^2$ and which is exponentially (in $t$) favoured by $\PlN$.

\begin{proof} If $A$ is any subset of $\Lambda^{N^2}$, we shall write $Z_{\lambda}^{N^2}\left[ A \right] = \sum_{\g \in A}\lambda^{|\g|}$ for the partition function of $A$. We will also write $Z_\la^{N^2}:=Z_\la^{N^2}[\La^{N^2}]$. By definition,
\begin{equation}
    \PlN[\EA(\G)\geq tN, \text{ }\G\subset B_{K_1N}\setminus B_{K_2N}]=\frac{1}{Z_\lambda^{N^2}}\sum_{K_2N\leq a,b\leq K_1N}\lambda^{a+b}\sum_{\substack{\g: (0,a)\rightarrow (b,0)\\ \mathcal{A}(\g)\geq N^2+tN}}1.
\end{equation}

If $\g: (0,a)\rightarrow (b,0)$ with $K_2N\leq a,b\leq K_1N$, call $\mathsf{c}=\mathsf{c}(\g)$ the $y$-coordinate of the point of $\g$ of $x$-coordinate $\lfloor t/K_1\rfloor$. Splitting the path $\g$ at the point of coordinates $(\lfloor t/K_1\rfloor,\mathsf{c})$ splits $\g$ into a pair of elements of $\Lambda$ (after translation) that we denote by $(\g_1,\g_2)$, where $\g_1$ is a path from $(0,a-\mathsf{c})$ to $(\lfloor t/K_1\rfloor, 0)$ and $\g_2$ is an element of $\Lambda^{N^2}$ from $(0,\mathsf{c})$ to $(b-\lfloor t/K_1\rfloor, 0)$. As a result,
\begin{equation}
    \sum_{\substack{\g: (0,a)\rightarrow (b,0)\\ \mathcal{A}(\g)\geq N^2+tN}}1\leq \sum_{c=0}^{a} \big|\Lambda^{(0,a-c)\rightarrow (\lfloor t/K_1\rfloor},0)\big|\cdot\big|\Lambda^{(0,c)\rightarrow (b-\lfloor t/K_1\rfloor),0)}_{N^2}\big|,
\end{equation}
where $\Lambda^{x\rightarrow y}$ is the set of path of starting at $x$ and ending at $y$, and the subscript $N^2$ accounts for the condition of enclosing an area of at least $N^2$.
Now, notice that
\begin{equation}
    \sum_{b=K_2N}^{K_1N}\lambda^{c+b-\lfloor t/K_1\rfloor}\big|\Lambda^{(0,c)\rightarrow (b-\lfloor t/K_1\rfloor, 0)}_{N^2}\big|\leq Z_\lambda^{N^2}[\mathsf{y}(\G)=c],
\end{equation}
where we recall that $\mathsf{y}(\g)$ is the $y$-coordinate of the starting point of $\g$. Recall also that $\mathsf{x}(\g)$ is the $x$-coordinate of the last point of $\g$. Hence,
\begin{align*}
    \sum_{K_2N\leq a,b\leq K_1N}\lambda^{a+b}&\sum_{\substack{\g: (0,a)\rightarrow (b,0)\\ \mathcal{A}(\g)\geq N^2+tN}}1 \\
    &\leq\sum_{a=K_2N}^{K_1N}\sum_{c=0}^a \big|\Lambda^{(0,a-c)\rightarrow (\lfloor t/K_1\rfloor,0)}\big|\lambda^{a-c+\lfloor t/K_1\rfloor} Z_\lambda^{N^2}[\mathsf{y}(\G)=c]\\ 
    &\leq \sum_{c=0}^{K_1N}Z_\lambda^{N^2}[\mathsf{y}(\G)=c] \sum_{a\geq c}|\Lambda^{(0,a-c)\rightarrow (\lfloor t/K_1\rfloor, 0)}\big|\lambda^{a-c+\lfloor t/K_1\rfloor} \\ 
    &\leq \sum_{c=0}^{K_1N} Z_\lambda^{N^2}[\mathsf{y}(\G)=c]Z_\lambda[\mathsf{x}(\G)=\lfloor t/K_1\rfloor]\\ 
    &\leq  Z_\lambda^{N^2}Z_\lambda[|\G|\geq\lfloor t/K_1\rfloor]\\ 
    &\leq Z_\lambda^{N^2}(2\lambda)^{\lfloor t/K_1\rfloor}.
\end{align*}
The proof follows readily.
\end{proof}

\subsection{Non-existence of large flat sections of a typical sample of \texorpdfstring{$\PlN$}{}}\label{subsection Boundedness of the derivative of the limit shape function }

In the analysis of samples of $\PlN$ in given angular sectors, we will sometimes need to ensure that the marginal is not ``degenerate'' in the sense that it is not supported on ``flat'' paths (i.e  almost horizontal / vertical paths). Because of the oriented feature of the model, this is an additional difficulty in comparison to the setup of subcritical statistical mechanics models, where it is often known thanks to the Ornstein--Zernike theory that the surface tension is analytic and bounded away from 0 and infinity (see for instance~\cite[Theorem A]{campaninonioffevelenikozrandomcluster}).

Let $\mathbf{A}$ be a cone of apex the origin in the first quadrant of angle $\Theta_\mathbf{A}\in (0,\pi/2]$. For $\gamma \in \La^{N^2}$, let $x_\mathbf{A}=x_\mathbf{A}(\g)$ (resp $y_\mathbf{A}=y_\mathbf{A}(\g)$) be the left-most (resp. right-most) point of $\gamma\cap \mathbf{A}$. Also, recall that the set $\Lambda^{x \rightarrow y}\subset \La$ consists of all the oriented path going from $x$ to $y$. Observe that if $\gamma$ is a path of $\Lambda$ containing $x$ and $y$, then there is a natural notion of restriction of $\gamma$ between $x$ and $y$: it is the only element of $\Lambda^{x \rightarrow y}$ which coincides with $\gamma$ between $x$ and $y$. We also denote by $\theta(x_\mathbf{A}, y_\mathbf{A})\in[0,\pi/2]$ the angle formed by the segment $[x_\mathbf{A}, y_\mathbf{A}]$ and the horizontal line going trough $x_\mathbf{A}$. For $\eps > 0$, we define the following events
\begin{equation}
    \badplusA := \big\{ \gamma \in \La^{N^2}\: ,\: \theta(x_\mathbf{A}, y_\mathbf{A}) \in [0, \eps]  \big\},
\end{equation}
and,
\begin{equation}
    \badmoinsA := \big\{ \gamma \in \La^{N^2}\: , \:\theta(x_\mathbf{A}, y_\mathbf{A}) \in [\tfrac{\pi}{2}-\eps, \tfrac{\pi}{2}] \big\}.
\end{equation}

\begin{Prop}\label{prop: bad angle}
There exist $\eps > 0$ and $c=c(\varepsilon),C=C(\varepsilon) > 0$ such that, for any cone $\mathbf{A}$ as above,
\begin{equation}
    \PlN \big[ \Gamma \in \badplusA \cup \badmoinsA \big] \leq C\ElN\big[\e^{-c\Vert x_\mathbf{A}(\G)-y_\mathbf{A}(\G)\Vert}\big].
\end{equation}
\end{Prop}

For the proof of Proposition~\ref{prop: bad angle}, we will use a probabilistic version of the multi-valued map principle, stated and proved in the appendix (Lemma~\ref{lemme MVMP}). 
\begin{proof}[Proof of Proposition \textup{\ref{prop: bad angle}}]

We are going to create an appropriate multi-valued map $T$, transforming a path belonging to $\badplusA$ to another path of $\Lambda^{N^2} \setminus \badplusA$, in such a way that this map has a lot of possible images but very few pre-images, and such that the probability of one element of the image of a path $\gamma \in \badplusA$ by $T$ is not too small compared to $\PlN[\g]$.
\begin{figure}[H]
    \centering
    \includegraphics{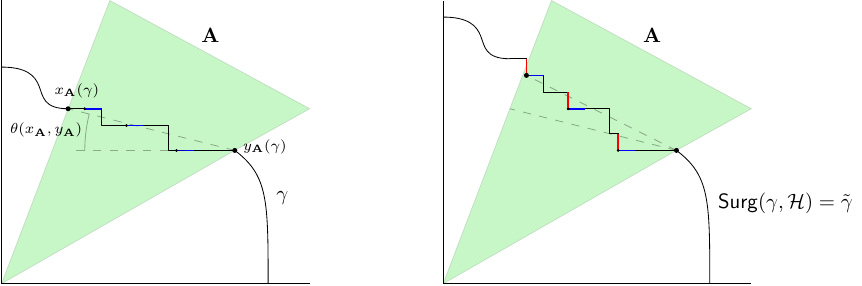}
    \caption{An illustration of the effect of the surgery map $\mathsf{Surg}(\gamma, \mathcal{H})$. The cone $\mathbf{A}$ is the light green region. The $\eta k$ edges of $\mathcal{H}$ that have been selected are the blue ones, and the vertical edges that have been added are depicted in red. We observe that the angle $\theta(x_\mathbf{A}, y_\mathbf{A})$ increases in $\mathsf{Surg}(\gamma, \mathcal{H})$.}
    \label{fig:illustration surgery}
\end{figure}

Let $\eta > 0$ be a small parameter to be fixed later. For $\gamma \in \Lambda^{N^2}$, we denote by $\ell_{\mathbf{A}}=\ell_\mathbf{A}(\g)$ the length of the section of the path enclosed by $\mathbf{A}$, that is, $\ell_{\mathbf{A}}(\g)=|\g_{x_\mathbf{A},y_\mathbf{A}}|$. We first define a surgical procedure modifying $\g$: if $\mathcal{H}$ is some subset of the set of the horizontal steps of $\gamma$ inside $\mathbf{A}$, we set $\mathsf{Surg}\left(\gamma, \mathcal{H}\right)$ to be the new path obtained by adding a vertical step just before each step belonging to $\mathcal{H}$ (see Figure~\ref{fig:illustration surgery}). 
We are going to apply this procedure for $\mathcal{H}$ being a proportion $\eta$ of the horizontal steps that $\g$ makes in $\mathbf{A}$. Of course to be able to do so we need to ensure that the paths $\g$ we apply the map to already contain a proportion at least $\eta$ of horizontal steps in $\mathbf{A}$. Since we will apply the map to elements of $\badplusA$, we expect the proportion of horizontal steps to be quite large. Indeed, if $\g\in  \badplusA$, and if $\alpha(\gamma)$ is the proportion of vertical steps in the piece of path of $\gamma$ in $\mathbf{A}$, then one has that
\begin{equation}\label{eq: proportion of vertical steps is small}
    \varepsilon\geq \arctan\left(\frac{\alpha(\gamma)}{1-\alpha(\gamma)}\right),
\end{equation}
so that $\alpha(\gamma)\leq \tan(\varepsilon)\leq 2\varepsilon$ (if $\varepsilon$ is small enough). For technical reasons, it will also be convenient to work on a set of paths for which $\ell_\mathbf{A}$ is fixed. Introduce, for $k\geq 1$,
\begin{equation}
    A_k:=\badplusA\cap \lbrace \ell_\mathbf{A}=k\rbrace \cap \Lambda^{N^2},
\end{equation}
and 
\begin{equation}
    B_k:=\lbrace k\leq \ell_\mathbf{A}\leq (1+\eta)k\rbrace\cap \Lambda^{N^2}.
\end{equation}
Define the multi-valued map $T_{\eta,k}: A_k\rightarrow \mathcal{P}(B_k)$ by setting for any $\gamma \in A_k$,
\begin{equation}
    T_{\eta,k}(\gamma) = \lbrace \mathsf{Surg}(\gamma, \mathcal{H}), \left| \mathcal{H} \right| = \eta k \rbrace.
\end{equation}
We notice that for $\gamma\in A_k$, $T_{\eta,k}(\gamma) \subset \Lambda^{N^2}$. Indeed, the area can only increase under the action of $\mathsf{Surg}(\gamma, \mathcal{H})$, for any $\mathcal{H}$ being a subset of the horizontal steps of $\gamma$. Moreover, after the surgery, the length of a path inside $\mathbf{A}$ has increased by at most $\eta k$. Thus, the map $T_{\eta, k}$ is well-defined. We will consider the probability measures $\PlN$ on $A_k$ and $B_k$. This way, we are exactly in the setting of Lemma~\ref{lemme MVMP}. Introducing
\begin{equation}
    \varphi(T_{\eta,k}):=\max_{a \in A_k}\max_{b \in T_{\eta,k}(a)} \frac{\PlN[a]}{ \PlN[b]}, \qquad \psi(T_{\eta,k}):= \frac{\max_{b \in \tilde{B}_k}\big| T_{\eta,k}^{-1}(b)\big|}{\min_{a \in A_k}\big| T_{\eta,k}(a) \big|},
\end{equation}
where $\tilde{B}_k:=\textup{Im}(T_{\eta,k})$.
A direct application of Lemma~\ref{lemme MVMP} yields that
\begin{equation}\label{eq: premiere application mvmp}
    \PlN[A_k]\leq \varphi(T_{\eta,k})\psi(T_{\eta,k})\PlN[B_k].
\end{equation}
We are now left with estimating $\varphi(T_{\eta,k})$ and $\psi(T_{\eta,k})$.

We start by observing that $\mathsf{Surg}(\gamma, \mathcal{H})$ increases the length of a path $\gamma$ by $\left| \mathcal{H} \right|$ units, so that for any $\gamma \in A_k$ and $\tilde{\gamma} \in T_{\eta,k}(\gamma)$, the ratio defining $\varphi$ is constant, and given by
\begin{equation}
    \varphi(T_{\eta,k})  = \lambda^{-\eta k}.
\end{equation}
Now, our task is to bound $\psi(T_{\eta,k})$. Let $\gamma \in A_k$. Then, it is clear that
\begin{equation}
    \big|T_{\eta,k}(\gamma)\big| \geq \binom{(1-\alpha(\gamma))k}{\eta k} \geq \binom{(1-2\eps)k}{\eta k},
\end{equation}
where $\alpha(\gamma)$ is the proportion of vertical steps in the piece of path of $\gamma$ in $\mathbf{A}$ and where we used \eqref{eq: proportion of vertical steps is small} to get that $\alpha(\gamma)\leq 2\varepsilon$ (if $\varepsilon$ is small enough).
Next, we have to bound $\left| T_{\eta,k}^{-1}(\widetilde{\gamma})\right|$, for $\widetilde{\gamma}\in \tilde{B}_k$. Some vertical steps may have been pushed outside the sector $\mathbf{A}$ by the effect of $T_{\eta,k}$, and we must explore them to reconstruct $\gamma$ knowing $T_{\eta,k}(\gamma)$. Since at most $\eta k$ of these steps disappeared in $\mathbf{A}$, we obtain
\begin{equation}
    \big|T^{-1}_{\eta,k} (\widetilde{\gamma})\big| \leq \binom{(2\eps+2\eta)k}{\eta k}.
\end{equation}
Putting all the pieces together and using \eqref{eq: premiere application mvmp}, we obtain
\begin{equation}\label{Equation controle proba bad}
    \PlN \left[ A_k\right] \leq \lambda^{-\eta k}\frac{\binom{(2\eps+2\eta)k}{\eta k}}{\binom{(1-2\eps)k}{\eta k}}\PlN[B_k]. 
\end{equation}
Standard estimates\footnote{For any $c_1 < c_2$, as $k \rightarrow \infty$,
\begin{equation}
    \binom{c_2 k}{c_1 k} \sim \e^{c_2 k I\left(\frac{c_1}{c_2}\right)}.
\end{equation}
} yield that, as $k\rightarrow \infty$,
\begin{equation}
    \lambda^{-\eta k}\frac{\binom{(2\eps+2\eta)k}{\eta k}}{\binom{(1-2\eps)k}{\eta k}}\sim \exp ( -k[ (1-2\eps)I\left(\frac{\eta}{1-2\eps}\right) - (2\eps + 2\eta)I\left( \frac{\eta}{2\eps + 2\eta} \right) + \eta \log \lambda ] ),
\end{equation}
where for $x\in (0,1)$, $I(x) := - x\log x - (1-x)\log(1-x)$.

Set 
\begin{equation}
    c_{\eps, \eta} = (1-2\eps)I\left(\frac{\eta}{1-2\eps}\right) - (2\eps + 2\eta)I\left( \frac{\eta}{2\eps + 2\eta} \right) + \eta \log \lambda.
\end{equation}
It is easy to check that one can find a pair $\eps, \eta(\eps) > 0$ satisfying $c_{\eps, \eta(\eps)} > 0 $. Let us fix such an $\eps>0$ and drop $\eta$ from the notations. We obtained that there exists $C=C(\varepsilon)>0$ such that for $k\geq 1$,
\begin{equation}
    \lambda^{-\eta k}\frac{\binom{(2\eps+2\eta)k}{\eta k}}{\binom{(1-2\eps)k}{\eta k}}\leq C\e^{-c_\eps k}.
\end{equation}
Hence,
\begin{equation}
    \PlN[\Gamma \in \badplusA, \: \ell_A=k]\leq C\e^{-c_\varepsilon k}\PlN[k\leq \ell_{\mathbf{A}}\leq (1+\eta)k].
\end{equation}
Summing over $k$ we finally obtain the existence of $c_1=c_1(\varepsilon),C_1=C_1(\varepsilon)$ such that
\begin{equation}
    \PlN\big[\G\in \badplusA\big]\leq C_1\mathbb E^{N^2}_\la\big[e^{-c_1 \ell_{\mathbf{A}}}\big]\leq C_1 \E^{N^2}_\lambda\big[\e^{-c_1 \Vert x_\mathbf{A}(\G)-y_\mathbf{A}(\G)\Vert}\big].
\end{equation}
A similar argument holds for $\badmoinsA$. This concludes the proof.
\end{proof}

\subsection{Couplings between \texorpdfstring{$\PlN$}{PlN} and a conditioned random walk}\label{subsection coupling}

This subsection is devoted to the description of two couplings between the measure $\PlN$ conditioned on a pinning event and the law of a random walk bridge. We first describe and prove the existence of these couplings and then explain what they will be useful for. 

Let $a,b \in \N^2$ such that $\arg(a) > \arg(b)$. We introduce the event $\mathsf{Fac}(a,b)$ defined by
\begin{equation}
    \mathsf{Fac}(a,b) := \left\lbrace \g\in \Lambda, \: \text{The segment } [a,b] \text{ is a facet of } \gamma \right\rbrace.
\end{equation}
Recall that the set $\Lambda^{a \rightarrow b}\subset \La$ consists of all the oriented path going from $a$ to $b$. We introduce the following partial order in the set $\Lambda^{a \rightarrow b}$: for $\gamma^1, \gamma^2 \in \Lambda^{a \rightarrow b}$, we shall say that $\gamma^1 \geq \gamma^2$ when
\begin{equation}
    \forall (x_1, y_1) \in \gamma^1, (x_2, y_2) \in \gamma^2, x_1 = x_2 \Rightarrow y_1 \geq y_2.
\end{equation}
Finally, we denote by $\PP_{a,b}^{-}$ the uniform measure over all the paths of $\Lambda^{a \rightarrow b}$ that remain under the segment $[a,b]$.

With these definitions in hand, we are now able to describe the first coupling.

\begin{Prop}\label{prop: couplage 1}
Let $a$ and $b$ be two vertices of $\N^2$ such that $\arg(a) > \arg(b)$. There exists a probability measure $\Psi$ on the space $\Lambda^{a \rightarrow b} \times \Lambda^{a \rightarrow b}$ such that:
\begin{enumerate}[(i)]
    \item the marginal of $\Psi$ on its first coordinate has the law of $\Gamma_{a,b}$ where $\G$ is sampled according to $\PlN[\:\cdot\: |\: \mathsf{Fac}(a,b)]$,
    \item the marginal of $\Psi$ on its second coordinate has law $\mathbb P^-_{a,b}$,
    \item $\Psi \big[ \{ (\gamma^1, \gamma^2) \in \Lambda^{a \rightarrow b} \times \Lambda^{a \rightarrow b}, \gamma^1\geq \gamma^2\} \big] = 1.$
\end{enumerate}
\end{Prop}

\begin{proof} The proof relies on a dynamical Markov chain argument and is, for percolation \textit{aficionados}, very similar to that of Holley's inequality (see for instance~\cite{grimmett10}). 

We shall start from the restriction of a sample of $\PlN[\:\cdot\:\vert\:\mathsf{Fac}(a,b)]$ to $\Lab$ that we denote $\Gamma_{a,b}$, and then describe two Markovian dynamics $(\G^1_k)_{k \geq 0}$ and $(\G^2_k)_{k \geq 0}$ on the state space $\Lab$ with the following properties:

\begin{enumerate}
\item[$(\mathbf{P1})$] $\G_0^1 = \G_0^2 = \G_{a,b}$,
\item[$(\mathbf{P2})$] $\forall k \geq 0$,  $\G_k^1 \geq \G_k^2$,
\item[$(\mathbf{P3})$] the law of $\G_{a,b}$ is the stationary measure of the Markov chain $(\G_k^1)_{k \geq 0}$,
\item[$(\mathbf{P4})$] $\PP^{-}_{a,b}$ is the stationary measure of the Markov chain $(\G_k^2)_{k \geq 0}$.
\end{enumerate}
Now, let us describe the Markovian dynamics. Let $p=\Vert a-b\Vert_1$ be the length of any path of $\Lab$. Let $(\ell_k)_{k \geq 0}$ be a family of independent random variables all having a uniform distribution on $\lbrace 1, \dots,p-1 \rbrace$, and $(X_k)_{k \geq 0}$ a family of independent random variables all having a uniform distribution on $\lbrace 0,1 \rbrace$. At step $k$, the random variable $\ell_k$ is the position of the vertex on which we will proceed a random modification.      

Looking at the two edges incident to the vertex $\G^{i}_k(\ell_k)$ for $i\in\lbrace 1,2\rbrace$, three different configurations can arise:
\begin{enumerate}
    \item[$1)$] If these two edges are co-linear, we do not modify anything: $\G^{i}_{k+1} = \G^{i}_k$ for $i\in \lbrace 1,2\rbrace$.
    \item[$2)$] If these edges form a corner of the type \begin{tikzpicture} [scale = 0.25] \draw (0,0) -- (1,0) ; \draw (0,0) -- (0,1) ; \end{tikzpicture}
    \begin{enumerate}
        \item[$(i)$] For $\G^1$ we flip the corner into the opposite corner \begin{tikzpicture}[scale = 0.25] \draw (0,0) -- (0,1) ; \draw (0,1) -- (-1,1) ; \end{tikzpicture} if and only if we have $X_k = 1$ \textbf{and} the modified path does not go above the facet $[a,b]$. If any of these two conditions is not verified, we do not modify $\G_k^1$.
        \item[$(ii)$] We apply the same procedure for $\G^2$.
    \end{enumerate}
    \item[$3)$] If these edges form a corner of the type \begin{tikzpicture}[scale = 0.25] \draw (0,0) -- (0,1) ; \draw (0,1) -- (-1,1) ; \end{tikzpicture}
    \begin{enumerate}
        \item[$(i)$] For $\G^1$, we flip this corner into the opposite corner \begin{tikzpicture} [scale = 0.25] \draw (0,0) -- (1,0) ; \draw (0,0) -- (0,1) ; \end{tikzpicture} if and only if $X_k = 1$ \textbf{and} the area enclosed by this new path is still greater than $N^2$. If any of these two conditions is not verified we do not modify $\G^1_k$.
        \item[$(ii)$] For $\G^2$, we flip the corner into the opposite one if and only if $X_k = 1$.
    \end{enumerate}
\end{enumerate}
It is straightforward that these Markov dynamics are irreducible, aperiodic, and reversible. Then, $\Gamma^1$ (resp. $\Gamma^2$) admits a stationnary measure and converges towards it: it is straightforward to see that this stationary measure is the law of $\G_{a,b}$ (resp $\PP^-_{a,b}$). 
Moreover, by construction, one always has $\Gamma^1 \geq \Gamma^2$. This yields the desired stochastic domination.
\end{proof}

This coupling will be very efficient when computing local statistics of a sample of $\PlN$ along a facet (e.g. the local roughness along a facet) as it allows us to compare them to the ones of a random walk excursion on which we know much more information.
More precisely, Proposition~\ref{prop: couplage 1} has the following consequence: an upper bound on the length of a facet automatically converts into an upper bound  on the local roughness under this facet. 

\begin{Lemma}\label{Lemme domination roughness random walk sous une facette}
  There exist $c,C,N_0 > 0$ such that the following holds. For any $a, b \in \N^2$ with $\arg(a) > \arg(b)$ and $\Vert b-a \Vert \geq N_0$ and any $t\geq 0$, 
    \begin{equation}
        \PlN\big[ \max_{x \in \Gamma_{a,b}} \LR(x) > t\Vert b-a \Vert^{1/2} \:\big\vert \: \mathsf{Fac}(a,b) \big] \leq C\exp(-ct^2).
    \end{equation}
\end{Lemma}
\begin{proof}
Observe that thanks to Proposition~\ref{prop: couplage 1}, the random variable $\max_{x \in \Gamma_{a,b}} \LR(x)$ where $\G$ is sampled according to $\PlN\left[\:\cdot\: \vert\:\mathsf{Fac}(a,b)\right]$ is stochastically dominated by the random variable $\max_{x \in \omega_{a,b}} \LR(x) $ where $\omega_{a,b}$ is sampled according to $\PP^-_{a,b}$.

The statement then follows by classical random walks arguments. Indeed, the results of~\cite{caravennachaumont} imply that a lattice random walk bridge conditioned to stay below a line segment with increments having exponential tails, after subtraction of the linear term corresponding to the equation of that segment, converges towards the standard Brownian excursion. As a consequence, knowing the fluctuations of the Brownian excursion (see~\cite{chung}) on an interval of size $\Vert b-a \Vert$ leads to the bound
 \begin{equation}
     \PP^-_{a,b} \big[ \max_{x \in \omega_{a,b}}\LR(x) \geq t\Vert b-a \Vert^{1/2} \big] \leq C\exp({-ct^2}),
 \end{equation}
 when $\Vert b-a \Vert$ is large enough.
\end{proof}

\begin{Rem}\label{remarque domination roughness avec le log}
In Section~\ref{section analysis of the maximal facet length and of the maximal local roughness}, we will use the preceding result in the regime $t =\left(\log\Vert b-a \Vert\right)^\alpha$ for some power $0<\alpha<1$. We claim that the result still holds true and is a consequence of classical moderate deviations estimates for random walks (see~\cite{deacostamoderatedeviations}). 
\end{Rem}

The second coupling we shall need is closely related to the first one with the only exception that we do not assume that $[a,b]$ is a facet anymore. This time, we will build an ``increasing coupling'' (in the sense of the partial order of $\Lambda^{a\rightarrow b}$) between the law of $\Gamma_{a,b}$ where $\Gamma$ is sampled according to $\PlN[\:\cdot\:|\: a,b\in \G]$ and the distribution of the random walk bridge between $a$ and $b$. We denote by $\PP_{a,b}$ the uniform distribution on $\Lambda^{a\rightarrow b}$.

\begin{Prop}\label{prop: couplage 2}
Let $a$ and $b$ be two vertices of $\N^2$ such that $\arg(a) > \arg(b)$. There exists a probability measure $\overline{\Psi}$ on the space $\Lambda^{a \rightarrow b} \times \Lambda^{a\rightarrow b}$ such that:
\begin{enumerate}[(i)]
    \item the marginal of $\overline{\Psi}$ on its first coordinate has the law of $\Gamma_{a,b}$ where $\G$ is sampled according to $\PlN[\:\cdot\: |\: a,b\in \G]$,
    \item the marginal of $\overline{\Psi}$ on its second coordinate has law $\mathbb P_{a,b}$,
    \item $\overline{\Psi} \big[ \{ (\gamma^1, \gamma^2) \in \Lambda^{a \rightarrow b} \times \Lambda^{a \rightarrow b}, \gamma^1\geq \gamma^2\} \big] = 1.$
\end{enumerate}
\end{Prop}
\begin{proof}
The proof follows the strategy used to obtain Proposition~\ref{prop: couplage 1} except that we do not have the condition of not going above the segment $[a,b]$ anymore.
\end{proof}

\subsection{A rough upper bound on the length of a typical facet}\label{section: exemple resampling}
The goal of this section is to give an illustration of the resampling strategy which will be used throughout the paper. The result proved in this section will also be useful for the proof of the upper bounds of Theorem~\ref{Theorem meanfl and meanlr}.
\begin{Prop}\label{prop: rough upper bounds}
Let $\epsilon>0$. There exist $c=c(\epsilon),C=C(\epsilon)>0$ such  that, for all $1\leq t \leq N^{2/3-\epsilon}$.
\begin{equation}\label{eq: prop rough upper 1}
    \PlN\big[\MiFL(\G)\geq t N^{\frac{2}{3}+\epsilon}\big]\leq C\e^{-ct^{3/2}N^{3\epsilon/2}}.
\end{equation}
Moreover, there exists $c',C'>0$ such that, for all $t\geq 1$
\begin{equation}\label{eq: prop rough upper 2}
    \PlN\big[\MiFL(\G)\geq tN^{\frac{4}{3}}\big]\leq C'\e^{-c'tN^{4/3}}.
\end{equation}
\end{Prop}
Before proving this result, we need to define an event which we will ask the resamplings to satisfy in order to ensure they capture a large enough area. 

Let us first recall and introduce some notations. For $x,y \in \mathbb N^2$, let $\mathbf{A}_{x,y}$ be the cone of apex $0$ bounded by $x$ and $y$ and let $\mathbf{T}_{0,x,y}$ be the triangle of apexes $0,x$ and $y$. For a cone $\mathbf{A}_{x,y}$ and $\g \in \Lambda$, we define $\Enc(\g\cap \mathbf{A}_{x,y})$ to be the region delimited by $\g$ and the two boundary axes of $\mathbf{A}_{x,y}$. The following event, illustrated in Figure~\ref{fig: the event GAC}, will be the one we will require our resampling to satisfy.

\begin{Def}[Good area capture]\label{definition good area capture}
Let $\g \in \La$ be a path, $\eta> 0$ and $x,y \in \N^2$. We say that $\g$ realises the event $\GAC{x}{y}{\eta}$ (meaning ``good area capture'') if 
\begin{enumerate}
    \item[$(i)$] $\g\cap \mathbf{A}_{x,y} \in \La^{x\rightarrow y}$, or in words, $\g$ connects $x$ and $y$ by a oriented path, 
    \item[$(ii)$] $|\Enc(\g \cap \mathbf{A}_{x,y})| - |\mathbf{T}_{0,x,y}| \geq \eta \dist x - y \dist^{\frac{3}{2}}$.
\end{enumerate}
\end{Def}

\begin{figure}[ht]
    \centering
    \includegraphics{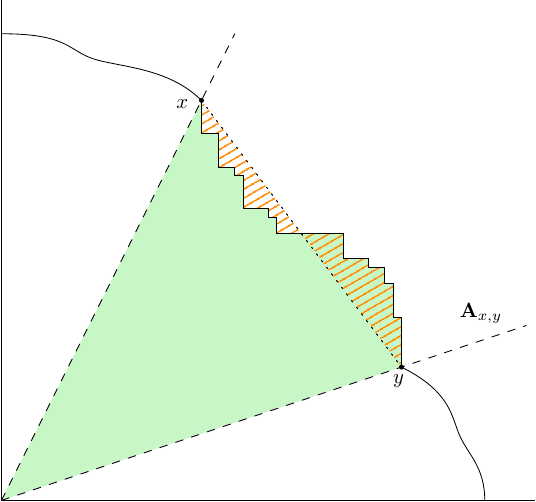}
    \caption{An illustration of the event $\GAC{x}{y}{\eta}$. The light green shaded region corresponds to $\Enc(\g \cap \mathbf{A}_{x,y})$. The event that the path $\g$ satisfies $\GAC{x}{y}{\eta}$ means that the (algebraic) area of the orange hatched region exceeds $\eta \Vert x-y\Vert^{3/2}$.}
    \label{fig: the event GAC}
\end{figure}

\begin{proof}[Proof of Proposition \textup{\ref{prop: rough upper bounds}}]
Let $1\leq t \leq N^{2/3-\epsilon}$. Introduce the event
\begin{equation}
    \BiF(t):= \big\{ \MiFL(\G) \geq tN^{\frac{2}{3}+\epsilon} \big\}.
\end{equation}
We pick uniformly at random, and independently of $\G$, two points  $\mathbf{x},\mathbf{y}$ (with $\arg(\textbf{x})>\arg(\textbf{y})$) in $B_{K_1N}$ and call $\GH$ the event that these two points coincide with the extremities of $\mathsf{MeanFac}$.
The measure associated with this procedure is denoted $\PP$.

Let us call $\bad$ the event that $\theta(\textbf{x}, \textbf{y}) \in [0, \eps) \cup (\pi/2-\eps, \pi/2]$ for the $\eps>0$ given by Proposition~\ref{prop: bad angle}. Let $\mathcal{E}:=\GH \cap (\bad)^c\cap \lbrace \G\subset B_{K_1N}\setminus B_{K_2N}\rbrace$.  Resample $\G$ between $\mathbf{x}$ and $\mathbf{y}$ according to the uniform law among paths $\g\in \Lambda^{\mathbf{x}\rightarrow \mathbf{y}}$ such that $(\G\setminus \G_{\mathbf{x},\mathbf{y}})\cup \g$ satisfies the area condition and call $\tilde{\G}$ the resampling. Note that $\tilde{\G}$ has law $\PlN$.

Since we are resampling along a facet, if $(\G, (\bfx,\bfy))$ realises $\mathcal{E}$, then $\G_{\mathbf{x},\mathbf{y}}$ lies underneath the segment joining $\textbf{x}$ and $\textbf{y}$. Thus, if $(\G, (\bfx,\bfy))\in \mathcal{E}$, replacing $\Gamma_{\mathbf{x},\mathbf{y}}$ by any path satisfying $\mathsf{GAC}(\mathbf x, \mathbf y, \eta)$ still satisfies the area condition and\footnote{We use the fact that $\GAC{\bfx}{\bfy}{\eta}$ can be seen as an event on $\Lambda^{\bfx\rightarrow \bfy}$ and for convenience we keep the same notation when looking at this event under the two measures of interest.}
\begin{equation}\label{eq: compare with uniform measure}
    \PP\otimes\PlN[\tilde{\G}\in \mathsf{GAC}(\mathbf x, \mathbf y, \eta)\text{ } | \text{ } (\G, (\bfx,\bfy)) \in \BiF(t)\cap \mathcal{E}]\geq \PP_{\mathbf{x},\mathbf{y}}[\mathsf{GAC}(\mathbf x, \mathbf y, \eta)].
\end{equation}
where we recall that $\mathbb P_{a,b}$ is the uniform law on $\Lambda^{a\rightarrow b}$. Indeed, conditioning on $(\G, (\bfx,\bfy)) \in \BiF(t)\cap \mathcal{E}$, the above inequality is equivalent to 
\begin{equation}
    \frac{|\lbrace\g \in \Lambda^{\mathbf{x}\rightarrow \mathbf{y}}, \text{ }\g \in \mathsf{GAC}(\mathbf x, \mathbf y, \eta)\rbrace|}{|\lbrace\g \in \Lambda^{\mathbf{x}\rightarrow \mathbf{y}}, \text{ }(\G\setminus \G_{\mathbf{x},\mathbf{y}})\cup\g \in \Lambda^{N^2}\rbrace|}\geq \frac{|\lbrace\g \in \Lambda^{\mathbf{x}\rightarrow \mathbf{y}}, \text{ }\g \in \mathsf{GAC}(\mathbf x, \mathbf y, \eta)\rbrace|}{|\Lambda^{\mathbf{x}\rightarrow \mathbf{y}}|},
\end{equation}
which is trivially true.
Now, if $\theta(\bfx,\bfy)\in [\varepsilon,\pi/2-\varepsilon]$, standard computations on the simple random walk bridge (see Lemma~\ref{good shape uniform law 1} for more details) ensure that for $\eta>0$ sufficiently small there exists $c_1=c_1(\eta,\varepsilon)>0$ such that,
\begin{equation}
   \PP_{\mathbf{x},\mathbf{y}}[\mathsf{GAC}(\mathbf x, \mathbf y, \eta)]\geq c_1.
\end{equation}
Notice also that
\begin{equation}
    \{\tilde{\G}\in \mathsf{GAC}(\mathbf x, \mathbf y, \eta), (\G,(\bfx,\bfy)) \in \BiF(t)\cap\mathcal{E}\} \subset \big\{ \EA(\tilde{\G})\geq \eta t^{\frac{3}{2}}N^{1+\frac{3}{2}\epsilon}\big\}.
\end{equation}
Putting all the pieces together, and using the fact that $\tilde{\G}$ has law $\PlN$,
\begin{equation}
    \PP\otimes\PlN[\BiF(t)\cap \mathcal{E}]\leq c_1^{-1}\PlN\big[\EA(\G)\geq \eta t^{\frac{3}{2}}N^{1+\frac{3}{2}\epsilon}\big].
\end{equation}
Finally, notice that for some $c_2>0$,
\begin{equation}
    \PP\otimes\PlN\left[\GH \text{ }|\text{ }\BiF(t)\cap \lbrace \G\subset B_{K_1N}\setminus B_{K_2N}\rbrace\right]\geq \frac{c_2}{N^4}.
\end{equation}
Putting all the pieces together,
\begin{multline}\label{equation all together upperbound 1}
    \PlN[\BiF(t)] \leq \frac{N^4}{c_2}\cdot\PP\otimes\PlN[\BiF(t)\cap \mathcal{E}]\\+\frac{N^4}{c_2}\cdot\PP\otimes\PlN[\BiF(t)\cap\GH \cap (\bad)\cap \lbrace \G\subset B_{K_1N}\setminus B_{K_2N}\rbrace]\\+\PlN\left[\BiF(t)\cap \lbrace \Gamma \subset B_{K_1N}\setminus B_{K_2N}\rbrace^c\right].
\end{multline}
To bound the second term we use Proposition~\ref{prop: bad angle} above. Notice that,
\begin{multline*}
    \BiF(t)\cap\GH \cap (\bad)\cap \lbrace \G\subset B_{K_1N}\setminus B_{K_2N}\rbrace\\\subset \bigcup_{\substack{{u,v}\in B_{K_1N}\setminus B_{K_2N}\\\Vert u-v\Vert\geq  tN^{\frac{2}{3}+\epsilon}}}\big\{\Gamma \in \mathsf{Bad}^+_{\varepsilon,\mathbf{A}_{u,v}} \cup \mathsf{Bad}^+_{\varepsilon,\mathbf{A}_{u,v}}\big\},
\end{multline*}
where $\mathbf{A}_{u,v}$ is the cone of apex $0$ bounded by $u$ and $v$. Finally, apply Lemma~\ref{lem: confinement lemma}, Proposition~\ref{prop: excessArea}, and Proposition~\ref{prop: bad angle} to get that for some $c=c(\varepsilon)>0$,
\begin{equation}
    \PlN[\BiF(t)] \leq\frac{N^4}{c_2}\cdot c_1^{-1} \cdot \exp(-ct^{\frac{3}{2}}N^{\frac{3}{2}\epsilon})+\frac{N^4}{c_2}\cdot (K_1N)^4\cdot 2\exp(-c t N^{\frac{2}{3}+\epsilon})+\exp(-cN),
\end{equation}
and thus the first inequality (recall that $t\leq N^{2/3-\epsilon}$). 

For $t\geq 1\footnote{In fact, one can take $t\geq3\sqrt{2}N^{-1/3}$ here.}$ notice that
\begin{equation}
    \PlN\big[\MiFL(\G)\geq tN^{\frac{4}{3}}\big]\leq \PlN\big[|\G|\geq (tN^{1/3})\cdot N\big].
\end{equation}
Using Lemma~\ref{lem: exp tail length} we get the result.
\end{proof}

\section{Analysis of the mean facet length and the mean local roughness}\label{section analysis of the mean facet length and the mean local roughness}

\subsection{Upper bounds}\label{subsection upper bounds}
In the following subsection, we actually prove a stronger statement than the upper bounds of Theorem~\ref{Theorem meanfl and meanlr}. Recall that $\MiFL$ and $\MiLR$ where defined in Definition~\ref{def: meanfl meanlr}. The goal of this subsection is to prove upper tail estimates on $\MiFL$ and $\MiLR$.
Our first result is a refinement of Proposition~\ref{prop: rough upper bounds}.
\begin{Prop}[Upper tail of $\MiFL$]\label{prop meanfl}
There exist $\tilde{c},c, C >0$ such that, for any $\tilde{c}\leq t \leq N^{2/3}$,
\begin{equation}
    \PlN\big[ \MiFL(\Gamma) \geq tN^{\frac{2}{3}}\big] \leq C\e^{-ct^{3/2}}.
\end{equation}
\end{Prop}

\begin{Prop}[Upper tail of $\MiLR$]\label{prop meanlr}
There exist $\tilde{c},c, C > 0$ such that for any $\tilde{c}\leq t \leq N^{5/6}$,
\begin{equation}
    \PlN\big[\MiLR(\G) \geq tN^{\frac{1}{3}}\big] \leq C\e^{-ct^{6/5}}.
\end{equation}
\end{Prop}
\begin{Rem} For larger values of $t$, we can use Lemma~\ref{lem: exp tail length} and Proposition~\ref{prop: rough upper bounds} to get explicit tails for $\MiFL$ and $\MiLR$. In particular, we obtain that
\begin{equation}
    \limsup_{N\rightarrow \infty}\mathbb E^{N^2}_\la\left[\frac{\MiFL(\G)}{N^{2/3}}\right]<\infty, \qquad \limsup_{N\rightarrow \infty}\mathbb E^{N^2}_\la\left[\frac{\MiLR(\G)}{N^{1/3}}\right]<\infty.
\end{equation}
\end{Rem}
We keep the notations of Section~\ref{section: exemple resampling}. We will again use the event $\GAC{x}{y}{\eta}$ introduced in Definition~\ref{definition good area capture}.
While it is difficult to estimate the probability of the event $\GAC{x}{y}{\eta}$ under $\PlN$, it is much simpler to estimate it under $\PP_{x,y}$ which is the uniform law over $\Lambda^{x\rightarrow y}$ (assuming that we chose $x,y\in \mathbb N^2$ such that $\Lambda^{x\rightarrow y}\neq \emptyset$). The following result, whose proof is postponed to the appendix, gives us a lower bound on $\PP_{x,y}[\GAC{x}{y}{\eta}]$. Denote by $\theta(x,y)\in [0,\pi/2]$ the angle formed by the horizontal axis and the segment joining $x$ and $y$.
\begin{Lemma}\label{good shape uniform law 1} Let $\eps, \eta > 0$. There exist $c=c(\varepsilon,\eta) > 0$ and $N_0=N_0(\varepsilon,\eta) > 0$ such that for any $x,y \in \N^2$ satisfying $\Lambda^{x\rightarrow y}\neq \emptyset$, $\varepsilon\leq \theta(x,y)\leq \pi/2-\varepsilon$ and $\Vert x-y\Vert\geq N_0$,
\begin{equation}
    \PP_{x,y} [\GAC{x}{y}{\eta} ] \geq c.
\end{equation} 
\end{Lemma}
\begin{Rem}
\begin{enumerate}
 \item The constant $c(\varepsilon,\eta)$ in the result above degenerates as $\varepsilon\rightarrow 0$. However, as we saw in Proposition~\ref{prop: bad angle}, up to an event of small probability, we will be able to assume that $\theta$ stays bounded away from $0$ and $\pi/2$.
 \item The constant $\tilde{c}$ in the statement of Propositions~\ref{prop meanfl} and~\ref{prop meanlr} will be chosen in terms of $N_0=N_0(\varepsilon,\eta)$ for $\varepsilon$ given by Proposition~\ref{prop: bad angle} and a properly chosen $\eta>0$.
\end{enumerate}
\end{Rem}

\subsubsection{Proof of Proposition~\ref{prop meanfl}}
As in the proof of Proposition~\ref{prop: rough upper bounds}, we would like to resample along the mean facet. However, as we saw, to pick the extremities of the mean facet without revealing any information about the path $\G$ comes with a cost of $O(N^4)$. This cost was previously handled by the requirement that $t\geq N^\epsilon$, but now we allow $t$ to be of smaller order. To tackle this difficulty, we will not resample between the exact extremities of the mean facet but between points which approximate them (see Figure~\ref{fig:illustration mean upper bound}).

\begin{proof}[Proof of Proposition \textup{\ref{prop meanfl}}]Let $t>0$ be fixed, and let $\delta, \eps>0$ be two small parameters that will be fixed later. Thanks to Proposition~\ref{prop: rough upper bounds}, we can additionally assume that $t\leq N^\epsilon$ for some $\epsilon\in (0,1/3)$. Indeed, if $N^\epsilon \leq t\leq N^{2/3}$, using \eqref{eq: prop rough upper 1}, we get
\begin{equation}
    \PlN\big[\MiFL(\G)\geq (tN^{-\epsilon})N^{2/3+\epsilon}\big]\leq C\e^{-ct^{3/2}}.
\end{equation}
We now assume $t\leq N^{\epsilon}$. Let us introduce the following event
\begin{equation}
    \BiF(t) = \Big\{ \MiFL(\Gamma) \geq tN^{\frac{2}{3}} \Big\}. 
\end{equation}
We also introduce, for any integer $k \geq 1$,
\begin{equation}
    \BiF_k(t) = \Big\{ ktN^{\frac{2}{3}} \leq \MiFL(\Gamma) < (k+1)tN^{\frac{2}{3}}  \Big\},
\end{equation}
so that the family $(\BiF_k(t))_{k\geq 1}$ partitions the event $\BiF$.

Notice that for large values of $k$, we can use Proposition~\ref{prop: rough upper bounds} again to conclude. If $k_N(\epsilon,t)$ is the smallest $k\geq 1$ such that $tk\geq N^\epsilon$, using \eqref{eq: prop rough upper 1} again,
\begin{equation}
    \PlN\Big[\MiFL(\G)\geq tk_N(\epsilon,t) N^{\frac{2}{3}}\Big]\leq Ce^{-c(tk_N)^{3/2}}.
\end{equation}
In particular it suffices to control the probabilities of $\BiF_k(t)$ for $1\leq k\leq k_N(\epsilon,t)$. In fact, it is sufficient to control the probabilities of $\BiF_k(t/3)$ for $3\leq k \leq k_N(\epsilon,t/3)=3k_N(\epsilon,t)$. We now let $s=t/3$ and $k_N=k_N(\epsilon,s)$.

Fix $3\leq k\leq k_N$. For $j\geq 0$, define the angles 
\begin{equation}
    \theta^+_j=\theta^+_j(\delta,k):=\frac{\pi}{4}+\arctan\left(\frac{j \delta k s}{N^{1/3}}\right), \qquad \theta^-_j=\theta^{-}_j(\delta,k):=\frac{\pi}{4}-\arctan\left(\frac{j \delta k s}{N^{1/3}}\right),
\end{equation}
and denote $\ell_j^+=\ell_j^+(\delta,k)$ (resp. $\ell_j^-=\ell_j^-(\delta,k)$) the half-line rooted at the origin and of argument $\theta^+_j$ (resp. $\theta^-_j$).

When $\BiF_k(s)$ occurs, there are at most $2/\delta$ indices $j$ such that $\ell_j^+$ (resp. $\ell_j^-$) intersects the mean facet. For a sample $\Gamma$ and $0\leq j \leq 2/\delta$, we call $a_j=a_j(\Gamma,\delta,k)$ (resp. $b_j=b_j(\Gamma,\delta,k)$) the points of $\mathbb N^2$ that are the closest\footnote{There might be ambiguity in the choice of such points and we solve this issue by asking $a_j$ (resp. $b_j$) to lie on the right (resp. on the left) of $\ell_j^+$ (resp. $\ell_j^-$).} to the intersection between $\Gamma$ and $\ell_j^+$ (resp. $\ell_j^-$). Observe that these points are defined from $\Gamma$ by a deterministic procedure which means that resampling a sample $\Gamma$ between two of these points yields an output distributed according to $\PlN$.
In particular, we can define $\bfa=\bfa(\Gamma,\delta,k)$ (resp. $\bfb=\bfb(\Gamma,\delta,k)$) to be the point of $\Gamma$ defined as follows: if $j_0$ (resp. $j_1$) is the largest $j$ such that $\ell_j^+$ (resp. $\ell_j^-$) intersects the mean facet, then $\bfa=a_{j_0}$ (resp. $\bfb=b_{j_1}$).

In order to be able to pick $\bfa$ and $\bfb$ and resample between them without revealing the portion $\Gamma_{\bfa,\bfb}$ of $\Gamma$ between $\bfa$ and $\bfb$, we introduce some extra randomness. We pick uniformly, a pair of points $(\mathbf{x},\mathbf{y})$ in $\lbrace a_j, \: 0\leq j\leq 2/\delta\rbrace \times \lbrace b_j, \: 0\leq j\leq 2/\delta\rbrace$ and call $\mathbb P$ the measure associated with this random procedure. Let $\mathsf{GoodHit}_k$ be the event that $(\mathbf{x},\mathbf{y})=(\bfa,\bfb)$.  We are going to show that with uniformly positive probability, if we resample between $\bfa$ and $\bfb$, the resampled path captures a linear excess of area, an event which is exponentially unlikely by Proposition~\ref{prop: excessArea}.
 \begin{figure}
     \centering
     \includegraphics[scale = 0.7]{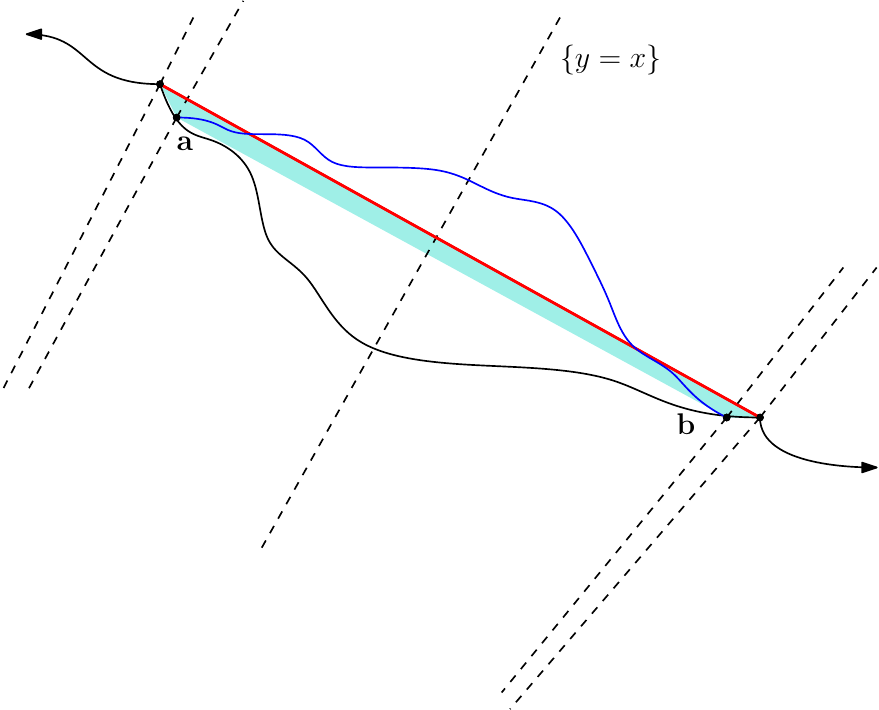}
     \caption{Illustration of the proof of Proposition~\ref{prop meanfl}. The original is path depicted in black. With large probability, the new resampled blue path captures an amount of area sufficient to compensate the area that could have been lost (in turquoise) due to the fact that $\mathbf{a}$ and $\mathbf{b}$ do not exactly coincide with the endpoints of the red facet. The event $\mathsf{SD}$ ensures that $\mathbf{a}$ and $\mathbf{b}$ are not atypically far from these two endpoints, allowing us to control the area of the turquoise region.}
     \label{fig:illustration mean upper bound}
 \end{figure}

As in the proof of Proposition~\ref{prop: rough upper bounds}, let us call $\bad$ the event that $\theta(\textbf{x}, \textbf{y}) \in [0, \eps) \cup (\pi/2-\eps, \pi/2]$ for the $\eps>0$ given by Proposition~\ref{prop: bad angle}. Write $\mathcal{E}_k:=\GH_k \cap (\bad)^c\cap \lbrace \G\subset B_{K_1N}\setminus B_{K_2N}\rbrace$. Resample $\G$ between $\mathbf{x}$ and $\mathbf{y}$ according to the uniform law among paths $\g\in \Lambda^{\mathbf{x}\rightarrow \mathbf{y}}$ such that $(\G\setminus \G_{\mathbf{x},\mathbf{y}})\cup \g$ satisfies the area condition and call $\tilde{\G}$ the resampling. Note that $\tilde{\G}$ has law $\PlN$. In the proof of Proposition~\ref{prop: rough upper bounds}, since we were resampling along the mean facet, it was clear that replacing the path below the facet by any path having a positive area (above the facet) kept the total area condition. However, in our case, the approximate ``mean facet resampling'' comes with a potential area loss. This possibility is ruled out for sufficiently small values of $\delta$. For $\eta>0$, introduce the (small deviation) event
\begin{equation}
    \mathsf{SD}(\eta):=\big\{\max(\LR(\bfa), \LR(\bfb) ) < \eta (\delta s (k+1))^{1/2}N^{1/3}\big\}.
\end{equation}

\begin{Claim}\label{Claim 1}
There exist $\delta_0>0$ sufficiently small and $c_0=c_0(\delta_0,\eta)>0$ such that the following holds: if $((\bfx,\bfy),\Gamma) \in\BiF_k(s)\cap \mathcal{E}_k\cap \mathsf{SD}(\eta)$, then, if we replace $\Gamma_{\mathbf{x},\mathbf{y}}$ by a path $\gamma\in \GAC{\mathbf{x}}{\mathbf{y}}{\eta}$, the excess area of the resulting path is at least $c_0(sk)^{3/2}N$. In particular, the new path lies in $\Lambda^{N^2}$.
\end{Claim}
\begin{proof}[Proof of Claim~\textup{\ref{Claim 1}}]
Notice that by definition $(\bfx,\bfy)=(\bfa,\bfb)$ are located below the mean facet. Moreover, these points are at distance at most $\delta(k+1)tN^{2/3}$ from the endpoints of the mean facet. Changing $\G$ between $\bfa$ and $\bfb$ by an element of $\GAC{\bfa}{\bfb}{\eta}$ gives a path $\tilde{\G}$ which gained an area at least $\eta\Vert \bfa-\bfb\Vert^{3/2}-\mathcal{A}(\bfa,\bfb,\G)$ where $\mathcal{A}(\bfa,\bfb,\G)$ is the area of the quadrilateral delimited by the extremities of the mean facet of $\G$ and the pair $(\bfa,\bfb)$, see Figure~\ref{fig:illustration mean upper bound}. Given the hypothesis on $\G$, we find that
\begin{equation}
    \mathcal{A}(\bfa,\bfb,\G)\leq \eta \sqrt{\delta} (s(k+1))^{3/2}N.
\end{equation}
Hence, for $\delta=\delta_0>0$ sufficiently small, there exists $c_0=c_0(\delta_0,\eta)>0$ such that (recall that $k\geq 3$)
\begin{equation}
    \EA(\tilde{\G})\geq \eta s^{3/2} (k-2)^{3/2}N-\eta \sqrt{\delta_0} (s(k+1))^{3/2}N\geq c_0 (sk)^{3/2}N.
\end{equation}
\end{proof}
We now fix $\delta=\delta_0$ and let $\eta>0$ to be fixed later. Using the same reasoning as in \eqref{eq: compare with uniform measure} we get that for $\Vert \bfa-\bfb\Vert\geq N_0$,
\begin{multline*}
    \mathbb P\otimes\PlN[\tilde{\G}\in \GAC{\mathbf{x}}{\mathbf{y}}{\eta}\text{ } | \text{ } ((\bfx,\bfy),\G)\in \BiF_k(s)\cap \mathcal{E}_k\cap \mathsf{SD}(\eta)]\\\geq \PP_{\mathbf{a},\mathbf{b}}[\GAC{\mathbf{a}}{\mathbf{b}}{\eta}]\geq c_1,
\end{multline*}
where $c_1,N_0>0$ are given by Lemma~\ref{good shape uniform law 1}. Since $\Vert \bfa -\bfb\Vert\geq s(k-2)N\geq sN$, the above bound always hold as soon as $s\geq N_0/N$. Assume that $s\geq N_0$.
By the claim above, one has 
\begin{equation}
    \big\{ \tilde{\G}\in \GAC{\mathbf x}{\mathbf y}{\eta}, \:\G\in \BiF_k(s)\cap\mathcal{E}_k\cap \mathsf{SD}(\eta)\big\}\subset \big\{ \EA(\tilde{\G})\geq c_0 (tk)^{3/2}N\big\}.
\end{equation}
Putting all the pieces together, and using the fact that $\tilde{\G}$ has law $\PlN$,
\begin{equation}\label{eq: bigmeanfl+goodhit bound by ea}
    \mathbb P\otimes\PlN[\BiF_k(s)\cap \mathcal{E}_k\cap \mathsf{SD}(\eta)]\leq c_1^{-1}\PlN[\EA(\G)\geq c_0 (tk)^{3/2}N].
\end{equation}
We  now argue that $\mathsf{SD}(\eta)$ occurs with uniform non-zero probability under $\PlN$.
\begin{Claim}\label{claim 2}
There exist $\eta>0$ and $c_2=c_2(\eta)>0$ such that
\begin{equation}
    \PlN[\mathsf{SD}(\eta) \: | \: \BiF_k(s)]\geq c_2.
\end{equation}
\end{Claim}
\begin{proof}[Proof of Claim~\textup{\ref{claim 2}}.] The claim is a consequence of the coupling obtained in Proposition~\ref{prop: couplage 1}. Indeed, $\bfa$ and $\bfb$ belong to the excursion of $\Gamma$ under the mean facet. Thus, conditioning on the extremities of the mean facet, which we call $\bfa_\mathsf{fac}$ and $\bfb_\mathsf{fac}$, the local roughness of $\bfa$ (resp. $\bfb$) is dominated by the height of a negative random walk excursion between $\bfa_{\mathsf{fac} }$ and $\bfb_{\mathsf{fac}}$ after $\mathrm{d}_\bfa$ steps, where $\mathrm{d}_\bfa$ is the number of steps from $\bfa_{\mathsf{fac}}$ to $\bfa$ (resp  $\mathrm{d}_\bfb$ steps, where $\mathrm{d}_\bfb$ is the number of steps from $\bfb$ to $\bfb_{\mathsf{fac}}$) which is of order $\sqrt{\mathrm{d}_\bfa}$ (resp. $\sqrt{\mathrm d_\bfb}$). Since by choice of $\bfa$ and $\bfb$ we can deterministically bound $\mathrm d_\bfa$ and $\mathrm d_\bfb$ by $O(\delta s kN^{2/3})$, we get the claim choosing $\eta$ large enough. 
\end{proof}
We are now able to conclude. Let $\eta>0$ be given by the claim above. 
Notice that
\begin{equation}\label{estimation good}
    \mathbb P \otimes\PlN [ \mathsf{Goodhit}_k ~\Big\vert~ \BiF_k(s)\cap \mathsf{SD}(\eta)  ]\geq \left( \tfrac{\delta}{2}\right)^2.
\end{equation}
Write
\begin{multline}\label{equation all together upperbound claim}
    \PlN[\BiF_k(s)] \leq \left(\tfrac{2}{\delta}\right)^2\cdot c_2^{-1}\cdot\mathbb P\otimes\PlN[\BiF_k(s)\cap \mathcal{E}_k\cap \mathsf{SD}(\eta)]\\+\left(\tfrac{2}{\delta}\right)^2\cdot\mathbb P\otimes\PlN[\BiF_k(s)\cap\GH_k \cap (\bad)\cap \lbrace \G\subset B_{K_1N}\setminus B_{K_2N}\rbrace]\\+\PlN\left[\BiF_k(s)\cap \lbrace \Gamma \subset B_{K_1N}\setminus B_{K_2N}\rbrace^c\right].
\end{multline}
To bound the first term we use \eqref{eq: bigmeanfl+goodhit bound by ea} together with Proposition~\ref{prop: excessArea}. To bound the second term we use Proposition~\ref{prop: bad angle} and proceed as in the proof of Proposition~\ref{prop: rough upper bounds}. Finally, the last term is bounded using Lemma~\ref{lem: confinement lemma}. Putting all the pieces together, we get the existence of $c,C>0$ such that for $3\leq k\leq k_N$ (recall that $s=t/3$),
\begin{equation}
    \PlN[\BiF_k(s)] \leq Ce^{-c(tk)^{3/2}}.
\end{equation}
As a consequence, setting $\tilde c=3N_0$, there exists some constants $c',C'>0$ such that, for all $\tilde{c}\leq t \leq N^\epsilon$,
\begin{eqnarray*}
    \PlN[\MiFL(\G)\geq t N^{\frac{2}{3}}] 
    &\leq& 
    \sum_{k=3}^{k_N}\PlN[\BiF_k(t/3)]+\PlN\big[\MiFL(\G)\geq (t/3)k_N N^{\frac{2}{3}}\big]
    \\
    &\leq& C'e^{-c't^{3/2}}.
\end{eqnarray*}
\end{proof}

\subsubsection{Proof of Proposition~\ref{prop meanlr}}
From Proposition~\ref{prop meanfl}, it is easy to deduce the bound of Proposition~\ref{prop meanlr} using the coupling introduced in Subsection~\ref{subsection coupling}.

\begin{proof}[Proof of Proposition \textup{\ref{prop meanlr}}]
Let $t>0$ and $\delta > 0$ be a (small) constant to be fixed later . We split the event $\lbrace \MiLR > tN^{1/3} \rbrace$ according to the value of $\MiFL$,
\begin{multline*}
    \PlN\big[\MiLR(\G) > tN^{\frac{1}{3}} \big] = \PlN\big[ \MiLR(\G) > tN^{\frac{1}{3}}, \MiFL(\G) \leq t^{2-\delta}N^{\frac{2}{3}} \big] \\ + \PlN\big[  \MiLR(\G) > tN^{\frac{1}{3}}, \MiFL(\G) \geq t^{2-\delta}N^{\frac{2}{3}} \big].
\end{multline*}
By Proposition~\ref{prop meanfl}, we know that, for $\tilde{c}\leq t^{2-\delta}\leq N^{2/3}$,
\begin{equation}
    \PlN[  \MiLR(\G) > tN^{\frac{1}{3}}, \MiFL(\G) \geq t^{2-\delta}N^{\frac{2}{3}} ] \leq C\exp(-ct^{\frac{3}{2}(2-\delta)}).
\end{equation}
Now for $j \geq 0$, let $A_j$ be the following event: 
\begin{equation}
    A_j := \{ \MiLR(\G) > tN^{\frac{1}{3}}, \MiFL(\G) = j \}.
\end{equation}
We can now make use of Lemma~\ref{Lemme domination roughness random walk sous une facette} to argue that for every $j \leq t^{2-\delta}N^{2/3}$,
\begin{eqnarray*}
    \PlN\left[A_j\right] &=& \sum_{\substack{\Vert a - b \Vert = j \\ \arg(a)> \frac{\pi}{4}\geq \arg(b)}}\PlN[\MiLR(\G) \geq tN^{\frac{1}{3}} ~\vert~ \mathsf{Fac}_{a,b} ]\PlN\left[\mathsf{Fac}_{a,b}\right]
    \\
    &\leq& C\exp\Big(-c\frac{t^2N^{2/3}}{j}\Big)\PlN\left[\MiFL(\G) = j \right]\\
    &\leq& C\exp(-ct^\delta)\PlN\left[\MiFL(\G) = j \right].
\end{eqnarray*}
The first equality comes from the fact that there exists exactly one pair $a,b$ such that $\arg(a)>\frac{\pi}{4}\geq\arg(b)$ and $\mathsf{Fac}(a,b)$ occurs. Moreover, if $\mathsf{Fac}(a,b)$ occurs with such a choice of $a$ and $b$, then $[a,b]$ will automatically be the mean facet, so Lemma~\ref{Lemme domination roughness random walk sous une facette} holds for the second line. Thus, we obtained,
\begin{eqnarray*}
     \PlN[\MiLR(\G) > tN^{\frac{1}{3}} ] &\leq& C\exp(-ct^{\frac{3}{2}(2-\delta)})  + \sum_{j=0}^{t^{2-\delta}N^{2/3}} C\exp(-ct^\delta)\PlN\left[\MiFL(\G) = j \right] \\
     &\leq& C\exp(-ct^{\frac{3}{2}(2-\delta)}) + C\exp(-ct^\delta).
\end{eqnarray*}
Equating the exponents yields an optimal value of $\delta = \frac{6}{5}$ and the result.
 \end{proof}
 
\subsection{Lower bounds}\label{subsectio lower bounds}
Recall the setting of Theorem~\ref{Theorem meanfl and meanlr}. Our target estimates are the following. 
\begin{Prop}\label{prop lower bound meanlr}
There exists a function $F_1: \R^+ \rightarrow [0,1]$ satisfying  $\lim_{t \rightarrow 0^+} F_1(t) = 0$, such that for any $t>0$ small enough,
\begin{equation}
    \limsup_{N \rightarrow \infty }\PlN\big[ \MiLR(\G) < tN^{\frac{1}{3}}\big] \leq F_1(t).
\end{equation}
\end{Prop}

\begin{Prop}\label{prop lower bound meanfl}
There exists a function $F_2: \R^+ \rightarrow [0,1]$ satisfying $\lim_{t \rightarrow 0^+} F_2(t) = 0$, such that for any $t > 0$ small enough, 
\begin{equation}
\limsup_{N\rightarrow \infty} \PlN\big[ \MiFL(\G) < tN^{\frac{2}{3}}\big] \leq F_2(t).
\end{equation}
\end{Prop}
\begin{Rem}
As before, we obtain  
\begin{equation}
    \liminf_{N\rightarrow \infty}\mathbb E^{N^2}_\la\left[\frac{\MiFL(\G)}{N^{2/3}}\right]>0, \qquad \liminf_{N\rightarrow \infty}\mathbb E^{N^2}_\la\left[\frac{\MiLR(\G)}{N^{1/3}}\right]>0.
\end{equation}
\end{Rem}

Let us briefly describe the strategy of the proof. We carefully analyse the marginal of $\PlN$ in the cone $\bfA_N$ of apex 0 and angular opening $N^{-1/3}$. Due to the Brownian Gibbs property, conditionally on the outside of the cone, this marginal is the law of a random walk bridge conditioned on capturing a random amount of area.  We first prove in Lemma~\ref{lemma aire secteur fixé pas trop grande} that this random area is of order $N$, with Gaussian tails. Hence, the marginal of $\PlN$ in $\bfA_N$ is absolutely continuous with respect to the law of a random walk bridge without the area conditioning, in the large $N$ limit. The strategy of proof is thus clear: we prove the corresponding statements for the unconditioned random walk bridge (see Lemmas~\ref{lemme mean lr for a Brownian process} and~\ref{lemme mean fl for a Brownian process}), and transfer them to $\PlN$ thanks to the latter observation.

Let us introduce a few quantities that we will be useful in the remainder of this section.
\begin{Def}
    For $N\geq 1$, we define $\theta_N := N^{-1/3}$. We set $\mathbf{A}_N$ to be the cone of apex 0 and angular opening of $\theta_N$ centered around the line $\lbrace y=x \rbrace$. If $\gamma \in \Lambda^{N^2}$, let us define $\bfx_N(\gamma)$ (resp. $\bfy_N(\gamma)$) to be point of $\gamma$ that lies in $\mathbf{A}_N$ and that is the closest to the line constituting the left (resp. right) boundary of $\mathbf{A}_N$.
    Finally, we define
\begin{equation}
    \mathcal{A}_N(\gamma) := \left|\Enc(\gamma_{\bfx_N(\g),\bfy_N(\g)} \cap \mathbf{A}_{\bfx_N(\g),\bfy_N(\g)})\right| - \left| \mathbf{T}_{0,\bfx_N(\g),\bfy_N(\g)} \right|.
\end{equation}

We shall also make a slight abuse of notation in considering $\mathcal{A}_N(\widetilde{\gamma})$ for some oriented path $\widetilde{\gamma}$ linking $\bfx_N(\gamma)$ to $\bfy_N(\gamma)$ (for some $\gamma \in \Lambda^{N^2}$). In that case, the meaning of $\mathcal{A}_N(\widetilde{\gamma})$ is clear: it is the corresponding $\mathcal{A}_N(\gamma')$ for any $\gamma' \in \Lambda^{N^2}$ which extends $\widetilde{\gamma}$ outside the cone $\mathbf{A}_{\bfx_N(\g),\bfy_N(\g)}$.
\end{Def}
The following lemma will play a central role in the proofs of Propositions~\ref{prop lower bound meanlr} and~\ref{prop lower bound meanfl}.
\begin{Lemma}\label{lemma aire secteur fixé pas trop grande}
There exist two constants $c,C>0$ such that for any $\beta > 0$, and for $N$ large enough, 
\begin{equation}
     \PlN\left[ \mathcal{A}_{N}(\Gamma) \geq \beta N \right] \leq C\e^{-c\beta^2}.
\end{equation}

\end{Lemma}

\begin{figure}
    \centering
    \includegraphics[scale=0.9]{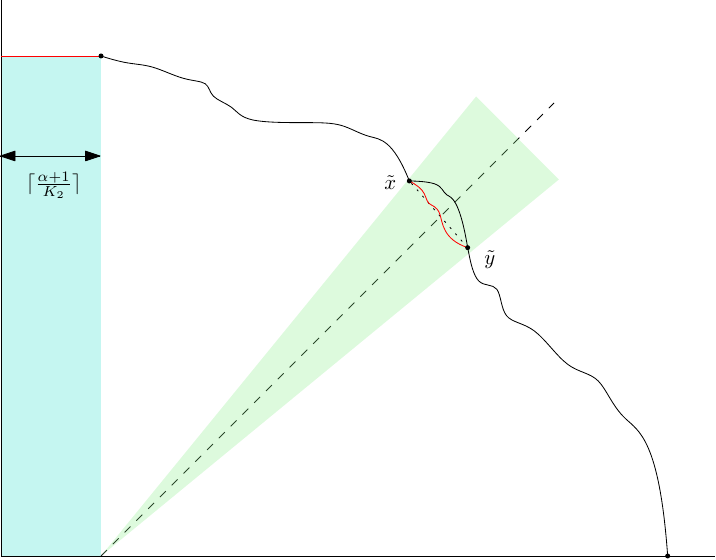}
    \caption{One possible output of the map $T_{x,y}$. The shift of $(x,y)$ is denoted $(\tilde{x},\tilde{y})$. The initial path is the black path (after a shift to the left). The two modifications involved in the procedure described above are represented in red. The red path in the light green shaded region has been sampled uniformly amongst the paths that capture a non-negative area (computed with respect to the black dashed line). The area captured by the turquoise region is sufficient to compensate the possible area loss due to the resampling of the path between $x$ and $y$.}
    \label{fig:figure surgery 2}
\end{figure}

\begin{proof} The proof of this statement heavily relies on the probabilistic multi-valued map principle stated in Lemma~\ref{lemme MVMP}. For $\alpha > 0$, we introduce
\begin{equation}
    \mathsf{BadArea}^N_{\alpha} :=  \big\{ \gamma \in \Lambda^{N^2},\: \mathcal{A}_{N}(\gamma) \in [\alpha N, (\alpha + 1)N]\big\}.
\end{equation}
Define for $x,y\in \mathbb N^2$,
\begin{multline*}
    A_{x,y}:=\mathsf{BadArea}^N_{\alpha} \cap \big\lbrace \G\subset B_{K_1N}\setminus B_{K_2N}\big\rbrace\cap \big\lbrace \bfx_N(\G)=x,\: \bfy_N(\G)=y\big\rbrace\\\cap \big\lbrace \G\notin \mathsf{Bad}^{+,-}_{\varepsilon,\bfx_N(\G),\bfy_N(\G)}\big\rbrace,
\end{multline*}
where $\mathsf{Bad}^{+,-}_{\varepsilon,\bfx_N(\G),\bfy_N(\G)}:=\mathsf{Bad}^+_{\varepsilon,\mathbf{A}_{\bfx_N(\G),\bfy_N(\G)}}\cup \mathsf{Bad}^-_{\varepsilon,\mathbf{A}_{\bfx_N(\G),\bfy_N(\G)}}$, and $\varepsilon$ is given by Proposition~\ref{prop: bad angle}.

Fix $x,y\in \mathbb N^2$ such that $\PlN[A_{x,y}]\neq 0$.

Let us define a multi-valued map $T_{x,y}: A_{x,y} \rightarrow \mathcal{P}(\Lambda^{N^2})$ via the following two-step procedure (see Figure~\ref{fig:figure surgery 2} for an illustration of the procedure). From some path $\gamma \in \mathsf{BadArea}^N_{\alpha}$, we can create a new oriented path of the first quadrant by erasing the portion of $\gamma$ lying between $x$ and $y$ (previously called $\gamma_{x,y}$) and replacing it by any oriented path $\widetilde{\gamma}_{x, y}$ linking $x$ to $y$ thus obtaining an intermediate path $\widetilde{\g}$ which we require to satisfy $\mathcal{A}_{N}(\widetilde{\g})\geq 0$, and finally by adding  $\lceil \frac{\alpha +1}{K_2} \rceil$ horizontal steps at the beginning of $\widetilde{\gamma}$ ($K_2$ is the constant defined in~\ref{lem: confinement lemma}). We define $T_{x,y}(\gamma)$ to be the subset of $\Lambda$ of the paths that can be obtained from $\gamma$ by this procedure. Notice that $T$ is well defined since if $\g\in A_{x,y}$, the intermediate path $\widetilde{\g}$ defined above satisfies $\mathcal{A}(\widetilde{\g})\geq N^2-(\alpha+1)N$, and it is easy to check that the first $\lceil \frac{\alpha +1}{K_2} \rceil$ horizontal steps that have been added to $\widetilde{\g}$ capture an area at least equal to $(\alpha+1)N$ (this comes from the fact that $\g\in \lbrace \G\subset B_{K_1N}\setminus B_{K_2N}\rbrace$). Hence, for any $\g \in A_{x,y}$, $T_{x,y}(\g)\subset \Lambda^{N^2}$.

We are going to apply Lemma~\ref{lemme MVMP} to $T_{x,y}$. Before doing so, let us notice that we may co-restrict the map defined above to the set
\begin{equation}
    B_{x,y}:=\lbrace\bfx_N(\G^{\textup{cut}})=x,\bfy_N(\G^{\textup{cut}})=y\rbrace \cap \{ \textup{The first }\lceil \tfrac{\alpha +1}{K_2} \rceil\textup{ steps are horizontal}\}\subset \La^{N^2},
\end{equation}
where $\G^{\textup{cut}}$ is the path $\G$ minus the first $\lceil \frac{\alpha +1}{K_2} \rceil$ steps, translated of $\lceil \frac{\alpha +1}{K_2} \rceil$ towards the left direction, see Figure~\ref{fig:figure surgery 2}.
Now, applying Lemma~\ref{lemme MVMP} yields
\begin{equation}\label{eq: mvmp 2}
    \PlN[A_{x,y}]\leq \varphi(T_{x,y})\psi(T_{x,y})\PlN[B_{x,y}],
\end{equation}
where the quantities $\varphi(T_{x,y})$ and $\psi(T_{x,y})$ are defined in Lemma~\ref{lemme MVMP}. It is very easy to check that
\begin{equation}
    \varphi(T_{x,y}) = \left(\frac{1}{\lambda}\right)^{\lceil\frac{\alpha+1}{K_2}\rceil},\qquad \psi(T_{x,y}) = \frac{\left|\lbrace \gamma\in \Lambda^{x\rightarrow y}, \mathcal{A}_{N}(\gamma) \in [\alpha N, (\alpha +1)N]  \rbrace\right|}{\left|\lbrace \gamma\in \Lambda^{x\rightarrow y}, \mathcal{A}_N(\gamma) \geq 0 \rbrace\right|}.
\end{equation}
Recall that $\PP_{x, y}$ is the uniform measure on $\Lambda^{x\rightarrow y}$. We can write
\begin{equation}
    \psi(T_{x,y}) = \PP_{x,y}\big[ \mathcal{A}_N(\gamma) \in [\alpha N, (\alpha+1)N] ~\big\vert~ \mathcal{A}_N(\gamma)\geq 0\big].
\end{equation}
Let us call $\theta(x,y)$ the positive angle formed by the segment $[x,y]$ and the horizontal line passing through $x$. Since we assumed that $\PlN[A_{x,y}]\neq 0$, we have $\theta(x,y)\in [\varepsilon,\pi/2-\varepsilon]$. Let $h_N:=\Vert y -x \Vert$. As it turns out, $h_N$ is of order $N^{2/3}$. 
\begin{Claim}\label{claim 3} There exist two constants $c=c(\varepsilon),C=C(\varepsilon) > 0$ such that
\begin{equation}
    h_N \in [cN^{2/3}, CN^{2/3}].
\end{equation}
\end{Claim}
\begin{proof}[Proof of Claim~\textup{\ref{claim 3}}]
The constraint $\Gamma \subset B_{K_1N}\setminus B_{K_2N}$ enforces that for $N$ large enough $h_N \geq \frac{1}{2}K_2N\theta_N=\frac{1}{2}K_2N^{2/3}$. Then, since we restricted ourselves to the case $\eps \leq \theta(x,y) \leq \pi/2-\varepsilon$, there exists $C_1> 0$ such that
\begin{equation}\label{equation majoration hn}
    h_N \leq  \frac{C_1}{\sin \varepsilon} N\theta_N=\frac{C_1}{\sin \varepsilon} N^{2/3}.
\end{equation}
\end{proof}
By translation invariance, $\PP_{x,y}$ only depends on $\theta=\theta(x,y)$ and $h_N$. Since these parameters are more relevant we write $\PP_{h_N}^\theta:=\PP_{x,y}$. Moreover, with this observation and to the cost of rotating the picture by $\pi/4$, we can always assume that samples of $\mathbb P^\theta_{h_N}$ start at $0$ and end at $e^{i\pi/4}(y-x)$, see Figure~\ref{fig:rotation of the path}. We will denote by $X^{(\theta)}$ a sample of $\mathbb P^\theta_{h_N}$. 
\begin{figure}
    \centering
    \includegraphics{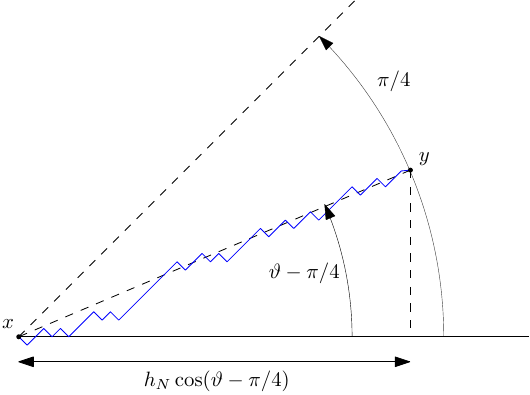}
    \caption{An illustration of a sample $X^{(\vartheta)}$ of $\mathbb P^{\vartheta}_{h_N}$ (in blue). This path corresponds to a symmetric random walk that does $\pm 2^{-1/2}$ jumps at times $k2^{-1/2}$ for $1\leq k \leq \sqrt{2}h_N\cos(\vartheta-\pi/4)$.}
    \label{fig:rotation of the path}
\end{figure}

Now, 
\begin{eqnarray*}
    \PP^\theta_{h_N} \big[ \mathcal{A}_N(X^{(\theta)}) \in [\alpha N, (\alpha+1)N]\big]
    &=& \PP^\theta_{h_N}\big[ h_N^{-3/2}\mathcal{A}_N(X^{(\theta)}) \in [\alpha N h_N^{-3/2}, (\alpha+1)N h_N^{-3/2}] \big] 
    \\
    &\leq& \PP^\theta_{h_N}\big[ h_N^{-3/2}\mathcal{A}_N(X^{(\theta)}) \in [C^{-3/2}\alpha , c^{-3/2}(\alpha+1)] \big]
    \\ 
    &\leq& \sup_{\vartheta\in [\varepsilon,\frac{\pi}{2}-\varepsilon]}\PP^\vartheta_{h_N}\big[ h_N^{-3/2}\mathcal{A}_N(X^{(\vartheta)}) \in [C^{-3/2}\alpha , c^{-3/2}(\alpha+1)]\big],
\end{eqnarray*}
where the second inequality follows by Claim~\ref{claim 3}. 

Classical invariance theorems for random walks (started at $0$) conditioned to reach $\alpha k$ after $k$ steps (see~\cite{liggett1968invariance} or~\cite[Remark~2.6]{caravennachaumont}) yield the following convergence in distribution\footnote{The convergence holds in the space $\mathcal{C}([0,1],\mathbb R)$ of continuous functions on $[0,1]$ equipped with the topology of uniform convergence.}: calling $(X_{th_N}^{(\pi/4)})_{0\leq t \leq 1}$ the trajectory\footnote{The sample is a priori only defined for the discrete times $k/\sqrt{2}$ for $0\leq k\leq \sqrt{2}h_N$ but we extend it to $[0,h_N]$ by linear interpolation.} of a sample of $\mathbb P^{\pi/4}_{h_N}$, one has that under $\mathbb P^{\pi/4}_{h_N}$
\begin{equation}
    \lim_{N\rightarrow \infty}\left(\frac{X_{th_N}^{(\pi/4)}}{\sigma\sqrt{2^{1/2}h_N}}\right)_{t\in[0,1]}=(\mathsf{BB}_{t})_{t\in [0,1]},
\end{equation}
where $(\mathsf{BB}_t)_{t\in [0,1]}$ is the standard Brownian bridge on $[0,1]$, and $\sigma=2^{-1/2}$. This observation, together with the strategy described in~\cite[Remark~2.6]{caravennachaumont}, yields a similar invariance principle for the trajectory $(X_{th_N\cos(\vartheta-\pi/4)}^{(\vartheta)})_{0\leq t \leq 1}$ of a sample of $\mathbb P^\vartheta_{h_N}$. For $\vartheta\in (0,\pi/2)$ and $t\in [0,1]$, denote
\begin{equation}\label{eq: definition f_N}
    f^{(\vartheta)}_N(t):=\frac{1}{\sigma(\vartheta)\sqrt{2^{1/2}h_N \cos(\vartheta-\pi/4)}}(X_{th_N\cos(\vartheta-\pi/4)}^{(\vartheta)}-th_N\sin(\vartheta-\pi/4)),
\end{equation}
where $\sigma(\vartheta):=2^{-1/2}\sqrt{1-\tan^2(\vartheta-\pi/4)}$. Then, for $\vartheta\in (0,\pi/2)$, under $\mathbb P^\vartheta_{h_N}$,
\begin{equation}\label{eq: cv distribution slope walk}
    \lim_{N\rightarrow \infty}\left(f_N^{(\vartheta)}(t)\right)_{t\in[0,1]}=(\mathsf{BB}_{t})_{t\in [0,1]},
\end{equation}
where the convergence holds in distribution. In particular, one has that
\begin{align*}
    h_N^{-3/2}\mathcal{A}_N(X^{(\vartheta)})
    &=h_N^{-3/2}\int_0^{h_N\cos(\vartheta-\pi/4)}(X^{(\vartheta)}_t-\tan(\vartheta-\pi/4)t)\mathrm{d}t
    \\&=2^{1/4}\sigma(\vartheta)[\cos(\vartheta-\pi/4)]^{3/2}\int_0^1 f_N^{(\vartheta)}(t)\mathrm{d}t.
\end{align*}
Hence, \eqref{eq: cv distribution slope walk} implies the following convergence in distribution under $\mathbb P^{\vartheta}_{h_N}$:
\begin{equation}
    \lim_{N\rightarrow\infty}h_N^{-3/2}\mathcal{A}_N(X^{(\vartheta)})= c_0(\vartheta)\int_0^1 \mathsf{BB}_t\mathrm{d}t,
\end{equation}
where $c_0(\vartheta):= 2^{1/4}\sigma(\vartheta)[\cos(\vartheta-\pi/4)]^{3/2}$. 
\begin{Rem}\label{rem: uniform cv in the angle}
It follows from~\cite[Remark 2.6]{caravennachaumont} that the latter convergence is uniform (for instance at the level of the convergence of cumulative distribution function) in $\vartheta \in [\eps, \pi/2 - \eps]$. Indeed, when the underlying random walk has increments with exponential moments, the authors identify the law of the random walk bridge conditioned to stay above a line of slope $\vartheta$ with a unconditioned random walk whose increments are given by a suitable exponential tilt of the increments of the former walk. Hence the Radon-Nikodym derivative of the $\vartheta$-tilted walk with respect to the $\pi/4$-tilted walk is explicit and is a continuous function of $\vartheta$. In particular, it is uniformly bounded on a compact interval such as $[\eps, \pi/2-\eps]$. We shall use this observation several times in what follows.
\end{Rem}
Call $\mathsf{P}$ the law of the Brownian bridge $(\mathsf{BB}_t)_{t\in[0,1]}$. The above observations yield that for $N$ large enough,
\begin{equation}\label{eq: first bound a(S) in alpha N}
    \PP^\theta_{h_N}\big[\mathcal{A}_N(X^\theta) \in [\alpha N, (\alpha +1)N]\big] 
    \leq 
    2\mathsf{P}\Big[\int_{0}^1 \mathsf{BB}_t \mathrm{d}t \geq c_0(\varepsilon)^{-1}C^{-3/2}\alpha\Big].
\end{equation} 
A similar reasoning yields the lower bound, for $N$ large enough,
\begin{equation}\label{eq: bound a(S) positive}
    \PP^\theta_{h_N}[ \mathcal{A}_N(X^\theta) \geq 0 ] \geq \frac{1}{2}\mathsf{P}\left[\int_{0}^1 \mathsf{BB}_t \mathrm{d}t \geq 0\right]=:\eta.
\end{equation}
Gathering \eqref{eq: first bound a(S) in alpha N} and \eqref{eq: bound a(S) positive}, we get 
\begin{equation}
    \psi(T_{x,y}) \leq 2\eta^{-1}\mathsf{P}\Big[\int_{0}^1 \mathsf{BB}_t \mathrm{d}t \geq c_0(\varepsilon)^{-1}C^{-3/2}\alpha\Big].
\end{equation}
Using the following estimate on the tail area of a Brownian bridge (see~\cite[Theorem 1.2]{tailsBrownian}), we obtain $c_1=c_1(\varepsilon),C_1=C_1(\varepsilon)>0$ such that
\begin{equation}\label{eq: tail area bb}
    \mathsf{P}\Big[\int_{0}^1 \mathsf{BB}_t \mathrm{d}t \geq c_0(\varepsilon)^{-1}C^{-3/2}\alpha\Big]\leq C_1\e^{-c_1\alpha^2}.
\end{equation}
Using \eqref{eq: mvmp 2}, \eqref{eq: bound a(S) positive}, and \eqref{eq: tail area bb}, we get that for $N$ large enough,
\begin{eqnarray*}
    \PlN\left[A_{x,y}\right]&\leq& 2C_1\eta^{-1}\e^{-c_1\alpha^2} \lambda^{-\lceil\frac{\alpha+1}{K_2}\rceil} \PlN[B_{x,y}] \\
    &\leq& C_2\e^{-c_2\alpha^2}\PlN[B_{x,y}].
\end{eqnarray*}
To obtain the lemma, it remains to sum over $\alpha \geq \beta$ and over all possible values for $x$ and $y$,
\begin{equation}\label{equation finale precise tails area}
\PlN[\mathcal{A}_N(\Gamma)\geq \beta N, \:  \G\subset B_{K_1N}\setminus B_{K_2N}, \: \G\notin \mathsf{Bad}^{+,-}_{\varepsilon,\bfx_N(\G),\bfy_N(\G)} ] \leq C_2\sum_{\alpha = \lfloor \beta \rfloor}^{+\infty}  \exp(-c_2\alpha^2).
\end{equation}
Then, for some $C_3>0$,
\begin{equation}
    \PlN[ \mathcal{A}_N(\Gamma) \geq \beta N, \: \G\subset B_{K_1N}\setminus B_{K_2N}, \: \G\notin \mathsf{Bad}^{+,-}_{\varepsilon,\bfx_N(\G),\bfy_N(\G)} ] \leq C_3\e^{-c_2\beta^2}.
\end{equation}
Using Lemma~\ref{lem: confinement lemma} and Proposition~\ref{prop: bad angle}, there exists $C_4>0$ such that for $N$ large enough (namely $N\geq \beta^{3}$), 
\begin{equation}
    \PlN\left[ \mathcal{A}_N(\Gamma) \geq \beta N\right]\leq C_4\e^{-c_2\beta^2}.
\end{equation}
\end{proof}

\begin{Def}
If $\gamma_{u,v} \in \Lambda^{u \rightarrow v}$ and $z \in \gamma_{u,v}$, we define $\LR(z)(\gamma_{u,v})$, \textit{the local roughness of the vertex} $z \in \gamma_{u,v}$, to be the local roughness of $z$ in the path $\gamma_{u,v}$ extended to the left with only horizontal steps to up to the $y$-axis and to the right with only vertical steps up to the $x$-axis.
We define the mean facet length of $\gamma_{u,v}$ similarly.

\end{Def}
The two next lemmas study the typical behaviours of the mean facet length and of the local roughness of a uniform path between $\bfx_N(\Gamma)$ and $\bfy_N(\Gamma)$ when $\Gamma$ is sampled according to $\PlN$. For simplicity, in the rest of the section, when $\Gamma$ is a sample of $\PlN$, we shall abbreviate $\bfx_N(\Gamma)$ (resp. $\bfy_N(\Gamma)$) by $\bfx_N$ (resp. $\bfy_N$). The proofs of these lemmas rely on Gaussian computations and are not specific to the considered model. They are deferred to the end of the section.

We will write $\mathcal{G}_N:=\lbrace \Gamma \in B_{K_1N}\setminus B_{K_2N} \rbrace\cap (\mathsf{Bad}^{+,-}_{\varepsilon,\bfx_N,\bfy_N})^c$.

\begin{Lemma}\label{lemme mean lr for a Brownian process}
There exists a function $\Psi$ independent of every parameter of the problem with $\Psi(t)
\goes{}{t\rightarrow 0^+}{0}$ such that, for $t>0$, if $N$ is large enough,
\begin{equation}\label{equation mean lr for a Brownian process}
    \E_\lambda^{N^2}\big[\1_{\mathcal{G}_N} \PP_{\bfx_N,\bfy_N}[ \MiLR(\Gamma_{\bfx_N,\bfy_N}) < tN^{1/3}] \big] \leq \Psi(t).
\end{equation}
\end{Lemma}
The corresponding statement for the mean facet length is the following 
\begin{Lemma}\label{lemme mean fl for a Brownian process}
There exists a function $\Phi$ independent of every parameter of the problem with $\Phi(t) \goes{}{t\rightarrow 0^+}{0}$ such that, for $t>0$, if $N$ is large enough,
\begin{equation}\label{equation mean fl for a Brownian process}
    \E_\lambda^{N^2}\big[ \1_{\mathcal{G}_N}\PP_{\bfx_N,\bfy_N}[ \MiFL(\Gamma_{\bfx_N,\bfy_N}) < tN^{2/3}] \big] \leq \Phi(t).
\end{equation}
\end{Lemma}

These two lemmas in hand, we can now prove the lower bound of Proposition~\ref{prop lower bound meanlr}. Proposition~\ref{prop lower bound meanfl} will be proved by the same method. 

\begin{proof}[Proof of Proposition \textup{\ref{prop lower bound meanlr}}] Fix some $t> 0$, and some $\beta > 0$ that is going to be chosen later as a function of $t$. We implement the following resampling procedure: let us start with a sample $\Gamma$ of $\PlN$. As usual we condition on the configuration outside of the cone $\mathbf{A}_{\bfx_N, \bfy_N}$ that we name $\Gamma_{\mathsf{ext}} := \Gamma\setminus \Gamma_{\bfx_N, \bfy_N}$. By the Brownian Gibbs property, the distribution of $\Gamma_{\bfx_N,\bfy_N}$ is uniform among the oriented paths linking $\bfx_N$ and $\bfy_N$ that enclose at least a certain amount of area which is measurable with respect to the exterior configuration. Let us call $A_\mathsf{ext}$ the area already enclosed by $\Gamma_{\mathsf{ext}}$ (namely, $A_\mathsf{ext} = \textup{Enclose}(\Gamma \cap \mathbf{A}_{\bfx_N,\bfy_N}^c \cap \N^2)$), so that the conditional distribution of $\Gamma_{\bfx_N,\bfy_N}$ is nothing but $\PP_{\bfx_N,\bfy_N}\big[\:\cdot\:\vert\:\mathcal{A}_N(\Gamma_{\bfx_N,\bfy_N})\geq N^2 - A_\mathsf{ext} \big]$.

Before starting, let us point out a crucial fact. With our definition of $\LR(z)(\Gamma_{u,v})$, it is a deterministic fact that for any $u,v \in \Gamma, z \in \Gamma_{u,v}$,
\begin{equation}
    \LR(z)\left(\Gamma_{u,v}\right) \leq \LR(z)\left(\Gamma\right) ~~\text{ and thus }~~ \MiLR(\Gamma_{\bfx_N,\bfy_N}) \leq \MiLR(\Gamma).
\end{equation}
Thus, we have
\begin{multline*}
\PlN\big[\MiLR(\G) < tN^{1/3}\big] \\ \leq \E_\lambda^{N^2}\big[ \PP_{\bfx_N,\bfy_N}\big[\MiLR(\Gamma_{\bfx_N,\bfy_N}) <tN^{1/3}~\vert~ \mathcal{A}_N(\Gamma_{\bfx_N,\bfy_N})  \geq N^2-A_\mathsf{ext}\big]\big].
\end{multline*}
Introduce
\begin{equation}
    \mathcal{F}_N:=\lbrace N^2-A_\mathsf{ext}\leq \beta N\rbrace \cap \lbrace \G \subset B_{K_1N}\setminus B_{K_2N}\rbrace\cap (\mathsf{Bad}^{+,-}_{\varepsilon,\bfx_N,\bfy_N})^c, 
\end{equation}
where $\mathsf{Bad}^{+,-}_{\varepsilon,\bfx_N,\bfy_N}$ was introduced in the proof of Lemma~\ref{lemma aire secteur fixé pas trop grande}. Note that by Lemma~\ref{lem: confinement lemma}, Proposition~\ref{prop: bad angle} and Lemma~\ref{lemma aire secteur fixé pas trop grande}, one has that for $N\geq \beta^3$,
\begin{equation}
    \PlN[\mathcal{F}_N^c]\leq Ce^{-c\beta^2}.
\end{equation}
Hence,
\begin{multline*}
\E_\lambda^{N^2}\big[ \PP_{\bfx_N,\bfy_N}[\MiLR(\Gamma_{\bfx_N,\bfy_N}) <tN^{1/3}\:\vert\: \mathcal{A}_N(\Gamma_{\bfx_N,\bfy_N})  \geq N^2-A_\mathsf{ext}\big]
\big] 
\\ \leq \E_\lambda^{N^2}\big[ \PP_{\bfx_N,\bfy_N}\big[\MiLR(\Gamma_{\bfx_N,\bfy_N}) <tN^{1/3}\:\vert \:\mathcal{A}_N(\Gamma_{\bfx_N,\bfy_N})  \geq N^2-A_\mathsf{ext}\big]\1_{
\mathcal{F}_N}\big]
 \\+ Ce^{-c\beta^2}.
 \end{multline*}
Now, note that
\begin{multline*}
    \1_{\mathcal{F}_N}\PP_{\bfx_N,\bfy_N} \big[  \MiLR(\Gamma_{\bfx_N,\bfy_N}) < tN^{1/3} \:\vert \: \mathcal{A}_N(\Gamma_{\bfx_N,\bfy_N})  \geq N^2-A_{\mathsf{ext}}\big] \\\leq \1_{\mathcal{F}_N}\frac{\PP_{\bfx_N,\bfy_N}[  \MiLR(\Gamma_{\bfx_N,\bfy_N}) <tN^{1/3} ]}{\PP_{\bfx_N,\bfy_N}[ \mathcal{A}_N(\Gamma_{\bfx_N,\bfy_N})  \geq \beta N ]}.
\end{multline*}
Using the methods developed above\footnote{This is again an application of Donsker's Theorem and the use of the asymptotic results of~\cite{tailsBrownian}, we do not write the full proof but rather refer to the proof of Lemmas~\ref{lemme mean lr for a Brownian process} and~\ref{lemme mean fl for a Brownian process} for details.} we could also show that for some $c',C'>0$, if $N$ is large enough,
\begin{equation}
    \PP_{\bfx_N,\bfy_N}\left[ \mathcal{A}_N(\Gamma_{\bfx_N,\bfy_N})  \geq \beta N \right]\geq C'\e^{-c'\beta^2}.
\end{equation}
By Lemma~\ref{lemme mean lr for a Brownian process}, we obtain that, integrating over $\Gamma_\mathsf{ext}$, and putting everything together,
\begin{equation}
    \PlN[ \MiLR(\G) < tN^{1/3} ] \leq  C\e^{-c\beta^2} + \frac{\Psi(t)}{C'\e^{-c'\beta^2}}.
\end{equation}
It remains to carefully chose $\beta$ in terms of $t$. We optimize the above equation by setting
\begin{equation}
    \beta(t) = \sqrt{(c+c')^{-1}\log(\tfrac{c}{c'}\tfrac{CC'}{\Psi(t)})}.
\end{equation}
Observe that $\beta(t) \goes{}{t \rightarrow 0}{+\infty}$. Choosing $N$ large enough, we obtain that for some $C''>0$,
\begin{equation}
    \PlN[\MiLR(\G) < tN^{1/3}] \leq C''(\Psi(t))^{\frac{c}{c+c'}}
\end{equation}
which is the announced result.
\end{proof}
Proposition~\ref{prop lower bound meanfl} is a consequence of the exact same argument. 
\begin{proof}[Proof of Proposition~\textup{\ref{prop lower bound meanfl}}] We proceed exactly as in the proof of Proposition~\ref{prop lower bound meanlr}. Indeed, observe that as previously,
\begin{equation}
\MiFL(\Gamma_{\bfx_N, \bfy_N}) \leq \MiFL(\Gamma). 
\end{equation}
Hence, the proof can be reproduced \textit{mutatis mutandi} using this time Lemma~\ref{lemme mean fl for a Brownian process}.
\end{proof}
We now turn to the proofs of the two Gaussian lemmas, namely Lemma~\ref{lemme mean lr for a Brownian process} and Lemma~\ref{lemme mean fl for a Brownian process}. The process of the concave majorant of a Brownian bridge has been extensively studied in~\cite{groeneboomconcavemajorant, suidanconvexminorantsofrandomwalks, balabdaouipitman}. Let us briefly summarise their results.
\begin{itemize}
    \item Groeneboom showed in~\cite{groeneboomconcavemajorant} that conditionally on the convex majorant of a Brownian motion, the difference process was a succession of independent Brownian excursions between the extremal points of the concave majorant. The result was later extended to the standard Brownian bridge in~\cite{balabdaouipitman} via a Doob tranformation. 
    \item Groeneboom also gave an explicit representation of the distribution of the process of the slopes of the concave majorant. Later, Suidan ~\cite{suidanconvexminorantsofrandomwalks} was able to derive the joint law of the ordered lengths of the segments of $[0,1]$ in the partition of $[0,1]$ induced by the extremal points of the concave majorant of a Brownian bridge.
\end{itemize}
We refer to the introduction of~\cite{balabdaouipitman} for a clear and detailed presentation of these results. In particular, we import the following statement from the three cited articles.

\begin{Th}\label{theorem input pitman}
    Let $\left(\mathsf{BB}_t\right)_{0\leq t \leq 1}$ be the standard Brownian bridge and let $\left(\mathcal{C}(t)\right)_{0\leq t \leq 1}$ be its smallest concave majorant. Let $s \in (0,1)$. Let us define the two following random variables:
    \begin{itemize}
        \item $L(s)$ is the length of the almost surely unique facet of $\mathcal{C}$ intersecting the line $\lbrace x=s \rbrace$,
        \item $R(s)$ is the difference process defined by $R(s):=\mathcal{C}(s) - B(s)$.
    \end{itemize}
Fix $\epsilon> 0$. Then, one has
\begin{equation}
    \sup_{s \in [\epsilon, 1-\epsilon]} \PP \left[L(s) < r\right] \goes{}{r\rightarrow 0^+}{0}, 
\end{equation}
and
\begin{equation}
    \sup_{s \in [\epsilon, 1-\epsilon]} \PP \left[R(s) < r\right] \goes{}{r\rightarrow 0^+}{0}.
\end{equation}
\end{Th}
We briefly sketch the proof of Theorem~\ref{theorem input pitman} and refer to~\cite{groeneboomconcavemajorant, suidanconvexminorantsofrandomwalks, balabdaouipitman} for more details. 
\begin{proof}[Sketch of proof of Theorem \textup{\ref{theorem input pitman}}]
Fix some $s \in [\epsilon, 1-\epsilon]$. For $L(s)$, we remark that the results of~\cite{groeneboomconcavemajorant} imply that $\mathcal{C}$ almost surely has a finite number of slope changes (or extremal points) in $[\epsilon, 1-\epsilon]$, and~\cite{balabdaouipitman} gives an explicit description of their lengths in terms of a size-biased uniform stick-breaking process. In particular, it implies that the distribution of $L(s)$ is absolutely continuous with respect to the Lebesgue measure on $\R^+$, which implies that $\PP\left[L(s)<r\right]\goes{}{r\rightarrow 0}{0}$. The statement follows by the observation that $s \in (0,1) \mapsto \PP\left[L(s)<r\right]$ is continuous. 

For $R(s)$, let us denote by $a\leq s\leq b$ the horizontal coordinates of the endpoints of the facet intersecting the line $\lbrace x=s \rbrace$. The exact same argument as above implies that $\PP\left[\min(s -a, b-s)< \eps \right]$ goes to 0 when $\eps$ goes to 0. On the complementary of the latter event, by~\cite{groeneboomconcavemajorant}, $R(s)$ is the height of a Brownian excursion at a positive distance of its endpoints, and is indeed continuous with respect to the Lebesgue measure on $\R^+$. We conclude as previously, observing that $s \in (0,1) \mapsto \PP\left[R(s)<r\right]$ is continuous.
\end{proof}

Now observe that in Lemmas~\ref{lemme mean lr for a Brownian process} and~\ref{lemme mean fl for a Brownian process}, $\Vert \bfy_N - \bfx_N \Vert$ is of order $N^{2/3}$ so that the scaling is Gaussian. Hence, Donsker's Theorem suggests that we may use the above-mentioned results to obtain lower tails for the random variables $N^{-1/3}\MiLR$ and $N^{-2/3}\MiFL$.

\begin{proof}[Proof of Lemma \textup{\ref{lemme mean lr for a Brownian process}}]
As in the proof of Lemma~\ref{lemma aire secteur fixé pas trop grande}, we write $h_N = \Vert \bfy_N-\bfx_N\Vert$. Moreover, let us call $\bft_N^0$ the intersection between the line segment $[\bfx_N,\bfy_N]$ and the line $\lbrace y=x \rbrace$. Remember that we work under the event $\lbrace \Gamma \in B_{K_1N}\setminus B_{K_2N} \rbrace\cap (\mathsf{Bad}^{+,-}_{\varepsilon,\bfx_N,\bfy_N})^c$, which allows us to reuse Claim~\ref{claim 3}, namely
\begin{equation}\label{eqution majoration hn 2}
   c N^{2/3} \leq h_N \leq C N^{2/3}.
\end{equation}
Moreover, there exists a constant $c_\eps> 0$ such that
\begin{equation}\label{equation controle t0N}
    \min(\Vert \bft^0_N - \bfx_N \Vert_2, \Vert \bfy_N - \bft^0_N \Vert) \geq c_\eps h_N.
\end{equation}
Recall the parametrisation of the probability measures $\PP^{\theta}_{h_N} = \PP_{\bfx_N, \bfy_N}$ introduced in the proof of Lemma~\ref{lemma aire secteur fixé pas trop grande}. Similarly as before,
\begin{eqnarray*}
   && \hspace{-100pt}\E_\lambda^{N^2}\big[\1_{\mathcal{G}_N} \PP_{\bfx_N,\bfy_N}\big[\MiLR(\Gamma_{\bfx_N,\bfy_N} )<tN^{1/3}\big] \big] 
   \\ &=& 
   \E_{\lambda}^{N^2}\big[\1_{\mathcal{G}_N}\PP^\theta_{h_N}\big[h_N^{-1/2}\LR(\bft_N^0) < th_N^{-1/2}N^{1/3} \big] \big] 
   \\&\leq& 
   \E_{\lambda}^{N^2}\big[\1_{\mathcal{G}_N}\PP^\theta_{h_N}\big[ h_N^{-1/2}\LR(\bft_N^0)< c^{-1/2}t \big] \big] \\&\leq& \E_{\lambda}^{N^2}\big[\1_{\mathcal{G}_N}\sup_{\vartheta\in[\eps, \pi/2-\eps ]}\PP^\vartheta_{h_N}[h_N^{-1/2}\LR(\bft_N^0) < c^{-1/2}t \big] \big],
   \end{eqnarray*}
where $c$ is the constant appearing in \eqref{eqution majoration hn 2}. Now, observe that $\bft_0^N$ is independent of the random path sampled according to $\PP^\vartheta_{h_N}$, and we can write 
\begin{equation}
    \PP^\vartheta_{h_N}\big[h_N^{-1/2}\LR(\bft^0_N)<c^{-1/2}t\big] \leq \sup_{s\in [c_\eps, 1-c_\eps]} \PP^\vartheta_{h_N}\big[h_N^{-1/2}\LR(h_N s)<c^{-1/2}t\big],
\end{equation}
Thus, we obtained that 
\begin{multline}\label{eq: domination mean LR brownian bridge}
    \E_\lambda^{N^2}\big[\1_{\mathcal{G}_N} \PP_{\bfx_N,\bfy_N}\big[\MiLR(\Gamma_{\bfx_N,\bfy_N} )<tN^{1/3}\big] \big]
    \\
\leq\E_{\lambda}^{N^2}\big[\1_{\mathcal{G}_N}\sup_{\vartheta\in[\eps, \pi/2-\eps ]}\sup_{s \in [c_\eps, 1-c_\eps]} \PP^\vartheta_{h_N}\big[h_N^{-1/2}\LR(s h_N)< c^{-1/2}t\big]\big],
\end{multline}
where $c_\eps$ is the constant appearing in \eqref{equation controle t0N}. Reasoning as above, one has the following convergence in distribution (the random variable $R$ was introduced in Theorem~\ref{theorem input pitman}), under $\mathbb P^\vartheta_{h_N}$ and assuming $\mathcal{G}_N$ occurs, for $s\in [c_\varepsilon,1-c_\varepsilon]$,
\begin{equation}
    \lim_{N\rightarrow \infty} h_N^{-1/2}\mathsf{LR}(sh_N)=\gamma(\vartheta)R(s),
\end{equation}
where $\gamma(\vartheta)=\sigma(\vartheta)\sqrt{2^{1/2}\cos(\vartheta-\pi/4)}$ and $\sigma(\vartheta)$ was defined below \eqref{eq: definition f_N}. Once again, we claim that this convergence is uniform in $\vartheta \in [\eps, \pi/2- \eps]$ (see Remark~\ref{rem: uniform cv in the angle}). Hence, the limsup of the right-hand side of~\eqref{eq: domination mean LR brownian bridge} is bounded by
\begin{equation}
    \sup_{s\in [c_\eps, 1-c_\eps]}\mathsf{P}[R(s) < tc^{-1/2}\gamma(\varepsilon)^{-1}].
\end{equation}
We argue as in the proof of Lemma~\ref{lemma aire secteur fixé pas trop grande} to get for $N$ large enough
\begin{eqnarray*}
    \E_\lambda^{N^2}\big[ \1_{\mathcal{G}_N}\PP_{\bfx_N,\bfy_N}\big[\MiLR(\Gamma_{\bfx_N,\bfy_N} )<tN^{1/3}\big] \big] &\leq& C_1\sup_{s \in [c_\eps, 1-c_\eps]} \mathsf{P}\big[R(s) < tc^{-1/2}\gamma(\varepsilon)^{-1}\big] \\ &=:& \Psi(t).
\end{eqnarray*}
The fact that $\Psi$ goes to 0 is a consequence of Theorem~\ref{theorem input pitman}.
\end{proof}
We turn to the proof of Lemma~\ref{lemme mean fl for a Brownian process}.
\begin{proof}[Proof of Lemma~\textup{\ref{lemme mean fl for a Brownian process}}]
The proof follows the exact same strategy as above. Indeed, keeping the notations of the proof of Lemma~\ref{lemme mean lr for a Brownian process}, we may write
\begin{multline}
    \E_\lambda^{N^2}\big[\1_{\mathcal{G}_N} \PP_{\bfx_N,\bfy_N}\big[ \MiFL(\Gamma_{\bfx_N,\bfy_N}) < tN^{2/3}\big] \big]  \\ \leq   \E_{\lambda}^{N^2}\big[\1_{\mathcal{G}_N}\sup_{\vartheta\in[\eps, \pi/2-\eps ]}\PP^\vartheta_{h_N}\big[h_N^{-1}\MiFL(\Gamma_{\bfx_N, \bfy_N})< tc^{-1}\big] \big].
\end{multline}
We are going to use the same method as in the proof of Lemma~\ref{lemme mean lr for a Brownian process}. We refer to~\cite[Theorems 4 and 5]{suidanconvexminorantsofrandomwalks} to observe that the facet length scales as the time $h_N$ in a random walk bridge. Some extra care is needed due to the fact that the function $\phi\in\mathcal{C}([0,1], \R) \mapsto \MiFL(\phi) \in \R^+$ is not continuous in $\phi$; however it is easy to check that its points of discontinuity are included in the set of continuous functions having at least three aligned local maxima. Using the fact that the set of such functions has measure 0 under $\mathsf{P}$ and the above-mentioned uniformity in $\vartheta$, we obtain from arguments similar to that of proof of Lemma~\ref{lemme mean lr for a Brownian process} that there exists a constant $\delta(\eps) > 0$ such that, for $N$ large enough,

\begin{eqnarray*}
    \E_{\lambda}^{N^2}\big[\1_{\mathcal{G}_N}\PP_{\bfx_N,\bfy_N}\big[\MiFL(\Gamma_{\bfx_N,\bfy_N}) < tN^{2/3}\big]\big] &\leq& C_1'\sup_{s\in[c_\eps, 1-c_\eps]}\mathsf{P}[L(s)<t\delta(\eps)c^{-1}]
    \\ &=:& \Phi\left(t\right).
\end{eqnarray*}
We conclude using Theorem~\ref{theorem input pitman}.
\end{proof}

\section{Analysis of the maximal facet length and of the maximal local roughness}\label{section analysis of the maximal facet length and of the maximal local roughness}

This section is devoted to the proof of Theorem~\ref{theorem maxfl and maxlr}. The corresponding statements in the context of the random-cluster have been obtained by Hammond in~\cite{alan1, alan2} (see Section~\ref{section extension to other models} for a detailed discussion). We will essentially adapt the robust arguments developed by Hammond in these papers although some new difficulties, coming from the oriented feature of the model, will emerge.

The proof of the upper bound of~\eqref{maxfl.thm} is in spirit the same as the one of Proposition~\ref{prop: rough upper bounds} observing that the proof can be reproduced (with care) replacing the $N^\epsilon$ by $C(\log N)^{1/3}$ with a large $C$.  


The proof of the lower bound is more technical and follows the lines of~\cite{alan2}. We use a different strategy of resampling consisting in successively resampling $\Gamma$ in deterministic angular sectors. We then show that each one of these resampling has a sufficiently large probability of producing a favourable output (meaning with a large local roughness) --- and thus that the final output has a large local roughness with very high probability. However, we will encounter several technical issues, the main one being the fact that a favourable ouptut could be destroyed by further resamplings. This technical point will be solved by carefully analysing the Markovian dynamic induced by the successive resamplings in Section~\ref{section lower bounds}. The bound on the maximal facet length will be a simple byproduct of the bound on the maximal local roughness. 

\subsection{Upper bounds}\label{section upper bounds}

We start by proving the upper bound on $\MFL$. The correct control is given by the following proposition.
\begin{Prop} \label{mfl.control.prop}
There exist constants  $\tilde{c},c,C>0$ such that, for any $\tilde{c} \leq t \leq N^{2/3} (\log N)^{-2/3}$,
\begin{equation}
    \PlN \big[ \MFL(\G) \geq tN^{\frac{2}{3}}(\log N)^{\frac{1}{3}} \big] \leq C e^{- ct^{3/2} \log N}.
\end{equation}
\end{Prop}
The proof of this proposition follows the same lines as the proof of Proposition~\ref{prop: rough upper bounds} with the exception that we modify the event $\GAC{x}{y}{\eta}$ of Subsection~\ref{subsection upper bounds} to require a poly-logarithmic deviation of the area from its typical value.
\begin{Def}[Logarithmic good area capture]\label{definition log good area capture}
Let $\g \in \La$ be a path, $\eta > 0$ and $x,y \in \N^2$. We say that $\g$ realizes the event $\mathsf{LogGAC}(x,y,\eta)$ (meaning ``logarithmic good area capture'') if 
\begin{enumerate}
    \item[$(i)$] $\g\cap \mathbf{A}_{x,y} \in \La^{x\rightarrow y}$, or in words, $\g$ connects $x$ and $y$ by a oriented path,
    \item[$(ii)$] $|\Enc(\g \cap \mathbf{A}_{x,y})| - |\mathbf{T}_{0,x,y}| \geq \eta \dist x - y \dist^{\frac{3}{2}} \left( \log \dist x-y \dist \right)^\frac{1}{2}$.
\end{enumerate}
\end{Def}
As previously, we need to estimate the probability of $\mathsf{LogGAC}(x,y,\eta)$ under $\PP_{x,y}$, the uniform measure on $\Lambda^{x\rightarrow y}$. The following result, whose proof is postponed to the Appendix~\ref{appendix computations on SRW}, gives us a lower bound on $\PP_{x,y}[\mathsf{LogGAC}(x,y,\eta)]$. Recall that $\theta(x,y)\in [0,\pi/2]$ the angle formed by the horizontal axis and the segment joining $x$ and $y$.
\begin{Lemma}\label{lemme log good shape uniform} Let $\eps, \eta > 0$. There exist $C=C(\varepsilon) > 0$ and $N_0=N_0(\varepsilon,\eta) > 0$ such that for any $x,y \in \N^2$ satisfying $\Lambda^{x\rightarrow y}\neq \emptyset$, $\varepsilon\leq \theta(x,y)\leq \pi/2-\varepsilon$ and $\Vert x-y\Vert\geq N_0$,
\begin{equation}
    \PP_{x,y} \left[\mathsf{LogGAC}(x,y,\eta)\right] \geq \Vert x-y \Vert^{- C\eta^2}.
\end{equation} 
\end{Lemma}
Proposition~\ref{prop: bad angle} will ensure that the condition on $\theta(x,y)$ holds with high probability.
\begin{proof}[Proof of Proposition~\textup{\ref{mfl.control.prop}}]
We repeat the same strategy as in the proof of Proposition~\ref{prop: rough upper bounds} considering this time the event
\begin{equation}
    \BF(t):=\big\{ \mathsf{MaxFL}(\G)\geq tN^{\frac{2}{3}}(\log N)^{\frac{1}{3}}\big\}
\end{equation}
instead of $\BiF(t)$, $\mathsf{LogGAC}$ instead of $\mathsf{GAC}$, and replacing $\GH$ by the event that the random pair $(\bfx,\bfy)$ hits the extremities of the largest facet. The control on the probability of $\mathsf{LogGAC}$ is given by Lemma~\ref{lemme log good shape uniform} above.
\end{proof}

With this result, it is now easy to estimate the correct scale for $\MLR$. The statement is given in the following proposition. 

\begin{Prop} \label{mlr.conrol.prop}
There exist constants $\tilde{c},c,C>0$ such that, for any $\displaystyle \tilde{c} \leq t\leq N^{5/6} (\log N)^{-5/6}$, then,
\begin{equation}
    \PlN \big[\MLR(\G) \geq tN^\frac{1}{3}(\log N)^\frac{2}{3}\big] \leq Ce^{-ct^{6/5} \log N}.
\end{equation}
\end{Prop}
\begin{proof}
Denote by $\MaxFac(\G)$ the facet where $\MLR(\G)$ is attained, and by $\MLRF(\G)$ its length. It is clear that 
\begin{equation}
    \MLRF(\G) \leq \MFL(\G).
\end{equation}
With this remark and Proposition~\ref{mfl.control.prop}, we can conclude using the exact same method as in the proof of Proposition~\ref{prop meanlr}, the only difference being that we now use the coupling along the facet which has the largest local roughness. More precisely, let $\delta >0$ to be fixed later. Let $t>0$, then
\begin{multline}\label{eq:1}
\big\lbrace \MLR(\G)\geq t N^{\frac{1}{3}}(\log N)^{\frac{2}{3}}\big\rbrace 
\\ \subset 
\big\lbrace \MLRF(\G)\geq t^{2-\delta}N^{\frac{2}{3}}(\log N)^{\frac{1}{3}}\big\rbrace \cup A   
\cup \big\lbrace \Gamma \subset (B_{K_1N})^c\big\rbrace \cup \bigcup_{j=N^{\frac{1}{3}}}^{t^{2-\delta}N^{\frac{2}{3}}(\log N)^{\frac{1}{3}}} A_j,
\end{multline}
where 
\begin{equation}
    A_j=\big\lbrace \MLR(\G)\geq t N^{\frac{1}{3}}(\log N)^{\frac{2}{3}}  ,~ \MLRF(\G)=j , ~ \G \subset B_{K_1N}\big\rbrace,
\end{equation}
\begin{equation}
    A=\big\lbrace \MLR(\G)\geq t N^{\frac{1}{3}}(\log N)^{\frac{2}{3}}, ~ \MLRF(\G) < N^{\frac{1}{3}}\big\rbrace,
\end{equation}
and $K_1$ is defined as in Lemma~\ref{lem: confinement lemma}. It is easy to check that $A=\emptyset$ for $t\geq 1$, which is now our assumption. Thanks to Lemma~\ref{lem: confinement lemma} and Proposition~\ref{mfl.control.prop}, we already know how to control the probability of the first three terms in the right-hand side of \eqref{eq:1}. The proof will follow from obtaining an upper bound on $\PlN[A_j]$. This bound is (again) given by the Lemma~\ref{Lemme domination roughness random walk sous une facette} (see Remark~\ref{remarque domination roughness avec le log}).  

Let us denote by $B(j,x)$ the event that $x$ belong to $\G$ and the facet above $x$ has length $j$. Then, for $N^{1/3}\leq j \leq t^{2-\delta}N^{2/3}(\log N)^{1/3}$,
\begin{eqnarray*}
  \PlN[A_j] &\leq& \sum_{x \in B_{K_1N}} \PlN\big[\LR(x) \geq tN^{\frac{1}{3}}(\log N)^{\frac{2}{3}},\:  B(j,x) \big] 
  \\
    &\leq& \sum_{x \in B_{K_1N}}\sum_{\substack{a,b \in B_{K_1N} \\ \Vert b-a \Vert = j \\ \arg(a)>\arg(x)\geq \arg(b)}}\PlN\big[ \LR(x) \geq tN^{\frac{1}{3}}(\log N)^{\frac{2}{3}}~\big\vert~\mathsf{Fac}(a,b)\big]\PlN[\mathsf{Fac}(a,b)] 
    \\
    &\leq& O(N^4j)\exp\Big( - c\frac{t^2N^{2/3}(\log N)^{4/3}}{2j} \Big) \\
    &\leq& \e^{-ct^{\delta}\log N}
\end{eqnarray*}
where the third inequality is a consequence of Lemma~\ref{Lemme domination roughness random walk sous une facette} (and more particularly Remark~\ref{remarque domination roughness avec le log}). The proof follows from setting $\delta=\frac{6}{5}$ and taking $ t_0 \leq t \leq N^{5/6}(\log N)^{-5/6}$, where $t_0$ is large enough.
\end{proof}

\subsection{Lower bounds}\label{section lower bounds}

We start with the proof of the lower bound for $\MLR$ and then see how we can deduce the lower bound for $\MFL$.

\subsubsection{The strategy}  We rely again on the strategy of resampling. We will divide the first quadrant of $\mathbb Z^2$ in angular sectors of angle
\begin{equation}\label{def theta n}
    \theta_N=\theta_N(\chi):= \chi N^{-\frac{1}{3}}(\log N)^{\frac{1}{3}},
\end{equation}
where  $\chi>0$ is a small constant that we will specify later. Then, we will randomly process a resampling in some of these sectors. We will argue that each one of these sectors has a probability of producing a ``favourable'' output (i.e. with a large local roughness) which is decaying slowly (see Lemma~\ref{first lemma good shape}) in $N$, i.e. roughly like $N^{-\delta}$ for some (small) $\delta > 0$. By iterating these resamplings on (sufficiently many) different sectors we ensure to recover a new path which presents at least one sector with a large local roughness.
There is a crucial point we will have to deal with: if a given sector produces a favorable output, this could be undone by a further resample. That is why we will resample on sufficiently distant sectors (see Proposition~\ref{main result resample successifs}).

Let us first fix some notations. Define 
\begin{equation}
m_N:=\left\lfloor \dfrac{\pi}{2\theta_N}\right\rfloor
\end{equation}
to be the number of sectors subject to resampling. Note that 
\begin{equation}\label{eq: asymptotics mn}
m_N\underset{N\rightarrow +\infty}\sim \dfrac{\pi}{2\chi}N^{\frac{1}{3}}(\log N)^{-\frac{1}{3}}.
\end{equation}
Denote by $\mathbf{A}_j(N)$ the sector of the first quadrant defined by the angles $(j-1)\theta_N$ and $j\theta_N$ for $j \in \lbrace 1, \ldots, m_N\rbrace$. Note that the first quadrant may contain a narrower sector which we will not make use of. Finally, we introduce the following notation: given two points $a$ and $b$ in the first quadrant, we denote by $\angle (a,b)$ the (non negative\footnote{We will always choose $\angle(a,b)\in [0,\pi/2]$.}) angle between the associated vectors $\overrightarrow{a}$ and $\overrightarrow{b}$\footnote{Sometimes we will also write $\angle (\overrightarrow{a}, \overrightarrow{b})$, depending on the context.}. For a closed shape $S$, $\Enc(S)$ is the region of the first quadrant enclosed by $S$ (and if needed the coordinate axes).

\subsubsection{Defining the resampling procedure}  We want to resample in the sectors $\mathbf{A}_j(N)$. However, choosing the beginning and ending points of the pieces subject to resampling in each sector must be done precociously.

\begin{Def}
Let $x,y\in \mathbb N^2$ be such that $\La^{x\rightarrow y} \neq \emptyset$. Denote by $\mathbf{A}_{x,y}$ the cone of apex $0$ bounded by $x$ and $y$. For all $\g \in \La^{N^2}$, $\Psi_{x,y}(\g)$ is the random element of $\La$ that is equal to $\g$ outside $\mathbf{A}_{x,y}$, and that is random on $\mathbf{A}_{x,y}$ having the marginal law of $\PlN$ on the considered sector, given the area condition and the path outside the sector.
\end{Def}

\begin{Rem}
In words, the formation of $\Psi_{x,y}(\g)$ can be summed up in two steps:
\begin{enumerate}
    \item[-] Step A: we condition on the marginal of $\g$ on $\mathbf{A}_{x,y}$ by the information of $\g \cap \mathbf{A}_{x,y}^c$.
    \item[-] Step B: we condition on the area: 
    $\vert\Enc((\g \cap \mathbf{A}_{x,y}^c)\cup \g_{x,y})\vert\geq N^2.$
\end{enumerate}
\end{Rem}

The strategy is now the following: for each sector $\mathbf{A}_j(N)$, we pick some points $x_j,y_j$ that belong to this sector, and to the random path we are considering, and we apply $\Psi_{x_j,y_j}$. 

\begin{Def}[Definition of $\mathsf{RES}_j$]\label{def procedure}
Let $j \in \lbrace 1,\dots, m_N \rbrace$. We define a procedure $\mathsf{RES}_j: \La^{N^2}\rightarrow \La$ that only acts on $\mathbf{A}_j(N)$. Let $\mathbf{B}_j(N)$ be the cone of apex $0$ defined by the angles $(j-1)\theta_N+\theta_N/4$ and $j\theta_N-\theta_N/4$. Notice that $\mathbf{B}_j(N)\subset \mathbf{A}_j(N)$. Let $\g \in \La^{N^2}$. Let $x_j=x_j(\g)$ (resp. $y_j=y_j(\g)$) be the left-most (resp. right-most) point of $\g\cap \mathbf{B}_j(N)$. The procedure $\mathsf{RES}_j$ consists in resampling $\g$ between $x_j$ and $y_j$. More precisely, we set $\mathsf{RES}_j(\g):=\Psi_{x_j,y_j}(\g)$.
\end{Def}
We will denote by $(\Omega, \mathcal{F}, \mathbb P)$ the probability space in which $\mathsf{RES}_j$ acts on an input $\G$ having the distribution $\PlN$.

The following proposition is a straightforward consequence of the definition of $\Psi_{x,y}$,

\begin{Prop}[$\PlN$ is invariant under the resampling procedure]\label{invariance resampling procedure}
Let $j \in \lbrace 1,\dots, m_N \rbrace$. Then, if $\G \sim \PlN$, $\mathsf{RES}_j(\G)\sim \PlN$. In words, the law $\PlN$ is invariant under the map $\mathsf{RES}_j$.
\end{Prop}

We have now properly defined a resampling procedure on each sector $\mathbf{A}_j(N)$. As we saw, this procedure can be realized in two steps. Our next objective is to find a sufficient condition for step B to be realized: how can we ensure that the resampled path captures enough area.

\subsubsection{A sufficient condition to capture enough area}
    Let $a$ and $b$ be two points in $\mathbb N^2$ (with $\text{arg}(a)>\text{arg}(b)$). Assume we are given $\g \in \Lambda^{N^2}$ that passes through $a$ and $b$. Modify $\g$ between these two points replacing the original piece of path $\g_{a,b}$ by a new element $\tilde{\g}_{a, b} \in \La^{a \rightarrow b}$. We are looking for a condition on $\tilde{\g}_{a,b}$ that is sufficient for the modified path $\tilde{\g}$ to belong to $\La^{N^2}$. It is clear that a sufficient condition for $\tilde{\g}$ to enclose a sufficiently large area is
\begin{equation}\label{sufficient area 2}
    \left|\Enc\left(\tilde{\g}_{a,b}\cup [0,a]\cup[0,b]\right)\right|\geq \left|\Enc(\g)\cap \mathbf{A}_{a,b}\right|.
\end{equation}
As in~\cite{alan2}, we obtain a quantitative sufficient condition (see Corollary~\ref{condition aire suffisante finale}) for \eqref{sufficient area 2} to be satisfied. We directly import without a proof the result from there as they trivially extend to our setting.

We will need a notation: for $x,y \in \mathbb N^2$ distinct, we denote by $\ell_{x,y}$ the unique line that passes through $x$ and $y$. The following result is illustrated in Figure~\ref{fig: lem sufficient condition area}.
\begin{Lemma}\label{lem: sufficient condition capture enough area 1} Let $\g \in \La^{N^2}$. Let $j \in \lbrace 1,\ldots,m_N\rbrace$. Denote by $z_j$ the point of $\mathcal{C}(\g)$ of argument $j\theta_N$. Let $\ell_j$ be the tangent line of $\mathcal{C}(\g)$ at $z_j$. Let $x_j,y_j \in \g \cap \mathbf{A}_{z_{j-1},z_j}$ and $\mathbf{E}_{x_j,y_j}:=\mathbf{E}_j$ be the pentagon delimited by the lines $\ell_{j-1}, \ell_{j}, \ell_{x_j,y_j}, \ell_{0,x_j}$ and $\ell_{0,y_j}$. Then, 
\begin{equation}
    (\Enc(\g)\cap \mathbf{A}_{x_j,y_j})\subset \mathbf{T}_{0,x_j,y_j}\cup \mathbf{E}_j,
\end{equation}
where $\mathbf{T}_{0,x_j,y_j}$ is the triangle of apexes $0,x_j,y_j$.
\end{Lemma}
\begin{figure}[H]
    \centering
    \includegraphics{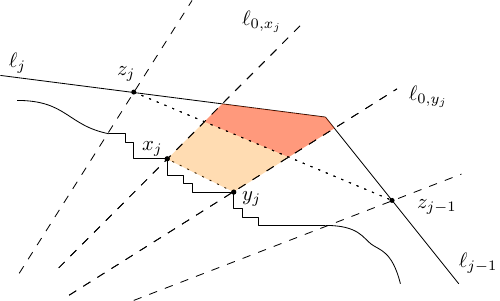}
    \caption{Illustration of Lemma~\ref{lem: sufficient condition capture enough area 1}. The light orange shaded region corresponds to $\mathbf{E}_j^0$, and the dark orange shaded region corresponds to $\mathbf{E}_j^1$.}
    \label{fig: lem sufficient condition area}
\end{figure}

An immediate consequence of the lemma is the fact that $(\ref{sufficient area 2})$ is satisfied if
\begin{equation}
    \left|\Enc\left(\tilde{\g}_{a,b}\cup [0,a]\cup[0,b]\right)\right|\geq \left|\mathbf{T}_{0,a,b}\right|+\left|\mathbf{E}_{a,b}\right|.
\end{equation}
We now write $\mathbf{E}_j=\mathbf{E}_j^0\cup \mathbf{E}_j^1$ (where $\mathbf{E}_j^0$ denotes the part of $\mathbf{E}_j$ that lies below $\ell_{z_{j-1},z_{j}}$), see Figure~\ref{fig: lem sufficient condition area}. It seems natural to seek an upper bound for $|\mathbf{E}_j|$. We start by determining an upper bound on $|\mathbf{E}_j^1|$. This is what we perform in the next lemma.

\begin{Lemma}[\hspace{1pt}{\cite[Lemma~3.9]{alan2}}]\label{lemma 1 area sufficient} Let $\g \in \La^{N^2}$. Assume $\g \subset B_{K_1 N}$. We keep the notations of the preceding lemma. Let $\overrightarrow{w_j}$ be the unit vector tangent to $\mathcal{C}(\g)$ at $z_j$ oriented in the counterclockwise sense. We define the set of sectors with \textup{moderate boundary turning} to be 
\begin{equation}
    \textup{MBT}=\Big\{j \in \lbrace 1, \ldots, m_N\rbrace, ~\Vert z_{j}-z_{j-1}\Vert \leq \dfrac{10\pi K_1 N}{m_N} \textup{ and } \angle (\overrightarrow{w_{j}},\overrightarrow{w_{j-1}})\leq \dfrac{10\pi}{m_N}\Big\}.
\end{equation}
Then, 
\begin{equation}
    |\textup{MBT}|\geq \dfrac{9m_N}{10},
\end{equation}
and if $j \in \textup{MBT}$, 
\begin{equation}
    |\mathbf{E}_j^1|\leq \dfrac{1}{2}(20)^3K_1^2\chi^{3} N\log N.
\end{equation}
\end{Lemma}

We now need to find an upper bound on $|\mathbf{E}_j^0|$. It is clear that such bound can only be obtained under the assumption that the points $x_j$ and $y_j$ defined above are not to far from $\mathcal{C}(\g)$. This remark motivates the following definition. Recall that if $\g\in \Lambda$ and $v\in \g$, $\mathsf{LR}(v,\g)=\mathrm{d}(v,\mathcal{C(\g)})$ denotes the local roughness of $v$ in $\g$.
\begin{Def}[Favourable sectors] Let $\g \in \La^{N^2}$.
Let $\varphi$ be a function that maps $(0,\infty)$ to itself and that satisfies, as $t \rightarrow 0$,
\begin{equation}
    \varphi(t)=o(\sqrt t).
\end{equation}
For $\chi >0$, we say that the sector $\mathbf{A}_j(N)$ is \textit{favourable under $\g$} if there exists $v \in \g\cap \mathbf{A}_j(N)$ such that 
\begin{equation}
    \LR(v,\g)\geq \varphi(\chi)N^{\frac{1}{3}}(\log N)^{\frac{2}{3}}.
\end{equation}
We define $\Unfav(\g,\chi)=\Unfav$ to be the set of $j \in \lbrace 1,\ldots, m_N\rbrace$ such that $\mathbf{A}_j(N)$ is not favourable under $\g$.
\end{Def}
\begin{Lemma}[\hspace{1pt}{\cite[Lemma~3.10]{alan2}}]\label{lemme 2 sufficient area}
Let $\g \in \La^{N^2}$. Assume that $\g \cap B_{K_2 N}=\emptyset$. Let $j \in \textup{MBT}\cap \Unfav$. Then, 
\begin{equation}
    |\mathbf{E}_j^0|\leq 20 K_1 \chi \varphi(\chi) N \log N.
\end{equation}
\end{Lemma}

From the preceding results, we can deduce the sufficient condition to capture enough area we were looking for,
\begin{Cor}[Sufficient condition to capture a large area]\label{condition aire suffisante finale} 
Let $\g \in \La^{N^2}$ be such that $\g \subset B_{K_1 N}\setminus B_{K_2 N}$. Let $j \in \lbrace 1,\ldots, m_N\rbrace$. Assume $j \in \textup{MBT}\cap \Unfav$. Let $x_j, y_j\in \g \cap \mathbf{A}_j(N)$ be chosen accordingly to the procedure described in Definition \textup{\ref{def procedure}}. Then, if we resample $\g$ between $x_j$ and $y_j$, replacing the path $\g_{x_j,y_j}$ by a new path $\tilde{\g}_{x_j,y_j}$, the new path $\tilde{\g}$ will be an element of $\La^{N^2}$ if the following condition is satisfied,
\begin{equation}
    \left|\Enc(\tilde{\g}_{x_j,y_j}\cup[0,x_j]\cup[0,y_j])\right|-\left|\mathbf{T}_{0,x_j,y_j}\right| \geq \left(\tfrac{1}{2}(20)^3K_1^2\chi^{3}+20K_1 \varphi(\chi)\chi\right)N\log N.
\end{equation}
\end{Cor}
\subsubsection{The right shape}\label{the right shape}
Now that we have a sufficient condition for a resampled path to be accepted, we need to exhibit an event that will ensure this sufficient condition. We make good use of the event $\mathsf{LogGAC}(x,y,\eta)$ introduced in Definition~\ref{definition log good area capture}. As we are about to see, the realisation of this event will ensure that the condition of Corollary~\ref{condition aire suffisante finale} will be satisfied. Nevertheless, we also need our resamplings to have a large local roughness and that is what motivates Definition~\ref{def sid}.

\begin{Def}[Significant inward deviation]\label{def sid}
Let $\eta > 0$ and $x,y \in \mathbb N^2$ with $\arg(x)>\arg(y)$. Let $\g \in \La$ be a path which passes by $x$ and $y$. We say that $\g$ realizes the event $\mathsf{LogSID}(x, y ,\eta)$ (meaning ``logarithmic significant inward deviation'') if there exists $z \in \g_{x,y}$ such that 
\begin{equation}
\mathrm{d} \left(z, \partial \text{conv}([0,x] \cup [0,y] \cup \g_{x,y} )\right) \geq \eta \dist x-y \dist^{\frac{1}{2}} \left( \log \dist x-y \dist \right)^\frac{1}{2},
\end{equation}
where $\partial\textup{conv}$ is the border of the convex hull.
\end{Def}
We are going to resample in the sectors $\mathbf{A}_j(N)$ introduced above. For a path $\g\in \Lambda^{N^2}$, call $x_j=x_j(\g),y_j=y_j(\g)$ the extremities of the portion of $\g$ subject to resampling in $\mathbf{A}_j(N)$ (as described in Definition~\ref{def procedure}). We first check that provided $\chi>0$ is small enough, for a path $\g$ satisfying ``nice'' properties, if we replace the portion of $\gamma$ between $x_j$ and $y_j$ by a path satisfying $\mathsf{LogGAC}(x_j, y_j , \eta)$ and $\mathsf{LogSID}(x_j, y_j, \eta)$, we obtain an element of $\Lambda^{N^2}$ with a large local roughness. 

\begin{Def}[Successful action of $\mathsf{RES}_j$] Let $\eta >0$ and $1 \leq j \leq m_N$.
Let $\g \in \La^{N^2}$. We say that our resampling operation in the $j$-th sector $\mathsf{RES}_{j}$ \textbf{acts} $\eta$-\textbf{successfully} on $\g$ if the new path $\g'$ obtained after the operation realises the event $\mathsf{LogGAC}(x_j, y_j , \eta) \cap \mathsf{LogSID}(x_j, y_j, \eta)$.
\end{Def}
We start by proving that the occurrence of $\mathsf{LogGAC}(x_j, y_j , \eta)$ forces the area condition to be satisfied.
\begin{Lemma}\label{lemme gac bon} Let $\eta >0$. There exists $\chi_1=\chi_1(\eta)>0$ such that for $\chi\in (0,\chi_1)$ the following assertion holds. Let $\g \in \La^{N^2}$ such that $\g \subset B_{K_1N}\setminus B_{K_2N}$. Let $j\in \lbrace 1,\ldots,m_N\rbrace$. Assume that $j \in \textup{MBT}\cap\Unfav$. Assume that we modify $\g$ between $x_j$ and $y_j$ and that the modified path $\g'$ satisfies $\g'\in \mathsf{LogGAC}(x_j, y_j , \eta)$. Then, $\g'\in \Lambda^{N^2}$.
\end{Lemma}

\begin{proof}
It is easy to see from Definition~\ref{def procedure} that
\begin{equation}
    \angle(x_j,y_j)\geq \dfrac{\theta_N}{4}.
\end{equation}
Using this and the fact that $x_j,y_j \notin B_{K_2N}$, we get 
\begin{equation}
    \Vert x_j-y_j\Vert \geq \dfrac{K_2 \chi}{2\pi}N^{\frac{2}{3}}(\log N)^{\frac{1}{3}}.
\end{equation}
The occurence of $\mathsf{LogGAC}(x_j, y_j , \eta)$ implies that
\begin{equation}
    \left|\Enc(\g'_{x_j,y_j}\cup [0,x_j]\cup [0,y_j])\right|-|\mathbf{T}_{0,x_j,y_j}|\geq \frac{\eta}{2}\left(\frac{2}{3}\right)^{\frac{1}{2}}(K_2/2\pi)^{\frac{3}{2}}\chi^{\frac{3}{2}}N\log N.
\end{equation}
Since $\varphi(t)=o(\sqrt t)$ as $t\rightarrow 0$, we see that condition in Corollary~\ref{condition aire suffisante finale} is ensured by fixing $\chi>0$ small enough. 
\end{proof}
A direct consequence of the proof is the following identity: if $\g\in \Lambda^{N^2}$  and $j\in \lbrace 1,\ldots, m_N\rbrace$ satisfy the properties above, and $\chi>0$ is small enough,
\begin{multline}\label{res j borne unif}
    \PP \left[ \mathsf{RES}_{j} \textup{ acts } \eta \textup{-successfully on }\g \right]=\\\frac{|\lbrace \g'_{x_j,y_j}\in \Lambda^{x_j\rightarrow y_j}, \text{ }(\g\setminus \g_{x_j,y_j})\cup \g'_{x_j,y_j}\in \mathsf{LogGAC}(x_j, y_j , \eta)\cap \mathsf{LogSID}(x_j, y_j, \eta)\rbrace|}{|\lbrace \g'_{x_j,y_j}\in \Lambda^{x\rightarrow y}, \text{ }(\g\setminus \g_{x_j,y_j})\cup \g'_{x_j,y_j}\in\Lambda^{N^2} \rbrace|}.
\end{multline}

\begin{Lemma}\label{lemme sid bon}
 Let $\eta>0$. There exists $\chi_2=\chi_2(\eta)>0$ such that for $\chi\in (0,\chi_2)$ the following assertion holds. Let $\g \in \La^{N^2}$ such that $\g \subset B_{K_1N}\setminus B_{K_2N}$. Let $j\in \lbrace 1,\ldots,m_N\rbrace$. Assume that $j \in \textup{MBT}\cap\Unfav$ and that $\mathsf{RES}_j$ acts $\eta$-successfully on $\g$. Then, $\mathbf{A}_j(N)$ is favourable under $\g'=\mathsf{RES}_j(\g)$.
 \end{Lemma}
 \begin{proof}
Since $\g'\in\mathsf{LogSID}(x_j, y_j, \eta)$, there exists $z \in \g'_{x_j,y_j}$ such that 
\begin{equation}
    \textup{d}\left(z, \partial\text{conv}\left([0,x_j]\cup[0,y_j]\cup \g'_{x_j,y_j}\right)\right)\geq \eta \Vert x_j-y_j \Vert^{\frac{1}{2}}(\log \Vert x_j-y_j \Vert)^{\frac{1}{2}}.
\end{equation}
 One may write $\g'=\left(\g\cap \mathbf{A}_{x_j,y_j}^c\right)\cup\g'_{x_j,y_j}$, so that we have
\begin{eqnarray*} 
    \textup{d}\left(z,\mathcal{C}\left(\g'\right)\right) &=& \textup{d}\left(z, \mathcal{C}\left(\left(\g\cap \mathbf{A}_{x_j,y_j}^c\right)\cup\g'_{x_j,y_j}\right)\right) \\ &\geq& \textup{d}\left(z,\partial\text{conv}\left([0,x_j]\cup[0,y_j]\cup \g'_{x_j,y_j}\right)\right)\\ &\geq & \eta \Vert x_j-y_j \Vert^{\frac{1}{2}}(\log \Vert x_j-y_j \Vert)^{\frac{1}{2}} \\ &\geq & \dfrac{1}{2}\sqrt{K_2}\left(\frac{2}{3}\right)^{\frac{1}{2}}\eta \chi^{\frac{1}{2}}N^{\frac{1}{3}}(\log N)^{\frac{2}{3}}.
\end{eqnarray*}
Since $\varphi(t)=o(\sqrt{t})$ as $t\rightarrow 0$, we get the result choosing $\chi>0$ small enough (in terms of $\eta$).
\end{proof}
We now fix $\eta>0$ and $\chi=\chi(\eta)>0$ sufficiently small so that Lemmas~\ref{lemme gac bon} and~\ref{lemme sid bon} are true.
Combining Lemma~\ref{lemme sid bon} and Equation (\ref{res j borne unif}) we get, if $\g\in \Lambda^{N^2}$ satisfies the conditions above and if $j\in \textup{MBT}\cap \Unfav$,
\begin{equation}
    \PP[\mathbf{A}_j(N) \text{ is favourable under }\mathsf{RES}_j(\g)]\geq \mathbb P_{x_j,y_j}[\mathsf{LogGAC}(x_j, y_j, \eta) \cap \mathsf{LogSID}(x_j, y_j, \eta)],
\end{equation}
where we recall that $\PP_{x_j,y_j}$ is the uniform law on $\Lambda^{x_j\rightarrow y_j}$.
The last step is then to estimate the probability on the right-hand side of the above inequality. This step is actually done using the following result, whose proof is postponed to the appendix.
\begin{Lemma}\label{Lemme log gac logsid}
    Let $\eta, \eps > 0$. There exist $C=C(\varepsilon), N_0=N_0(\varepsilon,\eta) > 0$ such that for any $N \geq N_0$, any $x,y \in \N^2$ such that $\arg(x) > \arg(y)$, $\theta(x,y) \in [\eps, \pi/2 - \eps]$ and $\Vert y-x \Vert \geq N_0$,
    \begin{equation}
        \PP_{x,y}\left[\mathsf{LogGAC}(x,y,\eta) \cap \mathsf{LogSID}(x,y,\eta)\right] \geq N^{-C\eta^2}.
    \end{equation}
\end{Lemma}
We can now state the main result of this section.
\begin{Lemma}[Successful resampling]\label{first lemma good shape} 
Let $\eta >0$. Let $\chi=\chi(\eta)>0$ be sufficiently small so that Lemmas \textup{\ref{lemme gac bon}} and \textup{\ref{lemme sid bon}} hold true. Let $\g \in \La^{N^2}$ such that $\g \subset B_{K_1N} \setminus B_{K_2N}$. Let $j\in \lbrace 1,\ldots, m_N\rbrace$. Assume that $j\in \textup{MBT}\cap \Unfav$ and that $\g \notin\badmoinsA\cup \badplusA$ for $\mathbf{A}=\mathbf{B}_j(N)$ where $\varepsilon>0$ is given by Proposition \textup{\ref{prop: bad angle}}. Then, there exist $C=C(\varepsilon)> 0$, and $N_0=N_0(\varepsilon,\eta)\in \mathbb N$ such that for $N\geq N_0$,
\begin{equation}
    \PP[\mathbf{A}_j(N) \text{ is favourable under }\mathsf{RES}_j(\g)]\geq \PP \left[ \mathsf{RES}_{j} \textup{ acts } \eta \textup{-successfully on }\g \right] \geq N^{- C \eta^2}.
\end{equation} 
\end{Lemma}

 \subsubsection{How not to undo a large local roughness}
 As described earlier, we want to perform our random surgery successively on different sectors $\mathbf{A}_j(N)$ in order to maximise the probability to obtain a favourable output. However, if we do it without caution, might it be that we will lose in the process favourable sectors. In order to avoid such problem we must only resample sufficiently distant sectors. This problem was handled by Hammond in~\cite{alan2}. Once again, the proof immediately extends to our setup. 

 \begin{Prop}[\hspace{1pt}{\cite[Lemmas~3.14 and 3.16]{alan2}}]\label{main result resample successifs}
 Let $s_1(N)\in \lbrace 1,\ldots, m_N\rbrace$. Let $s_2(N) \in (0,\infty)$ such that 
\begin{equation}
    s_2(N)<\dfrac{\chi}{2}s_1(N).
\end{equation}
 Set $s_3(N)=s_1(N)+1$. Let $(k,j)\in \lbrace 1, \ldots, m_N\rbrace^2$ be such that $|k-j|\geq s_3(N)$. Let $\g \in \La^{N^2}$ be such that $\g\cap B_{K_2 N}=\emptyset$. Define $\g'=\mathsf{RES}_j(\g)$. Assume that 
\begin{equation}
    \max\left\lbrace \MFL(\g),\MFL(\g')\right\rbrace\leq
    s_2(N)N^{\frac{2}{3}}(\log N)^{\frac{1}{3}}.
\end{equation}
 Then, the sector $\mathbf{A}_k(N)$ is favourable under $\g$ if and only if it is favourable under $\g'$. Moreover, 
\begin{equation}
    \mathcal{C}(\g)\cap \mathbf{A}_k(N)=\mathcal{C}(\g')\cap \mathbf{A}_k(N).
\end{equation}
 
 \end{Prop}

\subsubsection{The proof of the lower bound for \texorpdfstring{$\MLR$}{}}
A naive approach to the proof of the lower bound would be to apply the resampling successively on all the sectors $\mathbf{A}_j(N)$. As we saw in the preceding section, this raises a major problem. Applying the random surgery consecutively on two successive sectors might have for consequence the disappearance of a favourable sector, which is something we clearly seek to avoid. This problem can be solved working on distant sectors as seen in the last section. This is why we decide to resample roughly one in $s_3(N)$ sectors. We decide to randomly choose the sectors we resample: the reason is obvious, we do not want to resample a sector where there is already a favourable sector. We now explain in detail our strategy.
 
 We keep the notations of the above sections. We want to define a complete resampling $\mathsf{RES}$ that will involve all the procedures $\mathsf{RES}_j$ we constructed before and the quantities $\lbrace s_i(N)\rbrace_{i=1,2,3}$ introduced in Proposition~\ref{main result resample successifs}, with their form being specified later.
 
 Let $(\Omega,\mathcal{F},\mathbb P)$ be a probability space in which are defined: a random path $\G$ of law $\PlN$, and the operations of resampling $\mathsf{RES}_j$ for $1\leq j \leq m_N$, which act independently. We also generate a sequence of i.i.d random variables $\lbrace X_k\rbrace_{1\leq k \leq m_N}$ that have the law of a Bernoulli of parameter $\frac{1}{s_3(N)}$. We now properly define our entire resampling procedure $\mathsf{RES}$. This operation will be defined under $(\Omega,\mathcal{F},\mathbb P)$. We build it by induction. Set $\G_{0}=\G$ (the input). For $1\leq j \leq m_N$, if $X_j=1$ then set $\G_{j}=\mathsf{RES}_{j}(\G_{j-1})$, else set $\G_j=\G_{j-1}$ (no action is taken and the sector $\mathbf{A}_{j}(N)$ remains unchanged). We also set $\Unfav_j=\Unfav(\G_j,\chi)$ and $\text{MBT}_j=\text{MBT}(\G_j)$. As we previously saw, the law $\PlN$ is invariant under $\mathsf{RES}$. We will analyse $\PlN$ by identifying it with the law of the output $\G_{m_N}$ of $\mathsf{RES}$ under the measure $\mathbb P$. However, the only way to analyse the output of the procedure is if it the input lies in a space $\mathcal{G}$ of ``good'' paths. Introduce for $0\leq i \leq m_N$,
\begin{equation}
    \mathcal{G}_{1,i}:=\{ \MFL(\G_i)\leq s_2(N)N^{\frac{2}{3}}(\log N)^{\frac{1}{3}}\},
\end{equation}
\begin{equation}
    \mathcal{G}_{2,i}:=\lbrace \G_i\subset B_{K_1 N}\setminus B_{K_2 N}\rbrace,
\end{equation}
\begin{equation}
    \mathcal{G}_{3,i}:=\{ \G_i\notin \badmoins\cup\badplus\},
\end{equation}
where $\varepsilon>0$ is given by Proposition~\ref{prop: bad angle}.
We also define for $0\leq i \leq m_N$, 
\begin{equation}
\mathcal{G}_{i}:=\bigcap_{1\leq j \leq 3}\mathcal{G}_{j,i}, \textup{ and}\qquad \mathcal{G}_{(i)}:=\bigcap_{0\leq j \leq i}\mathcal{G}_j.
\end{equation}
Finally, the space of good outcomes that is interesting is nothing but
\begin{equation}
    \mathcal{G}:=\mathcal{G}_{(m_N)}.
\end{equation}
For $\varepsilon_1 \in (0,\frac{2}{3})$ to be fixed later, we define 
\begin{equation}\label{def des sn} 
    s_3(N):=N^{\varepsilon_1}, \qquad \text{and} \qquad         s_2(N):=\dfrac{\chi}{4}N^{\varepsilon_1},
\end{equation}
 so that the conditions of Proposition~\ref{main result resample successifs} are satisfied. Let us first prove that the set of good outcomes $\mathcal{G}$ happens with sufficiently large probability.
\begin{Lemma}\label{first lemma final proof} There exist two constants $c,C>0$ such that
\begin{equation}
    \mathbb P\left[\mathcal{G}^c\right]\leq C\exp(-c\chi^{\frac{3}{2}}N^{\frac{3\varepsilon_1}{2}}\log N)+C\exp(-cN)+C\exp(-cN\theta_N).
\end{equation}
\end{Lemma}
\begin{proof}
By definition, we have that 
\begin{multline*}
    \mathbb P[\mathcal{G}^c]\leq m_N\PlN\left[\MFL(\G)> s_2(N)N^{\frac{2}{3}}(\log N)^{\frac{1}{3}}\right]\\+m_N\PlN[\G \subset (B_{K_1N}\setminus B_{K_2N})^c]+m_N\PlN[\badmoinsun\cup \badplusun].
\end{multline*}
We then conclude by applying Proposition~\ref{mfl.control.prop}, Lemma~\ref{lem: confinement lemma} and Proposition~\ref{prop: bad angle}.
\end{proof}
Our second goal is to control the evolution of the sets $\text{MBT}$ and $\Unfav$ during the procedure (under the assumption of $\mathcal{G}$). This control is obtained in the following lemma.
\begin{Lemma}[\hspace{1pt}{\cite[Lemma~3.17]{alan2}}]\label{second lemme preuve finale}
 For $1\leq j \leq m_N-s_3(N)$, if $\mathcal{G}$ occurs then
 \begin{equation}\label{result 1}
     \Unfav_j\cap \lbrace j+s_3(N),\ldots, m_N\rbrace = \Unfav_{0}\cap \lbrace j+s_3(N), \ldots, m_N\rbrace,
 \end{equation}
 and
 \begin{equation}\label{result 2}
     \textup{MBT}_j\cap \lbrace j +s_3(N),\ldots, m_N\rbrace=\textup{MBT}_{0}\cap \lbrace j+s_3(N),\ldots, m_N\rbrace.
 \end{equation}
\end{Lemma}
\begin{proof}
Let $(j,k)\in \lbrace 1,\dots, m_N \rbrace^2$ satisfy $j+s_3(N)\leq k \leq m_N$. If $\mathcal{G}$ occurs, then we may apply Proposition~\ref{main result resample successifs} to each of the first $j$ stages of the procedure $\mathsf{RES}$. In particular, this tells us that $\mathbf{A}_k(N)$ is favourable under $\G_j$ if and only if it is favourable under $\G_0$: this is exactly (\ref{result 1}).

Then, the condition $k \in \text{MBT}_j$ is determined by the data $\mathcal{C}(\G_j)\cap \mathbf{A}_k(N)$. However, one may also successively apply Proposition~\ref{main result resample successifs} to the first $j$ stages of the procedure $\mathsf{RES}$ to obtain that
\begin{equation}
    \mathcal{C}(\G_j)\cap \mathbf{A}_k(N)=\mathcal{C}(\G_0)\cap \mathbf{A}_k(N),
\end{equation}
which yields (\ref{result 2}).
\end{proof}
Now, recall from Lemma~\ref{lemma 1 area sufficient} that if $\G_0\subset B_{K_1 N}$ then
\begin{equation}
    |\text{MBT}_0|\geq \dfrac{9m_N}{10}\geq \dfrac{m_N}{2}.
\end{equation}
By the occurence of $\mathcal{G}_{2,0}$, we may then define a set $\overline{\text{MBT}_0}\subset \text{MBT}_0$ that satisfies the following properties:
\begin{enumerate}\label{condition}
    \item[-]  $|\overline{\text{MBT}_0}|\geq \frac{m_N}{4s_3(N)}$,
    \item[-] each pair of consecutive elements of $\overline{\text{MBT}_0}$  differ by at least $2s_3(N)+1$.
\end{enumerate}
We also write $\overline{\text{MBT}_0}\cap \Unfav_0=\lbrace p_1,\ldots,p_{r_1}\rbrace$ and $\overline{\text{MBT}_0}\cap \Unfav_0^c=\lbrace q_1,\ldots,q_{r_2}\rbrace$ with $r_1+r_2=|\overline{\text{MBT}_0}|$. For $1\leq r \leq r_1$, let $P_r$ denote the event that in the action of $\mathsf{RES}$, at stage $p_r$, $\mathsf{RES}_{p_r}$ is chosen to act and acts successfully while no action is taken at stages $j$ for $j \in \lbrace p_r-s_3(N),\ldots ,p_r-1\rbrace \cup \lbrace p_r+1,\ldots,p_r+s_3(N)\rbrace$. Similarly, for $1\leq r \leq r_2$, let $Q_r$ denote the event that in the action of $\mathsf{RES}$, at stage $q_r$ no action is taken, and same for the stages $j$ with $j \in \lbrace q_r-s_3(N),\ldots ,q_r\rbrace \cup \lbrace q_r+1,\ldots,q_r+s_3(N)\rbrace$. We make good use of these events to make favourable sectors appear.
\begin{Lemma}\label{troisieme lemme preuve finale}
 We keep the notations introduced above. Then, 
 \begin{enumerate}
     \item[$(i)$] for each $r \in \lbrace 1,\dots,r_1\rbrace$, $\mathcal{G}\cap P_r$ implies that the sector $\mathbf{A}_{p_r}(N)$ is favourable under the output $\G_{m_N}$,
     \item[$(ii)$] for each $r \in \lbrace 1,\dots,r_2\rbrace$, $\mathcal{G}\cap Q_r$ implies that the sector $\mathbf{A}_{q_r}(N)$ is favourable under the output $\G_{m_N}$.
 \end{enumerate}
\end{Lemma}
\begin{proof}
\begin{enumerate}
    \item[$(i)$] By Lemma~\ref{second lemme preuve finale}, we have $p_r \in \Unfav_{p_r-s_3(N)}\cap \text{MBT}_{p_r-s_3(N)}$, since $p_r \in \Unfav_{0}\cap \text{MBT}_{0}$. Given $P_r$, we then have $p_r\in \Unfav_{p_r-1}\cap \text{MBT}_{p_r-1}$ because $\G_{p_r-s_3(N)}=\G_{p_r-1}$. Applying Lemma~\ref{lemme sid bon} to the $\eta$-successful action of $\mathsf{RES}_{p_r}$ on $\G_{p_r-1}$, we find that $\mathbf{A}_{p_r}(N)$ is favourable under $\G_{p_r}$. This remains the case for $\G_{p_r+s_3(N)}$ because $\G_{p_r}=\G_{p_r+s_3(N)}$. Now, thanks to Proposition~\ref{main result resample successifs}, $\mathbf{A}_{p_r}(N)$ stays a favourable sector during the remaining stages of the procedure $\mathsf{RES}$ (recall that we also work under $\mathcal{G}$ in which $\mathcal{G}_{1,i}$ occurs for $1\leq i \leq m_N$).
    \item[$(ii)$] This is essentially the same argument as the one depicted above. We have that $q_r \notin \Unfav_{q_r-s_3(N)}$ so that the inaction of $\mathsf{RES}_{j}$ for $j \in \lbrace q_r-s_3(N),\dots, q_r+s_3(N)\rbrace$ entails $q_r \notin \Unfav_{q_r+s_3(N)}$. We can then conclude as above using Proposition~\ref{main result resample successifs}.
\end{enumerate}

\end{proof}
An immediate consequence of Lemma~\ref{troisieme lemme preuve finale} is that 
\begin{equation}\label{eq: final inclusion}
    \left(\bigcup_{i=1}^{r_1} P_i \cup \bigcup_{i=1}^{r_2} Q_i\right) \cap \mathcal{G}\subset \left\lbrace \MLR\left(\G_{m_N}\right)\geq \varphi(\chi)N^{\frac{1}{3}}(\log N)^{\frac{2}{3}}\right\rbrace.
\end{equation}
Also, note that by the discussion above, if $\mathcal{G}$ happens then $r_1+r_2\geq \frac{m_N}{4s_3(N)}$. Then,
\begin{multline}
    \left(\bigcup_{i=1}^{r_1} P_i \cup \bigcup_{i=1}^{r_2} Q_i\right)^c \cap \mathcal{G}\subset \nonumber\\ \bigcap_{i=1}^{\frac{m_N}{8s_3(N)}}( \lbrace r_1\geq i\rbrace \cap P_i^c \cap \mathcal{G}_{(p_i-s_3(N)}) \cup \bigcap_{i=1}^{\frac{m_N}{8s_3(N)}}( \lbrace r_2\geq i\rbrace \cap Q_i^c \cap \mathcal{G}_{(q_i-s_3(N)}).
\end{multline}
Using all the material above, we claim that, for any $K\in \N$, given $\lbrace r_1\geq K\rbrace \cap \mathcal{G}_{(p_K-s_3(N))}$ and the values of $\mathds{1}_{P_1},\ldots, \mathds{1}_{P_K-1}$, the conditional probability that $P_K$ occurs is at least 
\begin{equation}
    N^{-C\eta^2}\frac{1}{s_3(N)}\left(1-\frac{1}{s_3(N)}\right)^{2s_3(N)},
\end{equation}
where $\eta>0$ and $C=C(\varepsilon)$ is given by Lemma~\ref{first lemma good shape}. Indeed, the event on which we condition is measurable with respect to $\sigma(\G_0,\ldots,\G_{p_K-s_3(N)})$, and if it occurs, $\G_{p_K-s_3(N)}$ satisfies the hypothesis of Lemma~\ref{first lemma good shape}. The claim then follows by this lemma. As a result,
\begin{equation}\label{eq impo 1}
    \mathbb P\left[\bigcap_{i=1}^{\frac{m_N}{8s_3(N)}}\left(\lbrace r_1\geq i\rbrace \cap P_i^c \cap \mathcal{G}_{(p_i-s_3(N)}\right)\right]\leq \left(1-N^{-C\eta^2}\frac{1}{s_3(N)}\left(1-\frac{1}{s_3(N)}\right)^{2s_3(N)}\right)^{\frac{m_N}{8s_3(N)}}.
\end{equation}
Similarly,
\begin{equation}\label{eq impo 2}
    \mathbb P\left[\bigcap_{i=1}^{\frac{m_N}{8s_3(N)}}\left( \lbrace r_2\geq i\rbrace \cap Q_i^c \cap \mathcal{G}_{(q_i-s_3(N)}\right)\right]\leq \left(1-\left(1-\frac{1}{s_3(N)}\right)^{2s_3(N)+1}\right)^{\frac{m_N}{8s_3(N)}}.
\end{equation}
Now, Proposition~\ref{invariance resampling procedure} implies that
\begin{equation}
    \PlN\big[\MLR(\G)\geq \varphi(\chi)N^{\frac{1}{3}}(\log N)^{\frac{2}{3}}\big]=\mathbb P\big[ \MLR\left(\G_{m_N}\right)\geq \varphi(\chi)N^{\frac{1}{3}}(\log N)^{\frac{2}{3}}\big].
\end{equation}
Using (\ref{eq: final inclusion}) we then obtain, 
\begin{equation}
    \PlN[\MLR(\G)\geq \varphi(\chi)N^{\frac{1}{3}}(\log N)^{\frac{2}{3}}]\geq 1-\mathbb P[\mathcal{G}^c]-\mathbb P\Big[ \left(\bigcup_{i=1}^{r_1} P_i \cup \bigcup_{i=1}^{r_2} Q_i\right)^c \cap \mathcal{G}\Big].
\end{equation}
Using Lemma~\ref{first lemma final proof}, Equations \eqref{eq impo 1}, \eqref{eq impo 2},~\eqref{eq: asymptotics mn} and Definition~\ref{def des sn}, we obtain the vanishing of $\PlN\big[\MLR(\G)< \varphi(\chi)N^{\frac{1}{3}}(\log N)^{\frac{2}{3}}\big]$ to $0$ as $N$ goes to infinity. More precisely, we have a quantitative result. Recall that one may take $\eta>0$ sufficiently small so that it always satisfies $C\eta^2<\varepsilon_1$. For some constant $c(\chi)>0$,
\begin{equation}
    \mathbb P[\mathcal{G}^c]\leq \exp(-c(\chi)N^{\frac{3\varepsilon_1}{2}}\log N),
\end{equation}
\begin{equation}
    \left(1-N^{-C\eta^2}\frac{1}{s_3(N)}\left(1-\frac{1}{s_3(N)}\right)^{2s_3(N)}\right)^{\frac{m_N}{8s_3(N)}}\leq \exp(-c(\chi)N^{\frac{1}{3}-3\varepsilon_1}(\log N)^{\frac{1}{3}}),
\end{equation}
and
\begin{equation}
    \left(1-\left(1-\frac{1}{s_3(N)}\right)^{2s_3(N)+1}\right)^{\frac{m_N}{8s_3(N)}}\leq \exp(-c(\chi)N^{\frac{1}{3}-\varepsilon_1}(\log N)^{\frac{1}{3}}).
\end{equation}
Taking $\varepsilon_1=\frac{2}{27}$, we obtain the following result,
\begin{Prop}\label{borne inf mlr}
 For any $\chi>0$ sufficiently small, there exist $c(\chi),N_0>0$ such that for $N\geq N_0$,
\begin{equation}
     \PlN\big[\MLR(\G)< \varphi(\chi)N^{\frac{1}{3}}(\log N)^{\frac{2}{3}}\big]\leq \exp(-c(\chi)N^{\frac{1}{9}}\log N).
\end{equation}
\end{Prop}
\begin{Rem}
 Recall that above we first consider values of $\varepsilon_1$ sufficiently small (typically $\varepsilon_1<1/9$), and then fix a value of $\eta>0$ such that $C\eta^2\ll\varepsilon_1$ (where $C$ is the constant given by Lemma~\ref{first lemma good shape}. Finally, we choose $\chi=\chi(\eta)$ sufficiently small such that Lemmas~\ref{lemme gac bon} and~\ref{lemme sid bon} hold. We did not try to obtain an optimal bound in Proposition~\ref{borne inf mlr}.
\end{Rem}
\subsubsection{The proof of the lower bound for \texorpdfstring{$\MFL$}{}}
Once (\ref{maxlr.thm}) is obtained, getting a lower bound on $\MFL$ is immediate with the coupling described in Section~\ref{section preliminary results}. 

  
\begin{Prop}
Let $\chi>0$ be sufficiently small. There exist a function $\phi$ such that $\phi(t) \goes{}{t \rightarrow 0^+}{+\infty}$ such that for $N$ large enough,
\begin{equation}
    \PlN [ \MFL(\G) < \varphi(\chi)^3N^\frac{2}{3} (\log N)^\frac{1}{3} ] \leq N^{-\phi(\chi)}.
\end{equation}
\end{Prop}
\begin{proof}
We split the event $\{ \MFL(\G) < \varphi(\chi)^3N^\frac{2}{3} (\log N)^\frac{1}{3} \} $ into the union of two events, according to the value of the maximal local roughness,
\begin{multline}
\PlN [\MFL(\G) < \varphi(\chi)^3N^\frac{2}{3} (\log N)^\frac{1}{3} ]  \leq  \PlN [ \MLR < \varphi(\chi)(\log N)^{\frac{2}{3}}N^{\frac{1}{3}} ]
 \\ + \PlN [ \MFL(\G) < \varphi(\chi)^3N^\frac{2}{3} (\log N)^\frac{1}{3},~ \MLR \geq \varphi(\chi)(\log N)^{\frac{2}{3}}N^{\frac{1}{3}} ].  
\end{multline}
The first term of this sum has been shown to be smaller than $\exp(-c(\chi)N^{1/9}\log N)$ for $\chi$ sufficiently small and $N$ sufficiently large. To bound the second term, we use the coupling described in Proposition~\ref{prop: couplage 1}.
Indeed, we invoke the proof of Proposition~\ref{mlr.conrol.prop} to get that for some $c>0$,
\begin{multline}
    \PlN [ \MFL(\G) < \varphi(\chi)^3N^\frac{2}{3} (\log N)^\frac{1}{3},~ \MLR \geq \varphi(\chi)(\log N)^{\frac{2}{3}}N^{\frac{1}{3}} ]\leq \\ O(N^{6})\exp(-\tfrac{c}{\varphi(\chi)}\log N)+\exp(-cN),
\end{multline}
where the second term comes from the bound on $\PlN[\G\cap (B_{K_1N})^c\neq \emptyset]$. This gives the result provided $\chi>0$ is sufficiently small in which case one can set $\phi(\chi):=-(6-c\varphi(\chi)^{-1})>0$.
\end{proof}

\section{Discussion: extending the results to the Wulff setting}\label{section extension to other models}

In that section, we briefly explain how this model --- despite being simple --- can be seen as a toy model for the phase separation interface in the Wulff setting. In particular, we argue that it shares all the crucial properties of the latter object. We choose to illustrate this fact through the example of the outermost circuit in a constrained subcritical random-cluster model, a setting that was introduced by Hammond in~\cite{alan3, alan1, alan2}.

\subsection{Definition of the random-cluster model}
We quickly recall the definition of the random-cluster model (also known as FK percolation) and give a few of its basic properties (see~\cite{duminilcopin2017lectures} for more details). 
 
The random-cluster model on $\Z^2$ is a model of random sub-graphs of $\Z^2$. Its law is described by two parameters, $p \in [0,1]$ and $q > 0$. We start by defining it on a finite graph. Let $G = \left( V, E \right)$ be a finite sub-graph of $\Z^2$. The \emph{boundary} of $G$ is defined by 
\begin{equation}
\partial G := \left\lbrace x \in V, \exists y \notin V, \lbrace x,y \rbrace \in E(\Z^2) \right\rbrace.
\end{equation}

A \textit{percolation configuration} on $G$ is an element $\om$ of $\left\lbrace 0,1 \right\rbrace^{E}$. Each edge of $G$ takes the value 0 or 1. We say that an edge $e \in G$ is \textit{open} if $\om(e) = 1$ and is \textit{closed} otherwise. Two vertices $x,y \in G$ are said to be connected in $\om$ if there exists a collection of vertices of $G$ $x=x_0, x_1, \dots, x_n = y$ such that for all $0\leq i \leq n-1$, $\lbrace x_i, x_{i+1} \rbrace\in E$ and $\om(\lbrace x_i,x_{i+1}\rbrace)=1$. We call this event $\lbrace x \leftrightarrow y \rbrace$.  A \textit{cluster} of $\om$ is a maximal connected component of the set of vertices (it can be an isolated vertex). Given a percolation configuration $\om$, we let $o(\om)$ denote its number of open edges, and let $k(\om)$ denote its number of vertex clusters. A \textit{boundary condition} on $G$ is a partition $\eta = \mathcal{P}_1 \cup \dots \cup \mathcal{P}_k$ of $\partial G$. From a configuration $\om \in \left\lbrace 0,1 \right\rbrace^{E}$, we create a configuration $\om^\eta$ by identifying together the vertices that belong to the same set $\mathcal{P}_i$ of $\eta$. We shall call the \textit{free boundary} condition (resp. wired boundary condition) the partition made of singletons (resp. of the whole set $\partial G$). We will denote it $\eta= 0$ (resp. $\eta=1$).
\begin{Def}
Let $G =  \left( V, E\right)$ be a finite subgraph of $\Z^2$, and $\eta$ be a boundary condition on $G$. Let $p \in [0,1]$ and $q > 0$. The random-cluster measure on $G$ with boundary condition $\eta$ is the following measure on $\left\lbrace 0,1 \right\rbrace^{E}$,
\begin{equation}
\phi^\eta_{p,q,G}\left( \om \right) = \frac{1}{Z^\eta_{p,q,G}} \left(\frac{p}{1-p}\right)^{o(\om)}q^{k(\om^\eta)},
\end{equation}
where $Z^\eta_{p,q,G}>0$ is the normalisation constant ensuring that $\phi^\eta_{p,q,G}$ is a probability measure.
\end{Def}
For $\eta = 0$ and $\eta = 1$, this measure can actually be extended to the whole $\Z^2$ by taking the weak limit of the measures $\phi^\eta_{p,q,G_n}$ over any exhaustion $\left(G_n\right)_{n \in \N}$ of $\Z^2$. Moreover, the limit measure does not depend on the choice of the exhaustion. Below, we will simply write $\phi_{p,q}^\eta$ instead of $\phi_{p,q,\Z^2}^\eta$. A fundamental feature of this model is that it undergoes a \textit{phase transition}: for any $q \geq 1$, there exists a critical parameter $p_c = p_c(q) \in (0,1)$ such that:
\begin{itemize}
    \item[-] $\forall p < p_c(q), \phi^1_{p,q}\left( 0 \leftrightarrow \infty \right) = 0$.
    \item[-] $\forall p > p_c(q), \phi^0_{p,q} \left( 0 \leftrightarrow \infty \right) > 0$.
\end{itemize}
We will focus on the first case --- called \textit{the subcritical regime}. In this regime, it is classical that the choice of boundary conditions does not affect the infinite volume measure. Hence, we drop $\eta$ from the notation and simply write $\phi_{p,q}$ for the unique infinite volume measure when $p<p_c(q)$. Finally, in the subcritical regime, the following limit --- called \emph{the correlation length} --- exists and is strictly positive: for any $x\in \mathbb{S}^1$,
\begin{equation}
\xi_{p,q}(x) := -\lim_{n \rightarrow \infty}\frac{n}{\log \phi_{p,q}\left[0 \leftrightarrow \lfloor nx \rfloor\right]} >0.
\end{equation}
\subsection{Extension of the results to the random-cluster model}
For the rest of this section, let us fix $q \geq 1$ and $0<p<p_c(q)$. Due to the exponential decay of the size of the cluster of 0, it is easy to see that almost surely, the outermost open circuit surrounding 0 is well-defined. Following~\cite{alan1}, we call it $\Gamma_0$. $\Gamma_0$ sampled according to $\phi_{p,q}$ has to be seen as the analog of $\Gamma$ sampled according to $\PP_\lambda$ in our model. It remains to introduce the analog of the event $\lbrace \mathcal{A}(\Gamma) \geq N^2\rbrace$. Indeed introduce the area trapped by the circuit $\mathcal{A}(\Gamma_0)$, and define the conditioned measure $\phi^{N^2}_{p,q}:= \phi_{p,q}\left[\: \cdot \:\vert\:\mathcal{A}(\Gamma_0) \geq N^2\right]. $ This measure is the analog of $\PlN$. We now compare the properties of $\Gamma$ sampled according to $\PlN$ with the properties of $\Gamma_0$ sampled according to $\phi^{N^2}_{p,q}$.

\medskip

\noindent{\textbf{Global curvature: the Wulff shape. }}The first essential feature exhibited by $\PlN$ is the convergence of its sample paths at the macroscopic scale towards a deterministic curve given by the solution of a variational problem. This is a very classical result in statistical mechanics known as the \textit{Wulff phenomenon} (see~\cite{DobrushinKoteckyShlosmanWulffconstruction, Cerf2006} for detailed monographs about this theory). Define the following compact set:

\begin{equation}
    \mathcal{W} = \nu\bigcap_{u \in \mathbb{S}^1} \{ t \in \R^2, \langle t,u \rangle\leq \xi_{p,q}^{-1}(u) \},
\end{equation}
 with the constant $\nu > 0$ being chosen so that the set $\mathcal{W}$ is of volume 1 (and $\langle \cdot,\cdot \rangle$ denoting the usual scalar product on $\R^2$). The boundary $\partial \mathcal{W}$ of the Wulff shape plays the role of the limit shape $f_\lambda$ in our model. Indeed, one has the following result, which is the exact analog of Theorem~\ref{limit shape theorem}, and appears in~\cite[Proposition 1]{alan1}: for any $\eps > 0$,
 \begin{equation}
     \phi^{N^2}_{p,q}[ d_{\mathcal{H}}\left(\partial \mathcal{W}, N^{-1}\Gamma_0\right) > \eps ] \goes{}{n\rightarrow \infty}{0},
 \end{equation}
where $d_\mathcal{H}$ is the Hausdorff distance on compact sets of $\mathbb R^2$.

Moreover, we used at several occasions the strict concavity (and differentiability) of $f_\lambda$ (this is more or less the content of Lemma \ref{prop: bad angle}): it is well known (see~\cite[Theorem B]{campaninonioffevelenikozrandomcluster}), that $\mathcal{W}$ is indeed a strictly convex set and that its boundary $\partial \mathcal{W}$ is analytic.

\medskip 
\noindent{\textbf{Local Brownian nature of the interface. }  }A second dramatic feature of the model is the Brownian nature of the interface under a scaling of the type $N^{-1/3}\Gamma(tN^{2/3})$. Indeed, Donsker's invariance principle is heavily used in Subsection~\ref{subsectio lower bounds} to argue that the area captured by $\Gamma$ in a cone of angular opening of order $N^{-1/3}$ is of order at most $\beta N$, with some Gaussian tails on $\beta$. This is also the case in the random-cluster model, and is an instance of the celebrated \textit{Ornstein--Zernike theory}. Indeed, it is known since the breakthrough work of~\cite{campaninonioffevelenikozrandomcluster} that a subcritical cluster of FK percolation conditioned to link two distant points is asymptotically Brownian whenever the distance between the points goes to infinity. Hence, an analog of Donsker's invariance principle holds for a percolation interface at the scale $(N^{1/3}, N^{2/3})$, and the arguments of Subsection~\ref{subsectio lower bounds} can be reproduced \textit{mutatis mutandi}. For precise statements, see~\cite[Theorem A, Theorem C]{campaninonioffevelenikozrandomcluster} for the local Gaussian asymptotic for a subcritical cluster and the invariance principle towards a Brownian bridge for a rescaled subcritical cluster.

\medskip
\noindent{\textbf{Brownian Gibbs property in the random-cluster model. }} Perhaps the most essential tool used through the work is what was referred to --- according to~\cite{airylineensemble} --- as the \textit{Brownian Gibbs property} of the model, stating that if one forgets some portion of the walk $\Gamma$, then conditionally on the remaining portion of $\Gamma$, the distribution of the erased part is simply the distribution of a random walk conditioned to link both parts of the remaining non-erased walk $\Gamma$, conditioned on the event that the total enclosed area is at least $N^2$.

For the random-cluster model, the existence of this Brownian Gibbs property is less clear, for two distinct reasons. The first reason is that $\Gamma_0$ lacks --- at least locally --- the oriented structure of our random walk model. Thus, due to possible backtracks of $\Gamma_0$ it is not \textit{a priori} clear that one is able to perform the resampling operation between any two points, which could be a possible obstruction to the strategy described in this work. The second possible obstacle to the Brownian Gibbs property is the lack of independence: indeed the Spatial Markov property of the random-cluster model (see~\cite{duminilcopin2017lectures}) is the exact analog of the Brownian Gibbs property. However, opposed to our random walk model, one has to take into consideration the \textit{boundary conditions} enforced by the conditioning on the circuit outside of some fixed region. This could be a problem, as our strategy strongly relies on independence of the resampled piece --- conditionally on the area constraint.

These two obstructions have successfully been overcome by Hammond in~\cite{alan3}. We refer to this work to observe that one can indeed implement the resampling strategy in the context of the random-cluster model. However, in the remainder of this section, we quickly explain the strategy to overcome the two difficulties pointed out previously.

\begin{enumerate}
    \item[(i)] \textbf{Renewal structure and regularity of the outermost circuit. } It was observed in~\cite{alan3} that the first difficulty can be overcome by implementing the resampling strategy only between \textit{renewal points} of the outermost circuit. For the precise definition of these points, see~\cite[Definition 1.9]{alan3}: these are precisely the sites where the circuit can be resampled without any backtrack problem. Then, the main result of~\cite{alan3}, stated in Theorem 1.1, is that the probability of not seeing such a point in a cone of angular opening $uN^{-1}$ decays faster than $\exp(-u^2)$. Since all our resamplings take place in sectors of angular opening of order $N^{-1/3}$, one sees that imposing that the resampling occurs between renewals does not change the scaling of the quantities that we are computing.

    \item[(ii)] \textbf{Exponential mixing versus independence. } The second obstruction is ruled out by the exponential mixing property of subcritical random-cluster models (see~\cite{duminilcopin2017lectures}). Again, since our reasoning takes place at scales of order $N^{1/3}$ and the latter property ensures that random-cluster configurations decorrelate at a polynomial rate between spatial scales of logarithmic order, the lack of independence is not a problem in our setting (see the discussion in~\cite[Section 2.2]{alan1}).
\end{enumerate}

We hope to have convinced the careful reader that our techniques, even though written in a simple context, are robust and suitable to the analysis of any subcritical statistical mechanics models exhibiting the features discussed above (which in turn should be the case at least in a wide range of models). Finally let us conclude this section by mentioning that thanks to the well-known Edwards--Sokal coupling, this analysis might possibly allow the identification of the fluctuations scale of the facets of a droplet in a supercritical Potts model with $q$ colours, which is maybe a more physically appealing conclusion.  

\bibliographystyle{amsalpha}
\bibliography{bibowcat}
\appendix

\section{Computations on the simple random walk}\label{appendix computations on SRW}

This subsection is devoted to the proofs of Lemmas~\ref{good shape uniform law 1},~\ref{lemme log good shape uniform} and~\ref{Lemme log gac logsid} . The proofs are adapted from~\cite{alan1}. 

\begin{proof}[Proof of Lemma~\textup{\ref{good shape uniform law 1}}]
We assume that there exists some $\eps > 0$ such that $\eps \leq \theta(x,y) \leq \pi/2 - \eps$. Let $h = \Vert x-y \Vert$. Let $\mathbf{R}$ be the rectangle whose up-left corner is $x$, down-right corner is $y$ and which has horizontal and vertical sides. Since $\theta(x,y)\in [\varepsilon,\pi/2-\varepsilon]$, $\mathbf{R}$ is not degenerate. We start by changing the coordinates. We set $z_0 = (0,0)$, $z_1 = (h/4, 10 \eta \sqrt{h})$, $z_2 = (h/2, 5 \eta \sqrt{h})$, $z_3 = (3h/4, 10\eta \sqrt{h})$ and $z_4 =(h,0)$. Let $R_{\theta}: \R^2 \rightarrow \R^2$ be the rotation of the plane which maps $[x,y]$ to an horizontal segment and $T$ the translation which maps $x$ to $0$. Let $\Tilde{\mathbf{R}} = (T \circ \R_{\theta})(\mathbf{R})$. To ensure that $z_1, z_3 \in \Tilde{\mathbf{R}}$, we set $N_0=N_0(\varepsilon,\eta)$ large enough so that for $h\geq N_0$,
\begin{equation}
    \arctan(40\eta h^{-\frac{1}{2}}) \leq \varepsilon.
\end{equation}
We now assume that $h\geq N_0$. We will need the following definitions. 

A \emph{$\theta$-path} is an element of $(T \circ \R_{\theta})(\Lambda^{x\rightarrow y})$. Given $u,v\in \mathbb R^2$ and $\g\in \Lambda^{x\rightarrow y}$, we say that $u$ and $v$ are \emph{$\theta$-connected} in $\g$ if $u$ and $v$ belong to $(T \circ \R_{\theta})(\g)$. For $i \in \lbrace {0, \dots, 3} \rbrace$, let $H_i$ be the event that $z_i$ and $z_{i+1}$ are $\theta$-connected by a $\theta$-path which fluctuates less than $10 \dist z_{i+1} - z_i \dist^{\frac{1}{2}}$ around the segment $\left[ z_i, z_{i+1} \right]$. Finally, define the event $\GS = H_0 \cap H_1 \cap H_2 \cap H_3$, see Figure~\ref{fig: event shape}. 

Notice that  
\begin{equation}
\GS \subset \GAC{x}{y}{\eta}.
\end{equation}
\begin{figure}
    \centering
    \includegraphics{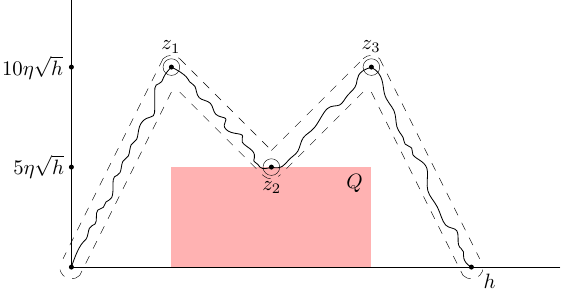}
    \caption{Illustration of the event $\mathsf{Shape}$ introduced in the proof of Lemma~\ref{good shape uniform law 1}. The $\theta$-path is allowed to fluctuate in the region surrounded by the dashed line.}
    \label{fig: event shape}
\end{figure}
Indeed, let $Q$ be the rectangle $[h/4, 3h/4] \times [0, 5 \eta \sqrt{h}]$ (see Figure~\ref{fig: event shape}). Then, if $\g$ realises $\mathsf{Shape}$, $(T \circ \R_{\theta})^{-1}(Q)$ and $\g$ do not cross, except maybe on a small region of area $O(h)$ around $(T \circ \R_{\theta})^{-1}(z_2)$. Hence, provided that $N$ is large enough, 
\begin{eqnarray*}
    \vert \Enc(\g \cap \mathbf{A}_{x,y}) \vert &\geq& \vert \mathbf{T}_{0,x,y} \vert + \frac{5}{2} \eta h^{\frac{3}{2}} - O(h)\\
    &\geq& \vert \mathbf{T}_{0,x,y} \vert + \eta h^{\frac{3}{2}}.
\end{eqnarray*}
Now, it remains to get a good estimate on $\PP_{x,y}[\GS]$. A simple computation coming from Gaussian fluctuations gives that there exists a constant $c>0$ such that for $i \in \lbrace 0, \dots, 3 \rbrace$, 
\begin{equation}
    \PP_{x,y}[H_i] \geq c \PP_{x,y}[ z_i \leftrightarrow z_{i+1}]\footnote{$\lbrace z_i \leftrightarrow z_{i+1} \rbrace$ is the event that $z_i$ and $z_{i+1}$ are $\theta$-connected.}.
\end{equation}
We then have,
\begin{eqnarray*}
    \PP_{x,y}[\GS] &=& \PP_{x,y}\left[\bigcap_{i=0}^3H_i\right] \\
    &\geq & c^4 \dfrac{\prod_{i=0}^3|\lbrace \theta\text{-paths from }z_i \text{ to }z_{i+1}\rbrace|}{|\lbrace \theta\text{-paths from }z_0 \text{ to }z_{3}\rbrace|}.
\end{eqnarray*}
Let $a=\cos \theta$ and $b=\sin \theta$. Define
\begin{equation}
    \begin{cases}
    \alpha_0 = \frac{h}{4}a+10\eta h^{\frac{1}{2}}b \\
    \beta_0 = \frac{h}{4}b-10\eta h^{\frac{1}{2}}a
    \end{cases}, \qquad \begin{cases}
    \alpha_1 = \frac{h}{4}a-5\eta h^{\frac{1}{2}}b \\
    \beta_1 = \frac{h}{4}b+5\eta h^{\frac{1}{2}}a
    \end{cases},
\end{equation}
\begin{equation}
    \begin{cases}
    \alpha_2 = \frac{h}{4}a+5\eta h^{\frac{1}{2}}b \\
    \beta_2 = \frac{h}{4}b-5\eta h^{\frac{1}{2}}a
    \end{cases}, \qquad \begin{cases}
    \alpha_3 = \frac{h}{4}a-10\eta h^{\frac{1}{2}}b \\
    \beta_3 = \frac{h}{4}b+10\eta h^{\frac{1}{2}}a
    \end{cases},\qquad  \begin{cases}
    \alpha_4 = ha \\
    \beta_4 = hb
    \end{cases}.
\end{equation}
It is nothing but a little combinatorial fact that the cardinal of the set of $\theta$-paths linking $z_i$ and $z_{i+1}$ is $\binom{\alpha_i + \beta_i}{\alpha_i}$. So that\footnote{We omit integer rounding and the fact that there might not be a $\theta$-path from $z_i$ to $z_{i+1}$ (but rather a $\theta$-path from $z_i+B_1$ to $z_{i+1}+B_1$ where $B_1$ is the unit ball).}, 
\begin{equation}
    \dfrac{\prod_{i=0}^3|\lbrace \theta\text{-paths from }z_i \text{ to }z_{i+1}\rbrace|}{|\lbrace \theta\text{-paths from }z_0 \text{ to }z_{3}\rbrace|}=\dfrac{\prod_{i=0}^{3}\binom{\alpha_i+\beta_i}{\alpha_i}}{\binom{\alpha_4+\beta_4}{\alpha_4}}.
\end{equation}
A quite tedious computation using Stirling's estimate yields the existence of some constant $C=C(\varepsilon, \eta)>0$ such that
\begin{equation}
    \PP_{x,y}[\GS]\geq Ch^{-\frac{3}{2}}.
\end{equation}
We are a factor $h^{-3/2}$ away from the desired result. This factor can be removed by considering variants of the event $\GS$ where the vertical coordinates of $z_1, z_2, z_3$ may differ from the original ones by at most $2\eta h^{1/2}$. Since these variants of $\GS$ are still included in $\GAC{x}{y}{\eta} $ and there being $(2\eta)^3h^{3/2}$ such events, we get the desired conclusion. We obtained ,
\begin{equation}
    \mathbb P_{x,y}\left[\GAC{x}{y}{\eta}\right]\geq C(\eta, \eps).
\end{equation}
\end{proof}
\begin{Rem}
    It is clear that the constant $C(\eta, \eps)$ degenerates when $\eps$ goes to 0, which explains the need for the \textit{a priori} estimate given by Proposition~\ref{prop: bad angle}.
\end{Rem}
\begin{proof}[Proof of Lemmas~\textup{\ref{lemme log good shape uniform}} and~\textup{\ref{Lemme log gac logsid}}]
Note that the statement of Lemma~\ref{Lemme log gac logsid} implies the one of Lemma~\ref{lemme log good shape uniform}. We then focus on the first one. The proof of Lemma~\ref{good shape uniform law 1} can be reproduced with a minor change of parameters. Indeed, define this time $z_0 = (0,0)$, $z_1 = (\tfrac{h}{4}, 10 \eta (h\log h)^{\frac{1}{2}})$, $z_2 = (\tfrac{h}{2}, 5 \eta (h\log h)^{\frac{1}{2}})$, $z_3 = (\tfrac{3h}{4}, 10\eta (h\log h)^{\frac{1}{2}})$ and $z_4 =(h,0)$.

Let $R_{\theta}: \R^2 \rightarrow \R^2$ be the rotation of the plane which maps $[x,y]$ on an horizontal segment and $T$ the translation which maps $x$ to $0$. Let $\Tilde{\mathbf{R}} = T \circ \R_{\theta}(\mathbf{R})$. To ensure that $z_1, z_3 \in \Tilde{\mathbf{R}}$, we set $N_0=N_0(\varepsilon,\eta)$ large enough so that for $h\geq N_0$,
\begin{equation}
    \arctan(40\eta (h\log h)^{-\frac{1}{2}}) \leq \varepsilon.
\end{equation}
Assume that $h\geq N_0$. As previously, let, for $i \in \lbrace {0, \dots, 3} \rbrace$, $H_i$ be the event that $z_i$ and $z_{i+1}$ are $\theta$-connected by a $\theta$-path which fluctuates less than $10 \dist z_{i+1} - z_i \dist^{\frac{1}{2}}$ around the segment $\left[ z_i, z_{i+1} \right]$, and define the event $\mathsf{LogShape} = H_0 \cap H_1 \cap H_2 \cap H_3$. As previously, we notice that 
\begin{equation}
    \mathsf{LogShape} \subset \mathsf{LogGAC}(x,y,\eta).
\end{equation}
Moreover, 
\begin{equation}
    \mathsf{LogShape} \subset \mathsf{LogSID}(x,y,\eta).
\end{equation}
Indeed, it is straightforward to check that $(T \circ
\R_{\theta})^{-1}(z_2)$ accomplishes the born on the local roughness:
\begin{eqnarray*}
\mathrm{d}((T \circ \R_{\theta})^{-1}(z_2), \mathcal{C}(\g)) &\geq& \mathrm{d}(z_2, [ z_1, z_3 ]) \\
&=& 5\eta ( h\log h )^{\frac{1}{2}}+O(1).
\end{eqnarray*}
Finally, the estimate of $\PP_{x,y}\left[\mathsf{LogShape}\right]$ follows the one of $\mathbb P_{x,y}[\mathsf{Shape}]$ conducted in the preceding lemma. Defining $a=\cos \theta, b=\sin \theta$ and 
\begin{equation}
    \begin{cases}
    \alpha_0 = \frac{h}{4}a+10\eta (h\log h)^{\frac{1}{2}}b \\
    \beta_0 = \frac{h}{4}b-10\eta (h\log h)^{\frac{1}{2}}a
    \end{cases}, \qquad \begin{cases}
    \alpha_1 = \frac{h}{4}a-5\eta (h\log h)^{\frac{1}{2}}b \\
    \beta_1 = \frac{h}{4}b+5\eta (h\log h)^{\frac{1}{2}}a
    \end{cases},
\end{equation}
\begin{equation}
    \begin{cases}
    \alpha_2 = \frac{h}{4}a+5\eta (h\log h)^{\frac{1}{2}}b \\
    \beta_2 = \frac{h}{4}b-5\eta (h\log h)^{\frac{1}{2}}a
    \end{cases}, \qquad \begin{cases}
    \alpha_3 = \frac{h}{4}a-10\eta (h\log h)^{\frac{1}{2}}b \\
    \beta_3 = \frac{h}{4}b+10\eta (h\log h)^{\frac{1}{2}}a
    \end{cases},\qquad  \begin{cases}
    \alpha_4 = ha \\
    \beta_4 = hb
    \end{cases}.
\end{equation}
As previously, we can use Stirling's estimate to infer that there exists $C=C(\varepsilon)>0$ such that
\begin{equation}
    \PP_{x,y}[\mathsf{LogShape}]\geq h^{-\frac{3}{2}}h^{-C\eta^2}.
\end{equation}
We recover the result by considering $h^{3/2}$ variants of the event $\mathsf{LogShape}$, as above.
\end{proof}
\section{Multivalued map principle}
If $A$ and $B$ are two finite sets, we say that any function $T: A \rightarrow \mathcal{P}(B)$ is a \textit{multivalued map} and for any $b\in B$, we write $T^{-1}(b) = \lbrace a \in A, b \in T(a)\rbrace$.  

\begin{Lemma}[Probabilistic multivalued map principle]\label{lemme MVMP}
Let $\left(\Omega_1, \PP_1\right), \left(\Omega_2, \PP_2\right)$ be two discrete probability spaces, let $A$ (resp. $B$) be a measurable subset of $\Omega_1$ (resp. $\Omega_2$) and let $T: A \rightarrow \mathcal{P}(B)$ be a multivalued map. We assume that the following quantities are finite:
\begin{equation}
    \varphi(T):= \max_{a \in A}\max_{b \in T(a)} \frac{\PP_1(a)}{ \PP_2(b)} < \infty,
\end{equation}
\begin{equation}
    \psi(T):=  \frac{\max_{b \in B}\left| T^{-1}(b)\right|}{\min_{a \in A}\left| T(a) \right|} < \infty.
\end{equation}
Then,
\begin{equation}
    \PP_1[A] \leq \varphi(T) \psi(T) \PP_2[B].
\end{equation}
\end{Lemma}

\begin{proof} The lemma follows by the following simple computation: 
\begin{eqnarray*}
    \PP_1\left[A\right] &=& \sum_{a \in A}\sum_{b \in T(a)}\frac{\PP_1\left[a\right]}{\PP_2\left[b\right]}\PP_2[b] \left|T(a)\right|^{-1} \\
    &\leq& \varphi(T) \sum_{b \in B} \PP_2[b] \sum_{a \in T^{-1}(b)} \left| T(a) \right|^{-1} \\
    &\leq& \varphi(T)  \psi(T) \PP_2[B].
\end{eqnarray*}
\end{proof}

\end{document}